\documentclass[10pt]{article}

\usepackage{url}
\usepackage{mathtools}
\usepackage{amssymb}
\usepackage{amsthm}
\usepackage{empheq}
\usepackage{latexsym}
\usepackage{enumitem}
\usepackage{eurosym}
\usepackage{dsfont}
\usepackage{appendix}
\usepackage{color} 
\usepackage[unicode]{hyperref}
\usepackage{frcursive}
\usepackage[utf8]{inputenc}
\usepackage[T1]{fontenc}
\usepackage{geometry}
\usepackage{multirow}
\usepackage{todonotes}
\usepackage{lmodern}
\usepackage{anyfontsize}
\usepackage{stmaryrd}
\usepackage{natbib}
\usepackage{cleveref}

\setcitestyle{numbers,open={[},close={]}}

\definecolor{red}{rgb}{0.7,0.15,0.15}
\definecolor{green}{rgb}{0,0.5,0}
\definecolor{blue}{rgb}{0,0,0.7}
\hypersetup{colorlinks, linkcolor={red},citecolor={green}, urlcolor={blue}}
			
\makeatletter \@addtoreset{equation}{section}

\newtheorem{theorem}{Theorem}[section]
\newtheorem{assumption}[theorem]{Assumption}

\newtheorem{example}[theorem]{Example}

\newtheorem{lemma}[theorem]{Lemma}
\newtheorem{proposition}[theorem]{Proposition}

\newtheorem{definition}[theorem]{Definition}
\newtheorem{remark}[theorem]{Remark}

\DeclareUnicodeCharacter{014D}{\=o}
\def \E{\mathbb{E}}
\def \F{\mathbb{F}}
\def \G{\mathbb{G}}
\def \H{\mathbb{H}}

\def \L{\mathbb{L}}
\def \M{\mathbb{M}}
\def \N{\mathbb{N}}

\def \P{\mathbb{P}}
\def \Q{\mathbb{Q}}
\def \R{\mathbb{R}}
\def \S{\mathbb{S}}

\def\Ac{{\cal A}}
\def\Bc{{\cal B}}
\def\Cc{{\cal C}}

\def\Ec{{\cal E}}
\def\Fc{{\cal F}}
\def\Gc{{\cal G}}
\def\Hc{{\cal H}}
\def\Ic{{\cal I}}

\def\Lc{{\cal L}}
\def\Mc{{\cal M}}

\def\Pc{{\cal P}}

\def\Uc{{\cal U}}

\def\Wc{{\cal W}}

\def\Fb{{\bar F}}
\def\Gb{{\overline \G}}
\def\Gcb{\overline \Gc}

\def\Pb{{\overline \P}}

\def\Sb{{\bar S}}

\def\Xb{{\overline X}}

\def\x{\times}
\def\eps{\varepsilon}

\def\Om{\Omega}

\def\Omh{\widehat{\Omega}}
\def\om{\omega}

\def\Omb{\overline{\Om}}
\def\omb{\bar \om}
\def\Fb{\overline{\F}}
\def\Gcb{\overline{\Gc}}
\def\Fcb{\overline{\Fc}}

\def\Ft{\widetilde{\F}}

\def\Xt{\widetilde{X}}

\def\Mh{\widehat{S}}
\def\Nh{\widehat{N}}
\def\Ph{\widehat{\P}}

\def\Lambdah{\widehat{\Lambda}}
\def\Fch{\widehat{\Fc}}
\def\Fh{\widehat{\F}}
\def\Xh{\widehat{X}}
\def\Yh{\widehat{Y}}
\def\Wh{\widehat{W}}
\def\Zh{\widehat{Z}}

\def\0{\mathbf{0}}
\def \Ec{\mathcal{E}}

\def \bb{\mathbf{b}}

\def \xb{\mathbf{x}}

\def \yb{\mathbf{y}}
\def \wb{\mathbf{w}}
\def \mub{\overline{\mu}}

\def \muh{\widehat{\mu}}
\def \zetah{\widehat{\zeta}}

\def \nub{\bar{\nu}}
\def \nuh{\widehat{\nu}}

\def\normeL2#1{\left\|{#1}\right\|_{L^2}}
\def\Fbb{\overline \F}

\def\Pcb{\overline \Pc}

\def \Prod{\displaystyle\prod}

\def \Lim{\displaystyle\lim}
\def \Liminf{\displaystyle\liminf}

\def \Nt{\widetilde N}
\def \Yt{\widetilde Y}

\def \Lambdat{\widetilde \Lambda}

\def \Sb{\overline S}

\def \omt{\tilde \om}
\def \Lambdab{\overline \Lambda}

\title{McKean–Vlasov optimal control: limit theory and equivalence between
different formulations\footnote{The authors would like to thank Daniel Lacker for his insightful comments.}}

\author{
	Mao Fabrice {\sc Djete}\footnote{Universit\'e Paris Dauphine - PSL, CNRS, CEREMADE, France, djete@ceremade.dauphine.fr. This author gratefully acknowledges support from the r\'egion \^Ile--de--France.}
	\and 
	Dylan {\sc Possama\"{i}} \footnote{Columbia University, Industrial Engineering \& Operations Research, 500 W 120th Street, New York, NY, 10027, dp2917@columbia.edu. This author gratefully acknowledges the support of the ANR project PACMAN ANR--16--CE05--0027.}
	\and 
	Xiaolu {\sc Tan}\footnote{Department of Mathematics, The Chinese University of Hong Kong. xiaolu.tan@cuhk.edu.hk}
}

\date{\today}

\begin{document}

\maketitle
 
\begin{abstract}
	We study a McKean--Vlasov optimal control problem with common noise, in order to establish the corresponding limit theory, as well as the equivalence between different formulations, including the strong, weak and relaxed formulation. In contrast to the strong formulation, where the problem is formulated on a fixed probability space equipped with two Brownian filtrations, the weak formulation is obtained by considering a more general probability space with two filtrations satisfying an $(H)$--hypothesis type condition from the theory of enlargement of filtrations. When the common noise is uncontrolled, our relaxed formulation is obtained by considering a suitable controlled martingale problem. As for classical optimal control problems, we prove that the set of all relaxed controls is the closure of the set of all strong controls, when considered as probability measures on the canonical space. Consequently, we obtain the equivalence of the different formulations of the control problem, under additional mild regularity conditions on the reward functions. This is also a crucial technical step to prove the limit theory of the McKean--Vlasov control problem, that is to say proving that it consists in the limit of a large population control problem with common noise.

\end{abstract}

\section{Introduction}

	We aim to study a McKean--Vlasov optimal control problem with common noise in the following form.
	Let $T>0$ be the time horizon, $\alpha$ be a control process. The (non--Markovian) controlled process $X^\alpha$ follows a McKean--Vlasov dynamic
	\begin{equation} 
	\label{eq:EDS_McKVlasov-intro}
		`\mathrm{d} X^{\alpha}_t 
		=
		b \big(t, X^{\alpha}_{t \wedge \cdot}, \Lc\big(X^{\alpha}_{t \wedge \cdot}, \alpha_t \big| B \big), \alpha_t \big) \mathrm{d}t
		+
		\sigma \big(t, X^{\alpha}_{t \wedge \cdot}, \Lc\big(X^{\alpha}_{t \wedge \cdot}, \alpha_t \big| B \big), \alpha_t \big)   \mathrm{d}W_t
		+
		\sigma_0 \big(t, X^{\alpha}_{t \wedge \cdot}, \Lc\big(X^{\alpha}_{t \wedge \cdot}, \alpha_t \big| B \big), \alpha_t \big)   \mathrm{d}B_t,`
	\end{equation}
	where $W$ and $B$ are two independent Brownian motions in some given probability space, 
	and $\Lc\big(X^{\alpha}_{t \wedge \cdot}, \alpha_t \big| B \big)$ denotes the conditional distribution of the pair $(X^{\alpha}_{t \wedge \cdot}, \alpha_t)$ given the common noise $B$. We consider the optimisation problem, written informally for now as
	\begin{equation} \label{eq:Value_McKVlasov-intro}
		`\sup_{\alpha} 
		\E \bigg[ 
		\int_0^T L \big(t, X^{\alpha}_{t \wedge \cdot}, \Lc\big(X^{\alpha}_{t \wedge \cdot}, \alpha_t \big| B \big), \alpha_t \big)  \mathrm{d}t 
		+ 
		g \big( X^{\alpha}_{T\wedge\cdot}, \Lc \big(X^{\alpha}_{T\wedge\cdot} \big| B\big) \big) \bigg].`
	\end{equation}

	The analysis of McKean--Vlasov optimal control problems has, in the recent years, drawn the attention of the applied mathematics community. 
	One of the main reasons is their close proximity mean--field games (MFGs for short), introduced in the pioneering work of \citeauthor*{lasry2006jeux} \cite{lasry2006jeux,lasry2006jeux2,lasry2007mean} and \citeauthor*{huang2003individual} \cite{huang2003individual,huang2006large,huang2007invariance,huang2007large,huang2007nash}, as  way to describe Nash equilibria for a large population of symmetric players, interacting through their empirical distribution. We refer the interested readers to \citeauthor*{carmona2013control} \cite{carmona2013control} for a more thorough discussion about the similarities and differences between these two theories.

	\medskip
Being an extension of the classical optimal control problem, McKean--Vlasov optimal control has been studied from different angles. The first one is the Pontryagin maximum principle, which aims at providing a necessary conditions characterising the optimal control, and uses techniques borrowed from calculus of variations. This approach has been applied successfully by \citeauthor*{buckdahn2011general} \cite{buckdahn2011general} and \citeauthor*{andersson2011maximum} \cite{andersson2011maximum}, in the case where the coefficients functions depend solely on some moments of the state process' distribution. In a more general framework, and using the notion of differentiability developed by \citeauthor*{lions2007theorie} \cite{lions2007theorie}, \citeauthor*{carmona2015forward} \cite{carmona2015forward} provide a general analysis of this approach (see also \citeauthor*{acciaio2018extended} \cite{acciaio2018extended} for an extension).
	A second important way to tackle optimal control problems is to use the so--called dynamic programming principle (DPP for short), which consists in decomposing a global optimisation into a series of local optimisation problems. However, compared to the classical setting, the presence of the (conditional) law of the controlled process in the coefficient functions generates heavy additional difficulty in establishing the DPP, as the problem becomes by essence time--inconsistent (see \citeauthor*{bjork2014theory} \cite{bjork2014theory,bjork2017time} or \citeauthor*{hernandez2020me} \cite{hernandez2020me} for a discussion and additional references on these issues). 
	A first breakthrough in this area was achieved in 5 years ago, when, by assuming the existence of a density with respect to Lebesgue measure for the marginal distribution of the state process, \citeauthor*{lauriere2014dynamic} \cite{lauriere2014dynamic} and \citeauthor*{bensoussan2015master} \cite{bensoussan2015master} reformulated the initial McKean--Vlasov control problem as a deterministic density control problem, associated to a family of deterministic controls, for which they could then straightforwardly establish the DPP. Without the density existence assumption, but under some regularity conditions on the coefficient functions, the DPP has been proved in \citeauthor*{pham2018bellman} \cite{pham2018bellman, pham2016dynamic}, and \citeauthor*{bayraktar2016randomized} \cite{bayraktar2016randomized} in different situations. Using abstract measurable selection arguments, 
	a general DPP has been established under minimal conditions in our accompanying article \citeauthor*{djete2019mckean} \cite{djete2019mckean}.

	\medskip

	In this paper, we are interested in establishing the limit theory for the McKean--Vlasov optimal control problem. In other words, we wish to rigorously prove that such a control problem naturally arises as the limit of a large population optimal control problem.
	In the uncontrolled case, this property is by now extremely well--known, and usually referred to as 'propagation of chaos'. Much effort has been devoted to it since the seminal works of \citeauthor*{kac1956foundations} \cite{kac1956foundations} and \citeauthor*{mckean1969propagation} \cite{mckean1969propagation}, see also the illuminating lecture notes of \citeauthor*{snitzman1991topics} \cite{snitzman1991topics}. Without any claim to comprehensiveness, we refer to \citeauthor*{oelschlager1984martingale} \cite{oelschlager1984martingale}, and \citeauthor*{gartner1988mckean} \cite{gartner1988mckean} for models in the Markovian context without common noise, 
	to \citeauthor*{budhiraja2012large} \cite{budhiraja2012large} for a large deviation principle associated to the limit theory,
	and also to \citeauthor*{meleard1987propagation} \cite{meleard1987propagation}, \citeauthor*{jourdain1998propagation} \cite{jourdain1998propagation}, and \citeauthor*{oelschlager1985law} \cite{oelschlager1985law} for the case of 'strong' and 'moderate' interactions,
	to \citeauthor*{shkolnikov2012large} \cite{shkolnikov2012large}, \citeauthor*{jourdain2013propagation} \cite{jourdain2013propagation} for rank--based models,
	and finally to \citeauthor*{meleard1996asymptotic} \cite{meleard1996asymptotic}, and \citeauthor*{graham1997stochastic} \cite{graham1997stochastic} for Boltzmann--type models.

	\medskip
In the controlled case, \citeauthor*{fischer2016continuous} \cite{fischer2016continuous} studied a mean--variance optimisation problem stemming from mathematical finance, and obtained results in this direction.
	For general McKean--Vlasov controlled equations, such a limit theory has been proved in \citeauthor*{lacker2017limit}  \cite{lacker2017limit} in a context without common noise,
	where an essential tool is a compactness argument, which is made accessible by formulating an appropriate relaxed control for McKean--Vlasov equations, in the spirit of \citeauthor*{el1987compactification} \cite{el1987compactification}, and by introducing suitable martingale problems, similar to those of \citeauthor*{stroock2007multidimensional} \cite{stroock2007multidimensional}.
	The same formulation and arguments have also been used in \citeauthor*{bahlali2017existence} \cite{bahlali2014existence,bahlali2017existence,bahlali2018relaxed,bahlali2019stability} and \citeauthor*{chala2014relaxed} \cite{chala2014relaxed} to study stability and approximation problems.

	\medskip
In the present article, our ultimate goal is to analyse a general McKean--Vlasov control problem with common noise in the form of  \eqref{eq:EDS_McKVlasov-intro}--\eqref{eq:Value_McKVlasov-intro}. Our first main objective is to establish the corresponding limit theory. To this end, we introduce three formulations. The strong one is given as in \eqref{eq:EDS_McKVlasov-intro}--\eqref{eq:Value_McKVlasov-intro}, in a fixed probability space equipped with two Brownian motions, as well as their natural filtrations. By considering more general probability spaces and filtrations, but imposing a technical $(H)$--hypothesis type condition, we obtain a weak formulation of the control problem.
	Our weak formulation is consistent with that of the classical optimal control problems, and enjoys some convexity and stability properties. More importantly, by considering them as probability measures on the canonical space, we show that any weak control rule can be approximated by strong control rules in the sense of weak convergence, which implies the equivalence between the strong and weak formulations. We emphasise that this first result is a crucial technical step in the proof of the DPP in our accompanying paper \cite{djete2019mckean}. We next restrict to the case where the common noise part $\sigma_0$ is not controlled, and the dependence of the coefficient functions $b$, $\sigma$, and $\sigma_0$ in $\Lc(X^{\alpha}_{t \wedge \cdot}, \alpha_t |B)$ is through 
	$\Lc(X^{\alpha}_{t \wedge \cdot} |B)$ only (in words, the conditional law of the control process is not included in the coefficient function), and then introduce a relaxed formulation. We subsequently prove that any relaxed control rule can be approximated by weak control rules, in the sense of weak convergence of probability measures on the canonical space. Besides, the relaxed formulation enjoys an additional closedness property, implying the existence of optimal control rules under mild additional technical conditions. The closedness property and our aforementioned equivalence results between the different formulations are also crucially used to obtain the limit theory.

	\medskip
Our main contribution lies in the fact that we are generalising several fundamental results for McKean--Vlasov control problems to a context with common noise, including the formulation of the weak and relaxed problems, their equivalence, and the corresponding limit theory.
	The presence of the common noise generates some significant technical hurdles, especially due to the appearance of the conditional distribution terms, which are generally not continuous with respect to the joint distribution. In the context of MFG, this difficulty has been tackled by \citeauthor*{carmona2014mean} \cite{carmona2014mean}, and \citeauthor*{lacker2016general} \cite{lacker2016general}. In the context of McKean--Vlasov optimal control problem however, we need to formulate appropriate notions of weak and relaxed control rules, and develop new techniques to ensure the approximation property. Another technical difficulty comes from the presence of the conditional law of the control process $\alpha$ in the coefficient functions (for the strong and weak formulations), a situation which has been rarely studied in the literature (see for instance \citeauthor*{graber2016linear} \cite{graber2016linear}, \citeauthor*{elie2016tale} \cite{elie2016tale}, \citeauthor*{zalashko2017causal} \cite{zalashko2017causal}, \citeauthor*{pham2018bellman} \cite{pham2018bellman}, \citeauthor*{acciaio2018extended} \cite{acciaio2018extended}, and \citeauthor*{basei2019weak} \cite{basei2019weak}). Our equivalence results between the strong and weak formulations is very general, and its proof is quite different from that in the case without common noise. It allows in particular to fill a subtle technical gap in the related literature (see \Cref{rem:Gap_StrongWeak} for more details). A second important point is that our approach also bypasses a second technical issue in the literature considering relaxed formulations for McKean--Vlasov control problems without common noise, namely \cite{lacker2017limit,bahlali2018relaxed}, and which proves equivalence results between several formulations. Indeed, their proofs are based on an incorrect technical result in an unpublished, and actually inaccessible, paper \cite{meleard1992martingale}\footnote{\label{foot:meleard}Through personal communications with {\color{green}S. M\'el\'eard}, it was confirmed to us that she and her co--authors discovered a mistake soon after finishing the paper, and hence abandoned it. Nevertheless, although the original manuscript is now nowhere accessible, some of its results have been announced in the conference proceedings \cite{meleard1992representation}. More specifically, the problematic result is \cite[Corollary on pages 196--197]{meleard1992representation}, which has been crucially used in \cite[Proposition 2.2.]{bahlali2018relaxed}, and \cite[Lemma 7.1.]{lacker2017limit}.}, see \Cref{rem:Gap_Weak_Relax} for more details. We instead adapt the approximation arguments in \cite{el1987compactification} to remedy this technical gap. 

	\medskip
The rest of the paper is structured as follows. After introducing some notations, we provide in \Cref{sec:different-Formulation} the notions of strong, weak and relaxed formulations for the McKean--Vlasov stochastic control problem in a common noise and non--Markovian setting, and define also an $N$--particles (strong) control problem. The main results of the paper are presented in \Cref{sec:MainResults}, including the existence of optimal control, the equivalence between the strong, weak and relaxed formulations and the limit theory. Most of the technical proofs are completed in \Cref{sec:proofs}.

\paragraph*{Notations}\label{par:notations}
	$(i)$
	Given a metric space $(E,\rho)$, let $\Bc(E)$ denote the Borel $\sigma$--algebra,
	and $\Pc(E)$ be the collection of all Borel probability measures on $E$.
	For $p \ge 0$, let $\Pc_p(E)$ denote the set of $\mu \in \Pc(E)$ such that $\int_E \rho(e, e_0)^p  \mu(\mathrm{d}e) < \infty$ for some (and thus for all) $e_0 \in E$.
	When $p \ge 1$, the space $\Pc_p(E)$ is equipped with the Wasserstein distance $\Wc_p$, defined by
	\[
		\Wc_p(\mu , \mu^\prime) 
		:=
		\Big(
			\inf_{\lambda \in \Pi(\mu, \mu^\prime)}  \int_E \int_{E} \rho(e, e^\prime)^p \lambda( \mathrm{d}e, \mathrm{d}e^\prime) 
		\Big)^{1/p},\; (\mu,\mu^\prime)\in\Pc_p(E)\times \Pc_p(E),
	\]
	where $\Pi(\mu, \mu^\prime)$ denotes the set of all probability measures $\lambda$ on $E \x E$ 
	such that $\lambda( \mathrm{d}e, E) = \mu$ and $\lambda(E,  \mathrm{d}e^\prime) = \mu^\prime( \mathrm{d}e^\prime)$.
	Let $\mu \in \Pc(E)$ and $\varphi: E \longrightarrow \R$ be a $\mu$--integrable function, we write
	\[
		\langle \varphi, \mu \rangle
		:=
		\langle \mu, \varphi \rangle
		:=
		\E^{\mu}[\varphi]
		:=
		\int_E \varphi(e) \mu(\mathrm{d}e).
	\]
	Let $(E^\prime, \rho^\prime)$ be another metric space and $\mu^\prime \in \Pc(E^\prime)$. We denote by $\mu \otimes \mu^\prime \in \Pc(E \x E^{\prime})$ their product probability measure.
	Given a probability space  $(\Om, \Fc, \P)$ equipped with a sub--$\sigma$--algebra $\Gc \subset \Fc$,
	we denote by $(\P^{\Gc}_{\om})_{\om \in \Om}$ the conditional probability measure on $\P$ knowing $\Gc$ (whenever it exists).
	For a random variable $\xi: \Om \longrightarrow E$, we write
	$\Lc^{\P}(\xi):=\P \circ \xi^{-1}$ the law of $\xi$ under $\P$, and for any $\omega\in\Omega$, $\Lc^{\P}( \xi | \Gc)(\om):=\P^{\Gc}_{\om} \circ \xi^{-1} $ the conditional distribution of $\xi$ knowing $\Gc$ under $\P$.


	\medskip
	\noindent $(ii)$	
	We let $\N^\star$ be the set of positive integers and $\R_+:=[0,+\infty)$. Given non--negative integers $m$ and $n$, we denote by $\S^{m \x n}$ the collection of all $m \x n$--dimensional matrices with real entries, equipped with the standard Euclidean norm, which we denote by $|\cdot|$ regardless of the dimensions, for notational simplicity. 
	We also denote $\S^n:=\S^{n \times n}$, and denote by $0_{m \times n}$ the element in $\S^{m \times n}$ whose entries are all $0$, and by $\mathrm{I}_n$ the identity matrix in $\S^n$. 
	Let $k$ be a non--negative integer, we denote by $C^k_b(\R^n;\R)$ the set of bounded maps $f: \R^n \longrightarrow \R$, having bounded continuous derivatives of order up to and including $k$. 
	Let $f: \R^n \longrightarrow \R$ be twice differentiable, we denote by $\nabla f$ and $\nabla^2f$ 
	the gradient and Hessian of $f$.

	\medskip	
	\noindent $(iii)$ 
	Let $(E, \rho)$ be a Polish space, and $T>0$ a time horizon. We denote by $C([0,T], E)$ the space of all continuous paths from $[0,T]$ to $E$,
	which is a Polish space under the uniform convergence topology.
	When $E = \R^n$, we write $\Cc^n := C([0,T], \R^n)$.
	For every $\xb \in \Cc^n$, we denote by $\| \xb \|_n :=\sup_{t \in [0,T]}|\xb_t|$ the uniform norm on $\Cc^n$, 
	which may also be simplified to $\|\xb \|$ when there is no ambiguity. 
	When $n=0$, the space $\R^n$, $\S^{m\x n}$ and $\Cc^n$ degenerate to be a singleton.

	\medskip
We also denote by $\M(E)$ the space of all Borel measures $q( \mathrm{d}t,  \mathrm{d}e)$ on $[0,T] \x E$, 
	whose marginal distribution on $[0,T]$ is the Lebesgue measure $ \mathrm{d}t$, 
	that is to say $q( \mathrm{d}t, \mathrm{d}e)=q(t,  \mathrm{d}e) \mathrm{d}t$ for a family $(q(t,  \mathrm{d}e))_{t \in [0,T]}$ of Borel probability measures on $E$.
	Let $\Lambda$ denote the canonical element on $\M(E)$, we define
	\begin{equation}\label{eq:lambda}
		\Lambda^t(\mathrm{d}s, \mathrm{d}e) :=  \Lambda(\mathrm{d}s, \mathrm{d}e) \big|_{ [0,t] \x E} + \delta_{e_0}(\mathrm{d}e) \mathrm{d}s \big|_{(t,T] \x E},\; \text{for some fixed $e_0 \in E$.}
	\end{equation}

	Throughout the paper, we fix a nonempty Polish space $(A, \rho)$ and an element $a_0 \in A$,
	and denote $\M := \M(A)$.
	Finally, consider the canonical space  $\Cc^n \x \M$ (resp. $\Cc^n \x A$), 
	with canonical element $(X, \Lambda)$ (resp. $(X, \alpha)$), 
	and $\nuh \in \Pc(\Cc^n \x \M)$ (resp. $\nub \in \Pc(\Cc^n \x A)$). We define, for each $t \in [0,T]$
	\begin{equation} \label{eq:muht}
		\nuh(t) := \nuh \circ (X_{t \wedge \cdot}, \Lambda^t)^{-1}, \; \mbox{$\big($resp.}\;
		\nub(t) := \nub \circ (X_{t \wedge \cdot}, \alpha)^{-1} \big).
	\end{equation}

\section{McKean--Vlasov optimal control: different formulations}
\label{sec:different-Formulation}
	
	We introduce here a strong, a weak and a relaxed formulation of the McKean--Vlasov optimal control problem,
	which can all be (re--)formulated on the same canonical space by considering appropriate martingale problems. 
	We next define a large population control problem, the limit of which is expected (and will be proved) to be the McKean--Vlasov control problem.
	
\medskip

	These formulations share a certain number of functions which we now introduce. Let $n \ge 1$, $d\ge 1$ be two positive integers, and $\ell \ge 0$ a non--negative integer, which are fixed throughout the paper. The controlled diffusion process \eqref{eq:EDS_McKVlasov-intro} has the following coefficient functions
	\[
		(b, \sigma, \sigma_0):[0,T] \x \Cc^n  \x \Pc(\Cc^n \x A) \x A 
		\longrightarrow
		\R^n \x \S^{n \x d} \x \S^{n \x \ell},
	\]
	and the reward value \eqref{eq:Value_McKVlasov-intro} is defined with the coefficient functions
	\[
		L:[0,T] \x \Cc^n \x \Pc(\Cc^n \x A) \x A \longrightarrow \R,
		\; \mbox{and}\; 
		g: \Cc^n \x \Pc(\Cc^n) \longrightarrow \R.
	\]
	Throughout the paper, we assume the following regularity and growth conditions on the coefficient functions.
	 
	\begin{assumption} \label{assum:main1}   
		The maps $(b, \sigma, \sigma_0, L, g)$ are Borel measurable
		and non--anticipative, in the sense that
		\[
			\big(b, \sigma, \sigma_0,L \big) (t,\xb, \nub, a) 
			=
			\big(b, \sigma, \sigma_0,L \big) (t,\xb(t \wedge \cdot), \nub(t), a),
			~\mbox{for all}~
			(t, \xb, \nub, a) \in [0,T] \x \Cc^n \x \Pc(\Cc^n \x A)\x A .
		\]
		Moreover, there exist positive constants $C$, $p$, $p^\prime$ and $\hat p$, such that $p^\prime > p \ge 2 \ge \hat p \ge 0$, and
		\begin{enumerate}
			\item[$(i)$] the function $(b,\sigma,\sigma_0)$ is continuous in $(\xb, \nub, a)$ and uniformly Lipschitz in $(\xb, \nub)$, i.e.
			for all $(t,\xb,\nub,a,\xb^\prime,\nub^\prime)\in[0,T]\times\Cc^n\times\Pc(\Cc^n\x A)\x A\x \Cc^n\times\Pc(\Cc^n\x A)$
		\begin{align*}
			\big| (b,\sigma,\sigma_0)(t,\xb,\nub, a )
			-
			(b,\sigma,\sigma_0)(t,\xb^\prime, \nub^\prime, a)
			\big |
			\le
			C \big( \| \xb-\xb^\prime \| + \Wc_p(\nub,\nub^\prime) \big);
		\end{align*}		
		 
		\item[$(ii)$] for all $(t,\xb, \nub, a) \in [0,T] \x \Cc^n  \x \Pc(\Cc^n \x A)\x A$
		\begin{align*}
			| b(t,\xb, \nub, a)|
			&\le 
			C \bigg ( 1+ \| \xb \| + \bigg( \int_{\Cc^n \x A}\big( \| \xb^\prime\|^p + \rho(a_0,a^\prime)^p\big) \nub(\mathrm{d}\xb^\prime,\mathrm{d}a^\prime) \bigg)^{\frac{1}{p}} + \rho(a_0,a) \bigg ),\\
			| (\sigma,\sigma_0)(t, \xb, \nub, a)|^2
			&\le 
			C \bigg ( 1+ \| \xb \|^{\hat p} + \bigg( \int_{\Cc^n\x A} \big(\| \xb^\prime\|^p + \rho(a_0,a^\prime)^p\big) \nub(\mathrm{d}\xb^\prime,\mathrm{d}a^\prime) \bigg)^{\frac{\hat p}{p}} + \rho(a_0,a)^{\hat p} \bigg );
		\end{align*}
    
		\item[$(iii)$] the function $g$ is lower semi--continuous, for every $t \in [0,T]$, the function $L$ is lower semi--continuous in $(\xb, \nub, a)$,
			and for an additional constant $C_L > 0$, we have for all $(t,\xb,\nub,\nu,a) \in [0,T] \x \Cc^n \x \Pc(\Cc^n \x A)\x \Pc(\Cc^n) \x A$
		\begin{align} \label{eq:cond_coercive}
			|g(\xb,\nu)|
			&\le  
			C \bigg ( 1+ \| \xb \|^{p} +  \int_{\Cc^n} \| \xb^\prime\|^{p} \nu(\mathrm{d}\xb^\prime)   \bigg ), \nonumber \\
			L(t, \xb, \nub, a) 
			&\le 
			C \bigg( 1+ \| \xb \|^p +  \int_{\Cc^n\x A} \big(\| \xb^\prime\|^p + \rho(a_0,a^\prime)^p\big) \nub(\mathrm{d}\xb^\prime,\mathrm{d}a^\prime)   \bigg ) - C_L\rho(a_0,a)^{p^\prime},\\
			L(t,\xb, \nub, a)
			&\geq 
			-C \bigg ( 1+ \| \xb \|^{p} +  \int_{\Cc^n\x A}\big( \| \xb^\prime\|^{p} + \rho(a_0,a^\prime)^p\big) \nub(\mathrm{d}\xb^{\prime},\mathrm{d}a^\prime) \bigg). \nonumber
		\end{align}
		\end{enumerate}

	\end{assumption}

 	\begin{remark}
		Most of the integrability conditions in Assumption {\rm\ref{assum:main1}} are consistent with $($or simply adapted from$)$ those in {\rm\citeauthor*{lacker2017limit} \cite[Assumption $A$]{lacker2017limit}}.
		Basically, they are here to ensure that the controlled processes remain sufficiently integrable to apply the weak convergence techniques.
		In particular, $(i)$ and $(ii)$ are used to ensure the well--posedness of the controlled {\rm SDE} \eqref{eq:EDS_McKVlasov-intro}, while the coercivity condition in Item $(iii)$ is used to ensure the $($pre--$)$compactness of the set of optimal relaxed control rules.
	\end{remark}

\subsection{A strong formulation}\label{sec:strong}
\label{subsec:strong_form}

	To give a strong formulation of the McKean--Vlasov optimal control problem,
	we first introduce a fixed probability space equipped with an initial random variable $X_0$, and two independent Brownian motions $W$ and $B$.
	Precisely, let us consider the canonical space
	\[ 
		\Om
		:=
		\R^n \x \Cc^{d} \x \Cc^{\ell},
	\]
	equipped with its Borel $\sigma$--algebra $\Fc := \Bc(\Omega)$ and canonical element $(X_0, W, B)$.
	Let $\F := (\Fc_t)_{0 \le t \le T}$ and $\G = (\Gc_t)_{0 \le t \le T}$ be two filtrations on $(\Om, \Fc)$ defined by
	\[
		\Fc_t
		:=
		\sigma 
		\big((X_0, W_s, B_s):s \in [0,t] \big),
		\; \mbox{and}\; 
		\Gc_t 
		:= 
		\sigma \big( B_s:s \in [0,t] \big),\; t\in[0,T].
	\]   
	Let $p \ge 2$ be the constant in \Cref{assum:main1} and $\nu \in \Pc_p(\R^n)$. We denote by $\P_{\nu}$ the probability measure on $(\Om,\Fc)$, 
	under which $X_0 \sim \nu$ and $(W,B)$ is a standard $\R^{d+\ell}$--dimensional Brownian motion, independent of $X_0$.
	Recall that $a_0$ is a fixed point in $A$. We denote by $\Ac_p(\nu)$ the collection of all $\F$--predictable, $A$--valued processes $\alpha = (\alpha_s)_{0 \le s \le T}$
	satisfying
	\begin{equation} \label{eq:alpha_Hp}
		\E^{\P_{\nu}} \bigg[ \int_0^T \big( \rho( \alpha_s, a_0) \big)^p \mathrm{d}s \bigg] 
		< \infty.
	\end{equation}
	Then given a control process $\alpha \in \Ac_p(\nu)$, the controlled McKean--Vlasov SDE
	\begin{equation} \label{eq:MKV_SDE}
		X^{\alpha}_t
		= 
		X_0
		+
		\int_0^t b \big(s, X^{\alpha}_{s \wedge \cdot}, \mub^{\alpha}_s, \alpha_s \big) \mathrm{d}s
		+
		\int_0^t \sigma\big(s, X^{\alpha}_{s \wedge \cdot}, \mub^{\alpha}_s, \alpha_s \big) \mathrm{d} W_s
		+
		\int_0^t \sigma_0\big(s, X^{\alpha}_{s \wedge \cdot}, \mub^{\alpha}_s, \alpha_s \big) \mathrm{d} B_s,\; t\in[0,T],\; 
		\P_{\nu}\mbox{--a.s.},
	\end{equation}
	with $\mub^{\alpha}_s:=\Lc^{\P_{\nu}}\big(X^{\alpha}_{s \wedge \cdot}, \alpha_s \big| \Gc_s \big),$ $\mathrm{d}t\otimes\mathrm{d}\P_\nu$--a.e., 
	has a unique strong solution, that is, there is a unique $\F$--adapted continuous process $X^{\alpha}$ on $(\Om,\Fc)$ satisfying \Cref{eq:MKV_SDE} and $\E^{\P_{\nu}}\big[\sup_{t \in [0,T]} |X^\alpha_t|^p \big]< \infty$
	(see for instance \cite[Theorem A.3]{djete2019mckean}).
	
	\medskip
	Denote also $\mu^{\alpha}_t :=  \Lc^{\P_{\nu}}\big(X^{\alpha}_{t \wedge \cdot} \big| \Gc_t \big)$ for all $t \in [0,T]$. The strong formulation of the McKean--Vlasov control problem is then given by
	\begin{align}\label{eq:def_VS}
		V_S(\nu)
		:=
		\sup_{\alpha \in \Ac_p(\nu)} 
		\E^{\P_{\nu}} \bigg[
		\int_0^T L \big(t, X^{\alpha}_{t \wedge \cdot}, \mub^{\alpha}_t, \alpha_t \big) \mathrm{d}t 
		+ 
		g \big(X^{\alpha}_{T \wedge \cdot}, \mu^{\alpha}_T \big) 
		\bigg].
	\end{align}

\subsection{A weak formulation}

	As in the classical SDE theory, one can consider all possible probability spaces to define a weak solution of the controlled SDE \eqref{eq:EDS_McKVlasov-intro}.

	\begin{definition} [Weak control] \label{def:weak_ctrl}
		Let $\nu \in \Pc_p(\R^n)$, we say that a term
		\[
			\gamma := \big( \Om^{\gamma}, \Fc^{\gamma},  \P^{\gamma},  \F^{\gamma} := (\Fc^{\gamma}_t)_{0 \le t \le T},\G^{\gamma}:= (\Gc^{\gamma}_t)_{0 \le t \le T}, X^{\gamma}, W^{\gamma}, B^{\gamma},  \mub^{\gamma}, \mu^{\gamma}, \alpha^{\gamma} \big),
		\]
		is a weak control associated with the initial $($distribution$)$ condition $\nu$ if
		\begin{enumerate}
			\item[$(i)$] $(\Om^{\gamma}, \Fc^{\gamma},\P^{\gamma})$ is a probability space, 
			equipped with two filtrations $ \F^{\gamma}$ and $\G^{\gamma}$ such that,  for all $t\in [0,T]$
			\begin{equation} \label{eq:H_Hypothesis}     
				\Gc_t^{\gamma} \subseteq  \Fc_t^{\gamma},
				\; \mbox{\rm and}\; 
				\E^{\P^\gamma} \big[1_D \big| \Gc^{\gamma}_t \big]
				=
				\E^{\P^\gamma} \big[1_D \big| \Gc^{\gamma}_T \big],
				~\P^{\gamma}\mbox{\rm--a.s.,}
				\;\mbox{\rm for all}\;D \in \Fc^{\gamma}_t \lor \sigma(W^\gamma);
			\end{equation}	
			
			\item[$(ii)$]  $X^{\gamma}:=(X^{\gamma}_s)_{s \in [0,T]}$ is an $\R^n$--valued $\F^{\gamma}$--adapted continuous process and $\alpha^{\gamma} := (\alpha^{\gamma}_s)_{0 \le s \le T}$ is an $A$--valued $\F^{\gamma}$--predictable process such that
				$\E^{\P^{\gamma}} \big[ \|X^\gamma\|^p+  \int_0^T \big( \rho( \alpha^{\gamma}_s, a_0 ) \big)^p \mathrm{d}s \big] <\infty;$
		
			\item[$(iii)$]		
			$(W^{\gamma}, B^{\gamma})$ is an $\R^d \x \R^{\ell}$--valued standard Brownian motion with respect to $\F^\gamma$,
			$B^{\gamma}$ is in addition adapted to $\G^{\gamma}$, $\Fc^{\gamma}_0 \vee \sigma(W^{\gamma})$ is independent of $\Gc^{\gamma}_T$, and
			$\mu^{\gamma}$ $($resp. $\mub^{\gamma})$ is a
			$\Pc(\Cc^n)$--valued $($resp. $\Pc(\Cc^n \x A)$--valued$)$ $\G^{\gamma}$--predictable process such that
			\[
				\mu^{\gamma}_t 
				= 
				\Lc^{\P^{\gamma}} \big(X^{\gamma}_{t \wedge \cdot}\big| \Gc^{\gamma}_t \big),
				\; \mbox{\rm and}\;
				\mub^{\gamma}_t
				=
				\Lc^{\P^{\gamma}} \big( \big(X^{\gamma}_{t \wedge \cdot}, \alpha^{\gamma}_t \big) \big| \Gc^{\gamma}_t \big),\; \mathrm{d}\P^\gamma \otimes\mathrm{d}t \mbox{\rm--a.s.;}	
			\]
			\item[$(iv)$] $X^{\gamma}$ satisfies $\P^{\gamma} \circ (X^{\gamma}_0)^{-1} = \nu$ and
			\[
				X^{\gamma}_t
				=
				X^{\gamma}_0 + \int_0^t b(s, X^{\gamma}_{s\wedge \cdot}, \mub^{\gamma}_s, \alpha^{\gamma}_s) \mathrm{d}s 
				+ \int_0^t \sigma(s, X^{\gamma}_{s \wedge \cdot}, \mub^{\gamma}_s, \alpha^{\gamma}_s) \mathrm{d}W^{\gamma}_s 
				+ \int_0^t \sigma_0(s, X^{\gamma}_{s \wedge \cdot}, \mub^{\gamma}_s, \alpha^{\gamma}_s) \mathrm{d}B^{\gamma}_s, \; t\in[0,T],\; 
				\P^\gamma\mbox{\rm--a.s.}
			\]
		\end{enumerate}
	\end{definition}

	\begin{remark}
		In {\rm \Cref{def:weak_ctrl}}, $\G^{\gamma}$ plays the role of the common noise filtration, to which $B^{\gamma}$ is adapted and of which  $(X_0, W^{\gamma})$ is independent.
		In the literature on enlargement of filtrations $($see {\rm \citeauthor*{jacod1985grossissement} \cite{jacod1985grossissement}} for instance$)$, the $(H)$--hypothesis states that for all $t\in[0,T]$
		\[\E^{\P^\gamma} \big[1_D \big| \Gc^{\gamma}_t \big] = \E^{\P^\gamma} \big[1_D \big| \Gc^{\gamma}_T \big],\; \text{\rm for all}\; D \in \Fc^{\gamma}_t.
		\]
		It is generally different from {\rm Condition \eqref{eq:H_Hypothesis}}, since the independence of the increment $(W^{\gamma}_s - W^{\gamma}_t)_{s \in [t,T]}$ from $\Fc_t^{\gamma}$ and $\Gc^{\gamma}_T$ does not imply the independence between $(W^{\gamma}_s - W^{\gamma}_t)_{s \in [t,T]}$ and $\Fc_t^{\gamma} \vee \Gc^{\gamma}_T$.
		In particular, {\rm Condition \eqref{eq:H_Hypothesis}} will be reformulated later on as \eqref{eq:muh_property} and \eqref{H-property}, which are in turn crucially used in the approximation of a weak control by strong control rules in {\rm\Cref{lemma:strongApp_step1}} and {\rm\Cref{lemma:strongApp_step2}}.
	\end{remark}

	Let us denote by $\Gamma_W(\nu)$ the collection of all weak controls associated
        with the initial condition $\nu$, and introduce the weak formulation of the control problem by
	\begin{equation} \label{eq:def_VW}
		V_W(\nu) 
		:=
		\sup_{\gamma \in \Gamma_W (\nu)} J(\gamma),
		\;\mbox{with}\;
		J( \gamma)
		:=
		\E^{\P^{\gamma}} \bigg[
			\int_0^T L(s, X^{\gamma}_{s \wedge \cdot}, \mub^{\gamma}_s, \alpha^{\gamma}_s) \mathrm{d}s 
			+ 
			g(X^{\gamma}_{T\wedge\cdot}, \mu^{\gamma}_T) 
		\bigg].
	\end{equation}

\subsection{Strong, weak and relaxed formulations on the canonical space}

	The above strong and weak control problem can be reformulated on a canonical space, by considering an appropriate martingale problem.
	Based on this canonical space formulation, we also introduce a notion of relaxed controls for the McKean--Vlasov control problem.

\subsubsection{The canonical space and admissible control rules}
\label{subsubsec:canonical_space}

	Recall that $A$ is a fixed nonempty Polish space,
	$\M:=\M(A)$ denotes the space of all positive Borel measures $q$ on $ [0,T] \x A$ such that the marginal distribution of $q$ on $[0,T]$ is the Lebesgue measure, implying that we can always write $q(\mathrm{d}t, \mathrm{d}a) = q_t(\mathrm{d}a) \mathrm{d}t$, where $(q_t(\mathrm{d}a))_{t \in [0,T]}$ is a Borel measurable kernel from $[0,T]$ to $\Pc(A)$.
	We also introduce a subset $\M_0 \subset \M$, which is the collection of all $q\in\M$ such that $q(\mathrm{d}t, \mathrm{d}a) = \delta_{\psi(t)}(\mathrm{d}a) \mathrm{d}t$ for some Borel measurable function $\psi:[0,T] \longrightarrow A$.
	We will consider two canonical spaces
	\[
		\Omh := \Cc^n \x \Cc^n \x \M \x \Cc^d,
		~\mbox{and}~
		\Omb := \Cc^n \x \Cc^n \x \M \x \Cc^d \x \Cc^\ell \x \Pc \big(\Omh \big).
	\]
	The canonical space $\Omh$ is equipped with the corresponding canonical element $\big(\Xh, \Yh,\Lambdah, \Wh\big)$, its Borel $\sigma$--algebra $\widehat{\Fc}:=\Bc(\Omh)$, and its canonical filtration $\Fh:= \big(\widehat{\Fc}_t \big)_{t \in [0,T]}$ defined by
	\[
	    \Fch_t
	    :=
	    \sigma \Big(\big(\Xh_s, \Yh_s, \Lambdah([0,s] \x D), \Wh_s\big): D \in \Bc(A),\; s \in[0, t] \Big),\; t\in[0,T].
	\]
	Notice that one can choose a version of the disintegration $\Lambdah(\mathrm{d}t,\mathrm{d}a)=\Lambdah_t(\mathrm{d}a)\mathrm{d}t$
	such that $(\Lambdah_t)_{t \in [0,T]}$ is a $\Pc(A)$--valued, $\Fh$--predictable process (see e.g. \cite[Lemma 3.2.]{lacker2015mean}).  
	
	\medskip
	Similarly, we equip the canonical space $\Omb$ with the canonical element  $(X, Y, \Lambda, W, B, \muh)$, and its Borel $\sigma$--algebra  $\Fcb := \Bc(\Omb)$.
	Moreover, based on $\muh$, let us define three processes $(\mu_t)_{t \in [0,T]}$, $(\mub_t)_{t \in [0,T]}$ and $(\muh_t)_{t \in [0,T]}$ on $\Omb$ by $\big($recall \eqref{eq:lambda} for the definition of $\Lambdah^t\big)$
	\begin{equation} \label{eq:muh2mu}
		\mu_t:= \muh \circ \big( \Xh_{t \wedge \cdot}\big)^{-1},
		\;
		\mub_t(\mathrm{d} \xb, \mathrm{d}a):=\E^{\hat \mu}\Big[\delta_{\Xh_{t \wedge \cdot}}(\mathrm{d} \xb)\Lambdah_t(\mathrm{d}a)\Big],
		\;\mbox{and}\;
		\muh_t:=\muh \circ \big( \Xh_{t \wedge \cdot},\Yh_{t \wedge \cdot},\Lambdah^t,\Wh\big)^{-1},\; t\in[0,T].
	\end{equation}
	We then introduce two filtrations $\overline\F:=(\overline\Fc_t)_{t \in [0,T]}$ and $\Gb:=(\Gcb_t)_{t \in [0,T]}$ on $(\Omb,\Fcb)$ by
	\[
		\overline \Fc_t
		:=
		\sigma \Big((X_s, Y_s,\Lambda([0,s] \x D),W_s,B_s, \langle \muh_s, \phi \rangle): D \in \Bc(A), \phi \in C_b(\Cc^n \x \Cc^n \x \M \x \Cc^d),\; s \in[0,t] \Big).
	\] 
	and
	\[
		\Gcb_t
		:=
		\sigma \Big((B_s, \langle \muh_s, \phi \rangle): \phi \in C_b(\Cc^n \x \Cc^n \x \M \x \Cc^d), s\in[0,t]  \Big).
	\]

	To interpret the strong or weak controls as probability measures on the canonical space $\Omb$,
	we will consider a controlled martingale problem.
	Let us define the maps $\bar b:[0,T]\times\Cc^n\x A\x\Pc(\Cc^n\times A)\longrightarrow \R^{n+n+d+\ell}$, and $\bar a:[0,T]\times\Cc^n\x\Pc(\Cc^n\times A)\x A\longrightarrow \S^{n+n+d+\ell}$, such that for any $(t,\xb,\yb, \wb,\bb, \nub,a)\in[0,T]\times\Cc^n \x \Cc^n \x \Cc^d \x \Cc^{\ell}\x\Pc(\Cc^n\times A)\x A$
	\begin{align*} 
		\bar b \big(t,\xb, \wb,\bb, \nub,a \big)
		:=\begin{pmatrix}
			b(t,\xb, \nub, a)\\ 
			b(t,\xb, \nub, a)\\
			0_d\\
			0_\ell
		\end{pmatrix},\;
		\bar a \big(t,\xb,\wb, \bb, \nub,a \big)
		:=
		\begin{pmatrix} 
			\sigma(t,\xb, \nub, a) & \sigma_0(t,\xb, \nub, a) \\
			\sigma(t,\xb, \nub, a) &  0_{n \x \ell} \\
			\mathrm{I}_{d \x d} & 0_{d \x \ell} \\ 
			0_{\ell \x d} & \mathrm{I}_{\ell \x \ell} 
		\end{pmatrix}
		\begin{pmatrix} 
			\sigma(t,\xb, \nub, a) & \sigma_0(t,\xb, \nub, a) \\
			\sigma(t,\xb, \nub, a) &  0_{n \x \ell} \\
			\mathrm{I}_{d \x d} & 0_{d \x \ell} \\ 
			0_{\ell \x d} & \mathrm{I}_{\ell \x \ell}
		\end{pmatrix}^{\top}.
	\end{align*}
	Next, for all $t\in[0,T]$ and $\varphi\in C^2_b(\R^{n+n+d +\ell})$, we define the generator $\overline \Lc_t$ by
	\begin{align} \label{eq:first_generator}
		\overline \Lc_t \varphi \big( \xb,\yb,\wb,\bb, \nub, a \big)
		:=
		\bar b(t, \xb, \nub, a)\cdot \nabla \varphi(\xb(t),\yb(t),\wb(t),\bb(t))
		+ 
		\frac{1}{2}\mathrm{Tr}\big[ \bar a(t, \xb, \nub, a) \nabla^2 \varphi(\xb(t),\yb(t),\wb(t),\bb(t))\big].
	\end{align}
	This allows to define, for any $\varphi\in C^2_b(\R^{n+n+d +\ell})$, $\Sb^{\varphi} := (\Sb^{\varphi}_t)_{t \in [0,T]}$ on $\Omb$ by
	\begin{align} \label{eq:associate-martingale}
		\Sb^{\varphi}_t
		:=
		\varphi(X_t, Y_t, W_t, B_t) 
		-
		\iint_{[0,t]\x A} \overline \Lc_s \varphi \big(X_s, Y_s, W_s, B_s, \mub_s, a \big) \Lambda_s(\mathrm{d}a) \mathrm{d}s,\; t\in[0,T],
	\end{align}
	where for a borel function $\phi:[0,T] \to \R,$ $\int_0^{\cdot} \phi(s)\mathrm{d}s:=\int_0^\cdot \phi^{+}(s)\mathrm{d}s-\int_0^\cdot \phi^{-}(s)\mathrm{d}s$ with the convention $\infty - \infty = -\infty.$

	\begin{definition} \label{def:admissible_ctrl_rule}
		Let $\nu \in \Pc_p(\R^n)$. A probability $\Pb$ on $(\Omb, \Fcb)$ is an admissible control rule with initial condition $\nu$ if 
		\begin{itemize}			
			\item[$(i)$] $\Pb \big[X_0=Y_0,\; W_0=0,\; B_0=0 \big]=1$, $\Pb \circ X_0^{-1} = \nu$, and $(X,\Lambda)$ satisfy $\E^{\Pb} \big[ \|X\|^p+\iint_{[0,T]\x A} \big( \rho(a_0, a) \big)^p \Lambda_t(\mathrm{d}a) \mathrm{d}t \big] < \infty;$

			\item[$(ii)$]  the pair $(X_0,W)$ is independent of $\Gcb_T$ under $\Pb$, and for all $t \in [0,T]$
			\begin{equation} \label{eq:muh_property}
				\muh_t (\omb)
				=
				\Pb^{\Gcb_T}_{\omb} \circ (X_{t \wedge \cdot},Y_{t \wedge \cdot}, \Lambda^t,W)^{-1},
				\;\mbox{\rm for}\;\Pb\mbox{--a.e.}\;\omb \in \Omb;
			\end{equation}

			\item[$(iii)$] the process $\big( \Sb^{\varphi}_t \big)_{t \in [0,T]}$ is an  $(\Fb, \Pb)$--martingale for all $\varphi \in C^2_b \big(\R^n \x \R^n \x \R^d \x \R^{\ell} \big)$.
		\end{itemize}
	\end{definition}
	Let us then define for any $\nu \in \Pc_p(\R^n)$,
	\[
		\Pcb_A( \nu)
		:=
		\big\{ \mbox{All admissible control rules}\; \Pb \;\mbox{with initial condition}\; \nu \big\}.
	\]

	\begin{remark}
    		$(i)$ Under Assumption \ref{assum:main1} and the integrability condition in Definition \ref{def:admissible_ctrl_rule}.(i),
		the process $\Sb^{\varphi}$ is $\Pb$-square integrable for $\varphi \in C^2_b \big(\R^n \x \R^n \x \R^d \x \R^{\ell} \big)$.
		Then it does not change the definition of the admissible control rule if one change Definition \ref{def:admissible_ctrl_rule}.(iii) to 
		`$\big( \Sb^{\varphi}_t \big)_{t \in [0,T]}$ is an  $(\Fb, \Pb)$--local martingale for all $\varphi \in C^2_b \big(\R^n \x \R^n \x \R^d \x \R^{\ell} \big)$.`
		
		\vspace{0.5em}
		
		$(ii)$ Under an admissible control rule $\Pb$,
		$B$ and $W$ are standard Brownian motions,
		$\Lambda$ is the $\Pc(A)$--valued process induced by the control process,
		$X$ is the controlled process, and $\mub$ is the conditional distribution of the control and controlled process.
		The process $Y$ will only be really used to introduce the relaxed formulation.
		In particular, when $\sigma_0=0$ or $\ell =0,$ one has $Y=X$.
		
		\vspace{0.5em}
	
		$(iii)$ Notice that $\muh_t$ is $\Gcb_t$--measurable, it follows that \eqref{eq:muh_property} is equivalent to 
		\begin{align} \label{eq:muh_property-details}
			\muh_t (\omb)=\Pb^{\Gcb_t}_{\omb} \circ (X_{t \wedge \cdot},Y_{t \wedge \cdot}, \Lambda^t,W)^{-1}=\Pb^{\Gcb_T}_{\omb} \circ (X_{t \wedge \cdot},Y_{t \wedge \cdot}, \Lambda^t,W)^{-1},\;\mbox{\rm for}\;\Pb\mbox{--a.e.}\;\omb \in \Omb.
		 \end{align}
	\end{remark}

\subsubsection{The strong formulation on the canonical space}

	To reformulate the strong formulation \eqref{eq:def_VS} of the control problem on the canonical space $\Omb$,
	it is enough to consider the class of measures induced by the controls and the controlled processes on the canonical space.
	Recall that for each $\nu \in \Pc_p(\R^n)$, $\P_{\nu}$ is defined in \Cref{sec:strong} as a probability measure on $(\Omega,\Fc)$, and that for any $\alpha \in \Ac_p(\nu)$, the controlled McKean--Vlasov SDE \eqref{eq:MKV_SDE} has a unique strong solution $X^\alpha$. Let us further define
	\begin{align*}
		Y^{\alpha}_t
		:=
		X^\alpha_t-\int_0^t \sigma_0(s, X^\alpha_{s \wedge \cdot}, \mub^\alpha_s,\alpha_s ) \mathrm{d}B_s,\; t\in[0,T],
		\;
		\Lambda^\alpha_t(\mathrm{d}a)\mathrm{d}t
		:=
		\delta_{\alpha_t}(\mathrm{d}a)\mathrm{d}t,
		\; \mbox{and}
		\;
		\muh^\alpha
		:=
		\Lc^{\P_\nu} \big( X^\alpha, Y^\alpha,\Lambda^\alpha,W \big| \Gcb_T \big).
	\end{align*}
	Then the set of all strong control rules $\Pcb_S(\nu)$ is defined as a collection of probability measures on the canonical space $(\Omb,\Fcb)$ induced by $\alpha$:
	\begin{align*}
	    \Pcb_S(\nu)
	    :=
	    \big\{
	        \P_\nu \circ \big(X^\alpha,Y^\alpha,\Lambda^\alpha,W, B, \muh^\alpha \big)^{-1}:\alpha \in \Ac_p(\nu)
	    \big\},
	\end{align*}
	and it is straightforward to see that
	\begin{align} \label{eq:def_VS-canon}
		V_S(\nu) 
		=
		\sup_{\Pb \in \Pcb_S(\nu)} J\big(\Pb\big),
		 \;\mbox{with}\;J\big(\Pb\big) 
		:=
		\E^{\Pb} \bigg[\iint_{[0,T]\x A} L\big(t, X_{t\wedge\cdot}, \mub_t, a \big) \Lambda_t(\mathrm{d}a) \mathrm{d}t +g\big( X_{T\wedge\cdot}, \mu_T \big) \bigg].
	\end{align}

	Let
	\[
			\Lc_0[A]
			:=
			\big\{
				\text{\rm All Borel measurable functions}\;\phi:[0,T] \x \R^n \x \Cc^d \x \Cc^\ell \longrightarrow A
			\big\}.
	\]
	\begin{proposition}{\rm\cite[Definition 4.2., Lemma 4.3.]{djete2019mckean}} \label{prop:strong_canonical}
		We have, for all $\nu \in \Pc_p(\R^n)$
		\begin{align*}
			\Pcb_S(\nu)
			=
			\Big\{
				\Pb \in \Pcb_A(\nu):\exists\; \phi \in \Lc_0[A], \;\Pb\big[\Lambda_t(\mathrm{d}a)\mathrm{d}t=\delta_{\phi(t,X_0,W_{t \wedge \cdot},B_{t \wedge \cdot})}(\mathrm{d}a)\mathrm{d}t \big]=1
			\Big\}.
		\end{align*}
	\end{proposition}
	
	\begin{remark} \label{rem:nonContinuityJ}
		Notice that the map $\Pc(\Cc^n \x \Cc^n \x \Cc^d \x \M) \ni  \muh \longmapsto \delta_{\mub_t}(\mathrm{d} \nub )\mathrm{d}t \in \M(\Cc^n \x A)$ is generally not continuous.
		Consequently, $\Pb \longmapsto J(\Pb)$ is not continuous in general, even if $L$ and $g$ are both bounded and continuous.
	\end{remark}

\subsubsection{The weak formulation on the canonical space}

	Now we introduce the set of weak control rules which is also a subset of $\Pcb_A(\nu)$.

	\begin{definition} \label{def:weak_ctrl_rule}
		Let $\nu \in \Pc_p(\R^n)$, an admissible control rule $\Pb \in \Pcb_A(\nu)$ is called a weak control rule with initial condition $\nu$ if 
		$\Pb \big[ \Lambda \in \M_0 \big]=1$.
		Denote
		\begin{align*}
			\Pcb_W( \nu)
			:=
			\big\{ \Pb \in \Pcb_A(\nu):\Pb \big[ \Lambda \in \M_0 \big]=1 \big\}.
		\end{align*}
	\end{definition}
	
	The next proposition links the set $\Pcb_W( \nu)$ and the weak control terms $\Gamma_W(\nu).$
	\begin{proposition}
		Let $\nu \in \Pc_p(\R^n)$ and $\gamma \in \Gamma_W(\nu)$. 
		Define, for any $t\in[0,T]$,
		\[
			Y^{\gamma}_t
			:=
			X^\gamma_t-\int_0^t \sigma_0(s, X^\gamma_{s\wedge\cdot}, \mub^\gamma_s,\alpha^{\gamma}_s ) \mathrm{d}B^{\gamma}_s,
			~
			\Lambda^\gamma_t(\mathrm{d}a)\mathrm{d}t
			:=
			\delta_{\alpha^{\gamma}_t}(\mathrm{d}a)\mathrm{d}t,
			~\mbox{\rm and}~
			\muh^\gamma
			:=
			\Lc^{\P^\gamma} \big( \big( X^\gamma, Y^\gamma,\Lambda^\gamma,W^\gamma \big) \big| {\Gc^{\gamma}_T} \big).
		\]
		Then with $J$ defined in \eqref{eq:def_VS-canon}, we have
		\begin{equation} \label{eq:redef_VW}
			\Pcb_W(\nu) 
			=
			\big\{
				\P^{\gamma} \circ \big(X^\gamma,Y^\gamma,\Lambda^\gamma,W^{\gamma}, B^{\gamma}, \muh^\gamma \big)^{-1}
				:\gamma \in \Gamma_W(\nu)
			\big\},
			~\mbox{\rm and}~
			V_W(\nu) = \sup_{\Pb \in \Pcb_W(\nu)} J\big(\Pb\big).
		\end{equation}	
	\end{proposition}
	\begin{proof}
		With a slight extension of {\cite[Lemma 4.3]{djete2019mckean}} by taking into account the process $Y$ and the small changes in the presentation of the definition of weak controls $\Gamma_W(\nu)$),
		every weak control rule $\Pb\in \Pcb_W(\nu)$, together with the canonical space $\Omb$ and canonical processes, can be viewed as a weak control $\gamma \in \Gamma_W(\nu)$.
		Conversely, every weak control $\gamma$ induces a weak control rule $\Pb \in \Pcb_W(\nu)$ on the canonical space. 
		It follows that \eqref{eq:redef_VW} holds true (see also {\cite[Corollary 4.5]{djete2019mckean}}).
	\end{proof}

	\begin{remark}
		By {\rm \Cref{prop:strong_canonical}}, it is straightforward to see that for all $\nu \in \Pc_p(\R^n)$
		\begin{align*}
			\Pcb_S(\nu)
			=
			\big\{
				\Pb \in \Pcb_W(\nu):\exists\; \phi \in \Lc_0[A],\;\Pb\big[\Lambda_t(\mathrm{d}a)\mathrm{d}t=\delta_{\phi(t,X_0,W_{t \wedge \cdot},B_{t \wedge \cdot})}(\mathrm{d}a)\mathrm{d}t \big]=1
			\big\}.
		\end{align*}
		In particular, as expected, any strong control rule is also a weak control rule, i.e.
		$
			\Pcb_S(\nu) 
			\subset
			\Pcb_W(\nu).
		$
	\end{remark}

\subsubsection{The relaxed formulation}

	In the classical optimal control theory, the set of relaxed control rules has been introduced to recover a closed and convex set, while ensuring that its elements could be appropriately approximated by strong or weak control rules. The point was that it then becomes easier in this formulation to deduce the existence and stability properties of the optimal solution, while ensuring under mild conditions that the value of the problem is not modified.
	In our context, when the coefficient functions $(b, \sigma, \sigma_0,L,g)$ do not depend on the marginal distribution $\nub$ or $\nu$, so that the control problem degenerates to the classical one,
	the relaxed control rule coincides with the admissible control rule $\Pcb_A(\nu)$ in \Cref{def:admissible_ctrl_rule} (or equivalently \Cref{def:weak_ctrl_rule} by removing the constraint $\Pb[ \Lambda \in \M_0] = 1$).
	For general McKean--Vlasov control problems, it is not hard to prove that $\Pcb_A(\nu)$ is closed and convex.
	However, in general, it is not the closure of the set of strong or weak control rules in the context with common noise (see \Cref{ex_relaxed-admissible} below).
	This motivated us to consider a more restrictive case, where the common noise is not controlled, for which we are able to provide an appropriate relaxed control rule set as a subset of $\Pcb_A(\nu)$, which is both convex and the closure of $\Pcb_S(\nu)$ or $\Pcb_W(\nu)$.

	\begin{assumption}\label{assum:constant_case}
		There exist Borel measurable functions $(b^{\circ}, \sigma^{\circ}, L^{\circ}): [0,T] \x \Cc^n \x \Pc(\Cc^n) \x A \longrightarrow \R^n \x \S^{n \x d}$ and $\sigma^{\circ}_0: [0,T] \x \Cc^n \x \Pc(\Cc^n) \longrightarrow \S^{n \x \ell}$ such that, 
		for all $(t,\xb,\nub,a) \in [0,T] \x \Cc^n \x \Pc(\Cc^n \x A) \x A$, with $\nu(\mathrm{d}\xb) :=\nub(\mathrm{d}\xb,A)$
		\[
			(b,\sigma, L)(t,\xb,a,\nub)
			=
			(b^{\circ},\sigma^{\circ}, L^{\circ})(t,\xb,a,\nu),
			\;\mbox{\rm and}\;
			\sigma_0(t,\xb,a,\nub)=\sigma^{\circ}_0(t,\xb,\nu).
		\]
		By abuse of notations, we still write $(b, \sigma, L, \sigma_0)$ in lieu of $(b^{\circ},\sigma^{\circ}, L^{\circ}, \sigma^{\circ}_0)$.
	\end{assumption}    

	We next introduce a martingale problem on $(\Omh,\Fch)$. 
	For any $(t,\xb,\nu,a) \in [0,T] \x \Cc^n \x \Pc(\Cc^n) \x A$, let
	\[
		\hat b \big(t,\xb, \nu,a \big)
		:=
		\begin{pmatrix}
		b(t,\xb,\nu,a) \\ 0_d
		\end{pmatrix},\; 
		\hat a \big(t, \xb, \nu, a \big)
		:=
		\begin{pmatrix} 
			\sigma(t,\xb,a,\nu)  \\ 
			\mathrm{I}_{d} 
		\end{pmatrix}
		\begin{pmatrix} 
			\sigma(t,\xb,a,\nu)  \\ 
			\mathrm{I}_{d}
		\end{pmatrix}^{\top},\]
	and then, for all $\varphi \in C^2_b(\R^{n + d})$ and $(t,\xb,\yb,\wb, \nu, a)\in[0,T]\times\Cc^n\times\Cc^n\times\Cc^d\times\Pc(\Cc^n)\times A$,
	let
	\begin{align} \label{eq:conditionnal_generator}		
		\widehat \Lc_t \varphi \big( \xb, \yb, \wb, \nu, a \big)
		:=
		\hat b(t, \xb, \nu, a)\cdot \nabla\varphi(\yb(t),\wb(t))
		+ 
		\frac{1}{2}\mathrm{Tr}\big[ \hat a(t, \xb, \nu, a) \nabla^2 \varphi(\yb(t),\wb(t))\big].
	\end{align}
	Then given  a family $(\nu(t))_{0 \le t \le T}$ of  probability measures in $\Pc(\Cc^n)$ such that $ [0,T]\ni t \longmapsto \nu(t) \in \Pc(\Cc^n)$ is Borel measurable, and $\varphi \in C^2_b(\R^{n + d})$, we introduce a process $(\widehat{S}^{\varphi, \nu}_t)_{t \in [0,T]}$ on $(\Omh,\Fch)$ by
	\begin{equation}\label{eq:Mvarphi}
		\widehat{S}^{\varphi, \nu}_t
		:=
		\varphi \big(\Yh_t, \Wh_t \big)-\varphi(\Yh_0, \Wh_0)
		-
		\iint_{[0,t]\x A}  \widehat \Lc_s \varphi\big(\Xh, \Yh, \Wh, \nu(s), a \big) \Lambdah_s(\mathrm{d}a)\mathrm{d}s,
	\end{equation}
	where for a Borel function $\phi:[0,T] \to \R,$ we write $\int_0^{\cdot} \phi(s)\mathrm{d}s:=\int_0^\cdot \phi^{+}(s)\mathrm{d}s-\int_0^\cdot \phi^{-}(s)\mathrm{d}s$ with the convention $\infty - \infty = -\infty.$

	\begin{definition}[Relaxed control rule]
	\label{def:relaxed_ctrl_rule}
		Let $\nu \in \Pc_p(\R^n)$. A probability measure $\Pb \in \Pc(\Omb)$ is called a relaxed control rule with initial condition $\nu$, if 
		$\Pb \in \Pcb_A(\nu)$, and moreover, for $\Pb$--a.e. $\omb \in \Omb$,
		the process $\widehat{S}^{\varphi, \mu(\omb)}$ is an  $\big(\Fh,\muh(\omb) \big)$--martingale for each $\varphi \in C^2_b(\R^n \x \R^d)$,
		where $\mu(\omb) := (\mu_t(\omb))_{t \in [0,T]}$ is defined from $\muh(\omb)$ in \eqref{eq:muh2mu}.
    	\end{definition}
	
	Let $\Pcb_R(\nu)$ be the set of all relaxed control rules with initial condition $\nu$, i.e. 
		\[
			\Pcb_R(\nu)
			:=
			\Big\{
				\Pb \in \Pcb_A(\nu): \Pb\;\mbox{--a.e.}\;\omb \in \Omb,\;\big( \widehat{S}^{\varphi, \mu(\omb)}_t\big)_{t \in [0,T]}\;\mbox{\rm is an}\;\big(\Fh,\muh(\omb) \big)\mbox{--martingale for each}\;\varphi \in C^2_b(\R^n \x \R^d)
			\Big\}.
		\]
	The relaxed formulation of the McKean--Vlasov control problem is then defined by, with $J\big(\Pb\big)$ given in \eqref{eq:def_VS-canon},
	\[
		V_R(\nu) := \sup_{\Pb \in \Pcb_R(\nu)} J\big(\Pb\big).
	\]
	
	\begin{remark} \label{rem:ContinuityJ}
		Under {\rm \Cref{assum:constant_case}}, the reward function $L$ depends on $\nu$ $($instead of $\nub)$.
		In this case, and in contrast to  the general situation in {\rm \Cref{rem:nonContinuityJ}}, 
		the map $\Pc(\Cc^n \x \Cc^n \x \Cc^d \x \M) \ni  \muh \longmapsto \delta_{\mu_t}(\mathrm{d} \nu )\mathrm{d}t \in \M(\Cc^n)$ is continuous, 
		so that $\Pb \longmapsto J(\Pb)$ is lower semi--continuous $($resp. continuous$)$ as soon as $L$ and $g$ are lower semi--continuous and bounded from below $($resp. continuous and bounded$)$.
	\end{remark}
    
	We observe that $\Pcb_R(\nu) \subseteq \Pcb_A(\nu)$ by definition,
	the next example shows that $\Pcb_R(\nu)$ is a proper subset of $\Pcb_A(\nu)$.

	\begin{example} \label{ex_relaxed-admissible}
		Let us consider the case where: $n=d=\ell=1,$ $\nu=\delta_{0},$ $A = \{a_1, a_2\} \subset \R,$ $b=0,$ $\sigma(t,\xb,a,\nub)=a\mathrm{I}_{n},$ and $\sigma_0=\mathrm{I}_{n}$. 
		Consider a filtered probability space $(\Om^\star,\Fc^\star, \F^\star, \P^\star)$ supporting an $\R^{d+d+\ell}$--valued standard Brownian motion $(W^1,W^2,B^\star)$,
		let
		\[
			X^\star_t:= a_1 \frac{\sqrt{2}}{2} W^1_t + a_2  \frac{\sqrt{2}}{2} W^2_t + B^\star_t,
			\;
			W^\star_t:=  \frac{\sqrt{2}}{2} W^1_t +  \frac{\sqrt{2}}{2} W^2_t,
			\;
			\overline W^\star_t:=  \frac{\sqrt{2}}{2} W^1_t -  \frac{\sqrt{2}}{2} W^2_t,
			\;
			\Gc^\star_t := \sigma \big((B^\star_s, \overline W^\star_s): 0\leq s \le t \big).
		\]
		By setting  $Y^\star_\cdot:= X^\star_\cdot - B^\star_\cdot$ and $\Lambda^\star_t(\mathrm{d}a)\mathrm{d}t:=\frac{1}{2}\delta_{a_1}(\mathrm{d}a)\mathrm{d}t+ \frac{1}{2}\delta_{a_2}(\mathrm{d}a)\mathrm{d}t$,
		it is direct to check that
		\begin{align*}
		    \Pb
		    :=
		    \Lc^{\P^\star} \Big( X^\star, Y^\star, \Lambda^\star, W^\star, B^\star, \Lc^{\P^\star} \big( X^\star, Y^\star, \Lambda^\star, W^\star, B^\star \big| \Gc^\star_T \big) \Big)  \in\Pcb_A(\nu).
		\end{align*}
		However, one observes that
		\[
			Y^\star_\cdot = \bigg( \frac12 a_1 + \frac12 a_2 \bigg) W^\star_\cdot + \bigg( \frac12 a_1 - \frac12 a_2 \bigg) \overline W^\star_\cdot,
		\]
		is not an It\^o process under the conditional law $\P^\star$ knowing $\Gc^\star_T$.
		Consequently, one has $\Pb \notin \Pcb_R(\nu)$.
	\end{example}    
	
	\begin{remark}
		$(i)$ The martingale problem under $\Pb$ in {\rm \Cref{def:admissible_ctrl_rule}} involves conditional distributions in the coefficient functions,
		which creates some regularity problem in the approximation procedure, since conditional distributions are not continuous with respect to joint distributions.
		By considering the conditional martingale problem under $\muh(\omb)$ in {\rm \Cref{def:relaxed_ctrl_rule}}, 
		the $\mu(\omb)$ term in the coefficient functions becomes deterministic, which in turn allows to avoid the regularity problem.
		This $($conditional$)$ martingale problem is partially inspired from a technical proof of {\rm\cite{lacker2017limit}},
		but in our context with common noise, we need to consider a family of martingale problems, and deal with some non--trivial measurability issues.
		Notice also that the canonical processes $Y$ and $\Yh$ do not play an essential role in the strong or weak formulations, 
		but they are crucially used in the conditional martingale problem in {\rm\Cref{def:relaxed_ctrl_rule}}.

		\medskip
		
		$(ii)$ With our techniques, we are only able to prove the equivalence $V_W=V_R$ $($c.f. {\rm\Cref{thm:equivalence}}$)$, as well as the desired approximation results, under {\rm\Cref{assum:constant_case}}.
		For more general cases, it seems to be a very challenging problem that we would like to leave for future research. 
		We nonetheless point out the fact that the great majority of the extant literature on either mean--field games or McKean--Vlasov control problems with common noise, does not allow for $\sigma_0$ or $\sigma$ to be controlled as well, see for instance {\rm\citeauthor*{ahuja2016wellposedness} \cite{ahuja2016wellposedness}, \citeauthor*{bensoussan2015master} \cite{bensoussan2015master}, \citeauthor*{cardaliaguet2015master} \cite{cardaliaguet2015master}, \citeauthor*{carmona2013mean2} \cite{carmona2013mean2}, \citeauthor*{carmona2014mean} \cite{carmona2014mean}, \citeauthor*{graber2016linear} \cite{graber2016linear}, \citeauthor*{gueant2011mean} \cite{gueant2011mean}, \citeauthor*{kolokoltsov2019mean} \cite{kolokoltsov2019mean}, \citeauthor*{lacker2016general} \cite{lacker2016general}}, and  {\rm\citeauthor*{lacker2015translation} \cite{lacker2015translation}}. Notable exceptions are {\rm\citeauthor*{carmona2014master} \cite{carmona2014master}}, though the discussion in the general setting remains at a rather informal level there, the monograph by {\rm\citeauthor*{carmona2018probabilisticII} \cite{carmona2018probabilisticII}}, although all the main results given have uncontrolled common noise, {\rm\citeauthor*{pham2016dynamic} \cite{pham2016dynamic}}, though the problem is considered in a Markovian setting, with feedback controls, and no limit theory is explored, {\rm\citeauthor*{pham2016linear} \cite{pham2016linear}} and {\rm\citeauthor*{yong2013linear} \cite{yong2013linear}} where only linear quadratic problems are considered, {\rm\citeauthor*{bayraktar2016randomized} \cite{bayraktar2016randomized}}, though no limit theory is addressed there as well, and our companion paper {\rm\cite{djete2019mckean}}, which encompasses the last two mentioned ones. We would also like to highlight the recent work of {\rm\citeauthor*{acciaio2018extended} \cite{acciaio2018extended}} which derives a general stochastic Pontryagin maximum principle for McKean--Vlasov control problems in strong formulation without common noise, where the coefficients depend on the joint law of the control and the state process. The authors also consider a weak formulation for their problem, but with uncontrolled volatility and for a drift which does not depend on the law of the controls, deriving again a stochastic maximum principle. Finally {\rm\citeauthor*{elie2016tale} \cite{elie2016tale}} considers a contract theory problem with a principal and mean--field agents, without common noise and where only the drift is controlled but can depend on the law of the controls, as well as {\rm\citeauthor*{elie2019mean} \cite{elie2019mean}} which also considers a contract theory problem, but with common noise and volatility controls.
	\end{remark}
	
	\begin{remark}
		As in {\rm\cite{djete2019mckean}}, our formulation covers the case without common noise by taking $\ell=0$ $($or $\sigma_0 \equiv 0)$. 
		Nevertheless, unlike  {\rm\cite{djete2019mckean}}, 
		we need to consider the case $\ell = 0$ separately $($see {\rm \Cref{thm:equivalence}} and {\rm \Cref{rem:3.2}.$(i)$} below$)$.
	\end{remark}

	We next show that $\Pcb_W(\nu) \subset \Pcb_R(\nu)$, where we use crucially the fact that $\muh_t$ is the conditional law of $(X_{t \wedge \cdot},Y_{t \wedge \cdot},\Lambda^t, W)$, and not only of $(X_{t \wedge \cdot}, \Lambda^t)$.

	\begin{proposition} \label{Proposition:property_WeakControl}
		Let $\nu \in \Pc_p(\Cc^n)$ and $\Pb \in \Pcb_A(\nu)$.
		Then for $\Pb$--almost every $\omb \in \Omb$, 
		$\Wh$ is an $\big(\Fh,\muh(\omb)\big)$--Brownian motion.
		In particular,  under {\rm\Cref{assum:constant_case}}, every $\Pb \in \Pcb_W(\nu)$ belongs to $\Pcb_R(\nu)$.
	\end{proposition}
    
	\begin{proof}
		Let $\Pb \in \Pcb_A(\nu)$, $0 \le s \le t \le T,$ $\phi \in C_b(\R^d),$ $\varphi \in C_b(\Cc^n \x \Cc^n \x \M \x \Cc^d)$ and $\psi \in C_b(\Cc^\ell \x C([0,T]; \Cc^n \x \Cc^n \x \M \x \Cc^d))$. Notice that $W$ is an $(\Fb,\Pb)$--Brownian motion, independent of $\Gcb_T$ under $\Pb$. Therefore, it follows that
		\begin{align*}
			&\ \E^\Pb \big[ \phi (W_t-W_s) \varphi (X_{s \wedge \cdot},Y_{s \wedge \cdot},\Lambda^s,W_{s \wedge \cdot}) \psi (B_{s \wedge \cdot},\muh_{s \wedge \cdot}) \big]\\
			=&\ 
			\E^\Pb \big[ \phi (W_t-W_s) \big] \E^\Pb \Big[ \E^\Pb \big[\varphi (X_{s \wedge \cdot},Y_{s \wedge \cdot},\Lambda^s,W_{s \wedge \cdot}) \big| \Gcb_s \big] \psi (B_{s \wedge \cdot},\muh_{s \wedge \cdot}) \Big] \\
			=&\
			\E^\Pb \Big[ \E^\Pb \big[ \phi (W_t-W_s) \big| \Gcb_s \big] \E^\Pb \big[\varphi (X_{s \wedge \cdot},Y_{s \wedge \cdot},\Lambda^s,W_{s \wedge \cdot}) \big| \Gcb_s \big] \psi (B_{s \wedge \cdot},\muh_{s \wedge \cdot}) \Big].
		\end{align*}
        This implies that
        \[
            \E^{\Pb} \big[ \phi (W_t-W_s) \varphi (X_{s \wedge \cdot},Y_{s \wedge \cdot},\Lambda^s,W_{s \wedge \cdot}) \big| \Gcb_s \big]
            =
            \E^{\Pb} \big[ \phi (W_t-W_s) \big| \Gcb_s\big]\E^{\Pb} \big[ \varphi (X_{s \wedge \cdot},Y_{s \wedge \cdot},\Lambda^s,W_{s \wedge \cdot}) \big| \Gcb_s \big],\;\Pb\mbox{--a.s.}
        \]
        By \eqref{eq:muh_property} in  \Cref{def:admissible_ctrl_rule}, it follows that for $\Pb$--a.e. $\omb\in\Omb$
	\[
		\E^{\hat\mu(\omb)} \big[ \phi (\Wh_t- \Wh_s) \varphi (\Xh_{s \wedge \cdot},\Yh_{s \wedge \cdot},\Lambdah^s, \Wh_{s \wedge \cdot}) \big]
		=
		\E^{\hat\mu(\omb)} \big[ \phi (\Wh_t- \Wh_s) \big] ~ \E^{\hat\mu(\omb)} \big[ \varphi (\Xh_{s \wedge \cdot},\Yh_{s \wedge \cdot},\Lambdah^s, \Wh_{s \wedge \cdot}) \big].
	\]
	In other words, $\Wh$ has independent increments with respect to $\Fh$ under $\muh(\omb)$, for $\Pb$--almost every $\omb \in \Omb.$ 
	
	\medskip
	
	Further, notice that under $\Pb$,  $W$ is a Brownian motion independent of $(B, \muh)$, 
	then $W$ is still a Brownian motion under the conditional law of $\Pb$ knowing $\Gcb_T$.
	It follows that the continuous process $\Wh$ has independent and (Gaussian) stationary increment w.r.t. $(\Fh, \muh(\omb))$,
	and hence it is an $(\Fh,\muh(\omb))$--Brownian motion, for $\Pb$--a.e. $\omb \in \Omb$. 

	\medskip

	Now, let \Cref{assum:constant_case} hold true and $\Pb \in \Pcb_W(\nu).$
	Using the definition of weak control rules in \Cref{def:weak_ctrl_rule} (see also proof of Proposition \ref{prop:eqivalence_def_relaxed}), it is direct to deduce that for $\Pb$--a.e. $\omb \in \Omb$, $\Wh$ is an $\big(\Fh,\muh(\omb)\big)$--Brownian motion, and
        \begin{align*}
		\Yh_t
		=
		\Xh_0 + \int_0^t b \big(s, \Xh_{\cdot},\hat \alpha_s, \mu_s(\omb) \big) \mathrm{d}s 
			+ \int_0^t \sigma \big(s, \Xh_{\cdot}, \hat \alpha_s, \mu_s(\omb) \big) \mathrm{d}\Wh_s,
		~t \in [0,T], ~ \muh(\omb)\mbox{\rm--a.s.},
	\end{align*}
	where $(\hat \alpha_t)_{t \in [0,T]}$ is an $\Fh$--predictable process satisfying $\Lambdah_t(\mathrm{d}a)\mathrm{d}a=\delta_{\hat \alpha_t}(\mathrm{d}a)\mathrm{d}t.$
	It follows that, for $\Pb$--a.e. $\omb \in \Omb$, $\widehat{S}^{\varphi, \mu(\omb)}$ is an $\big(\Fh,\muh(\omb) \big)$--martingale for each $\varphi \in C^2_b(\R^n \x \R^d)$,
	and hence $\Pb \in \Pcb_R(\nu)$.
	\end{proof}

\subsection{A large population stochastic control problem with common noise}
\label{subsec:N_controlpb}
	One of the main objectives of this paper is to provide the limit theory for the McKean--Vlasov control problem, that is,  the problem $V_S(\nu)$ in \eqref{eq:def_VS} can be seen as the limit of a large population problem.
	Let $N$ be a positive integer, we consider the canonical space
	\[ 
		\Om^N 
		:=
		\big(\R^n \x \Cc^{d} \big)^N \x \Cc^{\ell},
	\]
	with canonical process $\big((X_0^1, \dots, X_0^N), (W^1, \dots, W^N), B \big)$
	and canonical filtration $\F^N := (\Fc^N_t)_{0 \le t \le T}$ defined by
	\[
		\Fc^N_t
		:=
		\sigma \big( (X_0^i, W^i_s, B_s): i\in\{1, \dots, N\}, \; s \in [0,t] \big),\; t\in[0,T].
	\]
	Fix some $(\nu^1,\dots,\nu^N) \in \Pc_p(\R^n)^N$, and define $\nu_N:=\nu^1 \otimes \dots \otimes \nu^N$ the corresponding product measure. 
	We consider the probability measure $\P^N_{\nu}$ on $\big(\Om^N, \Fc^N)$ with $\Fc^N:= \Bc(\Om^N)$,
	under which $X_0 := (X_0^1, \dots, X_0^N) $\ has distribution $ \nu_N$, and $(W^1, \dots, W^N, B)$ is a standard Brownian motion, independent of $X_0$.
	Let us denote by $\Ac^N_p(\nu_N)$ the collection of all processes $\alpha := (\alpha^i)_{i=1, \dots, N}$,
	where each $\alpha^i := (\alpha^i_t)_{0 \le t \le T}$ is an $A$--valued, $\F^N$--predictable process satisfying
	\[
		\E^{\P^N_{\nu}} \bigg[ \int_0^T \big( \rho(\alpha^i_s, a_0) \big)^p \mathrm{d}s \bigg] 
		< \infty.
	\]
	Then under standard Lipschitz conditions on the coefficient functions (see \Cref{assum:main1}), for every fixed  $(\alpha^1,\dots,\alpha^N) \in \Ac^N_p(\nu_N)$, 
	there is a unique $(\R^n)^N$--valued $\F^N$--adapted continuous process $(X^{\alpha,1}, \dots, X^{\alpha,N})$ satisfying: for $i\in\{1,\dots,N\},$ $\E^{\P^N_{\nu}}\big[\|X^{\alpha,i}\|^p \big]< \infty$ and
	\begin{equation} \label{eq:N-agents_StrongMV_CommonNoise}
		X^{\alpha,i}_t
		=
		X_0^i
		+
		\int_0^t b\big(s, X^{\alpha,i}_{s\wedge\cdot},\varphi^{N}_s ,\alpha^i_s \big) \mathrm{d}s
		+
		\int_0^t\sigma \big(s,X^{\alpha,i}_{s\wedge\cdot},\varphi^{N}_s ,\alpha^i_s \big) \mathrm{d}W^i_s
		+
		\int_0^t\sigma_0 \big(s,X^{\alpha,i}_{s\wedge\cdot},\varphi^{N}_s, \alpha^i_s \big) \mathrm{d}B_s,\; t\in[0,T],
	\end{equation}
	with
	\[
		\varphi^{N}_{s}(\mathrm{d} \xb, \mathrm{d}a) := \frac{1}{N}\sum_{i=1}^N \delta_{(X^{\alpha,i}_{s\wedge \cdot},\alpha^i_s )}(\mathrm{d} \xb, \mathrm{d}a),
		~\mbox{and}~
		\varphi^{N,X}_{s}(\mathrm{d} \xb) := \frac{1}{N}\sum_{i=1}^N \delta_{X^{\alpha,i}_{s\wedge \cdot}}(\mathrm{d} \xb),
		~
		s \in [0,T].
	\]
	The value function of the large population stochastic control problem is then defined by
	\begin{equation} \label{eq:N-Value_StrongForm}
		V^{N}_S(\nu^1,\dots,\nu^N)
		:=
		\sup_{\alpha\in \Ac^N_p(\nu_N)}
		J_N(\alpha),
		\;\mbox{where}\;
		J_N(\alpha)
		:=
		\frac{1}{N}\sum_{i=1}^N
		\E^{\P^N_{\nu}} \bigg[
		\int_0^T L\big(t,X^{\alpha,i}_{t\wedge\cdot},\varphi^{N}_{t} ,\alpha^i_t \big) \mathrm{d}t 
		+ 
		g \big( X^{\alpha,i}_{T\wedge\cdot}, \varphi^{N,X}_{T} \big)
		\bigg].
	\end{equation}

\section{Main results}    
\label{sec:MainResults}

	Let us now provide the main results of the paper.
	The first one consists in the equivalence between different formulations of the McKean--Vlasov control problem.
	Recall that the constants $p$, $p^\prime$, and $\hat p$ are fixed in {\rm\Cref{assum:main1}}.
	\begin{theorem} \label{thm:equivalence}
		$(i)$ 
		Let  {\rm\Cref{assum:main1}} hold true. Then, for every $\nu \in \Pc_p(\R^n)$, the set $\Pcb_W(\nu)$ is non--empty and convex.
		Suppose in addition that {\rm\Cref{assum:constant_case}} holds true,
		then $\Pcb_R(\nu)$ is a non--empty convex closed subset of $\Pc_p(\Omb)$, under the Wasserstein topology $\Wc_p$.  

		\vspace{0.5em}

		$(ii)$ Let {\rm\Cref{assum:main1}} hold true. 
		Then for every $\nu \in \Pc_p(\R^n)$, one has $V_S(\nu) = V_W(\nu)$. 
		If in addition $\ell \neq 0$, then every weak control rule in $\Pcb_W(\nu)$ is the limit of a sequence of strong control rules in $\Pcb_S(\nu)$, 
		under the Wasserstein distance $\Wc_p$ on $\Pc_p(\Omb)$.
		
		\vspace{0.5em}
		
		$(iii)$ Let  {\rm\Cref{assum:main1}} and {\rm\Cref{assum:constant_case}} hold true, $\nu \in \Pc_{p^{\prime}}(\R^n)$, and $A \subset \R^j$ for some $j \ge 1$.
		Then the set $\Pcb_W(\nu)$ is dense in the closed set $\Pcb_R(\nu)$ under $\Wc_p$,
		and consequently
		\[
			V_S(\nu)=V_W(\nu)= V_R(\nu).
		\]
		If, in addition, $L$ and $g$ are continuous in all arguments, there exists some $\Pb^\star \in \Pcb_R(\nu)$ such that $V_R(\nu)=J\big(\Pb^\star\big)$.
	\end{theorem}
	
	\begin{remark}\label{rem:3.2}
		When $\ell =0$, or $\ell \neq 0$ and $\sigma_0=0$, the $($strong formulation of the$)$ McKean--Vlasov control problem \eqref{eq:def_VS} or \eqref{eq:def_VS-canon}, reduces to the non--common noise context.
		However, in the weak formulation \eqref{eq:redef_VW}, the $($conditional$)$ distribution term $\muh$ may still be random under a weak control rule $\Pb \in \Pcb_W(\nu)$.
		In the case $\ell \neq 0$ and $\sigma_0=0$, the Brownian motion $B$ can be seen as an external noise in \eqref{eq:def_VS-canon}, which allows to track the randomness of $\muh$ and approximate a weak control rule $\Pb \in \Pcb_W(\nu)$ by strong control rules.
		This is also the main reason why we consider the case $\ell \neq 0$ separately in {\rm\Cref{thm:equivalence}.$(ii)$}.
	\end{remark}
		
	\begin{remark}
		For the equivalence result $V_W=V_R$, {\rm\Cref{Proposition:Continuity-Existence}}, and also {\rm\Cref{thm:limit}}, assume that $A \subset \R^j$. But this is by no means a crucial point. Roughly speaking, what we actually need is that the set $A$ can be appropriately approximated by compact sets. For instance, our results still hold if $A$ is a $\sigma$--compact space $($that is to say the union of countably many compact subspaces$)$. We assumed here that $A \subset \R^j$ for simplicity.
	\end{remark}

	\begin{remark}
		The results in {\rm\Cref{thm:equivalence}} extend those in the no--common noise setting in {\rm\citeauthor*{lacker2017limit} \cite{lacker2017limit}}.
		Nevertheless, we insist on the fact that the equivalence results, the formulation of the weak and relaxed control rules, and the technical proofs below are not merely extensions of those in {\rm \cite{lacker2017limit}}, and are in fact quite different.
		The main reason is that
		with the presence of the common noise, the $\mub^{\alpha}$ term in \eqref{eq:MKV_SDE}--\eqref{eq:def_VS} is a conditional distribution term, 
		which, in general, is not continuous with respect to the joint distribution of $(X^{\alpha}, \alpha, W^{\alpha}, B^{\alpha})$.
		Moreover, the equivalence result  $V_S = V_W$ is also crucially used to establish the dynamic programming principle in our companion paper {\rm\cite[Theorem 3.4]{djete2019mckean}}.
	\end{remark}
    
    \begin{remark}
    A natural question that we have not addressed is that of the existence of so--called feedback controls, since {\rm\Cref{thm:equivalence}.$(iii)$} only gives existence of an optimal relaxed control. It is known in classical control theory that {\rm\citeauthor*{filippov1962certain}'}s condition {\rm\cite{filippov1962certain}}, which was notably used by {\rm\citeauthor*{haussmann1990existence} \cite{haussmann1990existence}}, and by {\rm\citeauthor*{lacker2015mean} \cite{lacker2015mean,lacker2017limit}} for {\rm MFGs} and McKean--Vlasov control problems without common noise, is usually sufficient to obtain, from any relaxed control, a control depending on the trajectories of $X$ only, and which achieves no worse value. In the common noise context, things become slightly more subtle. The intuitive result is that one should be able to obtain a similar result but with controls depending on the trajectories of both $X$ and $\mu$. In a work in progress, {\rm\citeauthor*{lacker2020superposition} \cite{lacker2020superposition}} will exactly prove such a result, with the additional desirable property that the feedback controls preserve the marginal laws of $(X,\mu)$.
    \end{remark}
	We next provide some results related to the limit theory,
	that is, the large population control problem converges to the McKean--Vlasov control problem under technical conditions.
	For every $\nu \in \Pc_p(\R^n)$, we denote by $\Pcb^\star_R(\nu)$ the set of optimal relaxed controls
	\[
		\Pcb^\star_R(\nu) 
		:=
		\big\{
			\Pb \in \Pcb_R(\nu) :
			V_R(\nu)=J(\Pb)
		\big\}.
	\]
	Let $(\nu^1,\dots,\nu^N) \in \Pc_p(\R^n)$, $\nu_N:=\nu^1 \otimes \dots \otimes \nu^N$ and $\alpha = (\alpha^1,\dots,\alpha^N) \in \Ac^N_p(\nu_N)$, we define
	\begin{equation} \label{eq:def_PN}
		\P^N(\alpha^1,\dots,\alpha^N)
		:=
		\frac{1}{N} \sum_{i=1}^N\Lc^{\P^N_{\nu}} 
		\big(X^{\alpha,i},Y^{\alpha,i}, \delta_{\alpha^i_t}(\mathrm{d}a)\mathrm{d}t,W^i,B,\overline{\varphi}_{N} \big)
		\in
		\Pc(\Omb),
	\end{equation}
	where 
	$ Y^{\alpha,i}_{\cdot}:=X^{\alpha,i}_{\cdot}-\int_0^\cdot \sigma_0 (s,X^{\alpha,i},\varphi^N_s,\alpha^i_s)\mathrm{d}B_s$
	and
	$\overline{\varphi}_{N}
		:=
		\frac{1}{N} \sum_{i=1}^N \delta_{\big(X^{\alpha,i},Y^{\alpha,i},\delta_{\alpha^i_t}(\mathrm{d}a)\mathrm{d}t, W^i \big)}.
	$

	\begin{theorem} \label{thm:limit}
		Let {\rm \Cref{assum:main1}} and {\rm \Cref{assum:constant_case}} hold true, 
		assume that $A \subset \R^j$ for some $j \ge 1$, and that $L$ and $g$ are continuous in all their arguments.
		With the constants $p$ and $p^\prime$ given in {\rm\Cref{assum:main1}}, 
		let $(\nu^i)_{i \ge 1} \subset \Pc_{p^{\prime}}(\R^n)$ be such that
		$\sup_{N \ge 1}  \frac{1}{N}\sum_{i=1}^N \int_{\R^n}| x|^{p^\prime} \nu^i(\mathrm{d} x) < \infty$.
\medskip
		{\color{black}
		
		$(i)$ Let $\big(\Pb^N\big)_{N \ge 1}$ be given by $\Pb^N := \P^N(\alpha^{N,1}, \dots, \alpha^{N,N})$,
		where
		$(\alpha^{N,1}, \dots, \alpha^{N,N}) \in \Ac^N_p(\nu_N)$ satisfies
		\begin{equation} \label{eq:eps_optimal_ctrl}
			J(\alpha^{N,1}, \dots, \alpha^{N,N}) \ge V_S^N(\nu^1,\dots,\nu^N) - \varepsilon_N,
			~\mbox{\rm for all}~N \ge 1,
		\end{equation}
		for a sequence $(\varepsilon_N)_{N \ge 1} \subset \R_+$ satisfying $\lim_{N\to\infty} \varepsilon_N=0$. Then the sequence $(\Pb^N)_{N \ge 1}$ is relatively compact under $\Wc_p$,
		and for any converging subsequence $\big(\Pb^{N_m}\big)_{m \ge 1}$, we have
		\[
			\lim_{m \to \infty} \Wc_p \bigg( \frac{1}{N_{m}}\sum_{i=1}^{N_{m}} \nu^i , \nu \bigg) = 0,
			~\mbox{\rm for some}\; \nu \in \Pc_p(\R^n),
			~\mbox{\rm and}~
			\lim_{m \to \infty} \Wc_{p} \big( \Pb^{N_{m}}, \Pb^\infty \big)=0,
			~\mbox{\rm for some}
			~\Pb^\infty \in  \Pcb^\star_R(\nu).
		\]
		
		$(ii)$ Assume in addition that $\Wc_p \big(N^{-1}\sum_{i=1}^{N} \nu^i , \nu \big)\underset{N\to\infty}{\longrightarrow}0,$ for some $\nu \in \Pc_p(\R^n)$, and let $\Pb^\star \in \Pcb^\star_R(\nu)$. Then we can construct a sequence $(\Pb^N)_{N \ge 1}$, together with $(\alpha^{N,1}, \dots, \alpha^{N,N})_{N \ge 1}$ satisfying \eqref{eq:eps_optimal_ctrl}, such that $\Wc_{p} \big( \Pb^{N}, \Pb^\star \big) \underset{N\to\infty}{\longrightarrow} 0$.
		
\medskip
		
		$(iii)$ Finally, we have
		\begin{equation} \label{eq:continuity_VNS_VS}
			\lim_{N \to \infty}
			\bigg| V_S^N\big(\nu^1,\dots,\nu^N \big)-V_S\bigg(\frac{1}{N}\sum_{i=1}^N \nu^i \bigg) 
			\bigg|
			=
			0.
		\end{equation}
		}
	\end{theorem}

	\begin{proposition} \label{Proposition:Continuity-Existence}
		Let {\rm \Cref{assum:main1}} and {\rm \Cref{assum:constant_case}} hold true, 
		suppose in addition that $A \subset \R^j$ for some $j \ge 1$, and that $L$ and $g$ are continuous in all their arguments.
		With the constants $p$ and $p^\prime$ given in {\rm\Cref{assum:main1}}, 
		let $(\nu^m)_{m \ge 1} \subset \Pc_{p^{\prime}}(\R^n)$ and  $\nu \in \Pc_{p}(\R^n)$ be such that
		\[
			\sup_{m\ge 1} \int_{\R^n}|x|^{p^{\prime}} \nu^m (\mathrm{d}x) < \infty,
			\;\mbox{\rm and}\;
			\lim_{m \rightarrow\infty} \Wc_p \big( \nu^m, \nu \big) = 0.
		\]
		Then
		\begin{equation} \label{eq:continuity_VS}
			\lim_{m \to \infty} V_S(\nu^m)
			=
			V_S \big( \nu \big).
		\end{equation}
		In particular, the map $V_S:\Pc_{p^{\prime}}(\R^n) \longrightarrow \R$ is continuous.
	\end{proposition}

	\begin{remark}
		$(i)$ As far as we know, the above results are new in the setting with presence of common noise.
		Even without taking into account the common noise,
		our results in {\rm\Cref{thm:limit}} and {\rm\Cref{Proposition:Continuity-Existence}} are also more general than the existing ones.  
		In particular, by taking $\ell=0,$ we recover the most essential results in {\rm\citeauthor*{lacker2017limit} \cite{lacker2017limit}} for the case without common noise.
		But in {\rm \Cref{thm:limit}}, the initial distribution does not need to be convergent, 
		it is only required to have finite moments in a uniform way, and the initial condition for each agents are allowed to have different distributions.
		Moreover, the continuity result of the value function $V_S(\nu)$ in {\rm\Cref{Proposition:Continuity-Existence}}
		requires less technical conditions $($such as the Lipschitz assumptions on $L$ and $g)$ than in {\rm\cite[Lemma 3.3]{pham2016dynamic}}. 

		\vspace{0.5em}
		
		$(ii)$
		{\rm \Cref{thm:limit}} shows that any $\eps_N$--optimal control of the large population stochastic control problem converges towards an optimal control of the McKean-Vlasov stochastic control problem.
		In particular, when there exists a unique strong optimal control of the McKean--Vlasov control problem, any $\eps_N$--optimal control of the large population control problem converges to the optimal control. 
	\end{remark}

\section{Technical proofs} 
\label{sec:proofs}
    
    	We first provide a moment estimate of the solution to the controlled SDEs, which will be repeatedly used in the upcoming proofs.
	This is in fact an easy extension of \citeauthor*{lacker2017limit} \cite[Lemmata 3.1 and 3.3]{lacker2017limit} (which are a succession of application of Gronwall's lemma),
	then for brevity we omit the proof.

	\begin{lemma} \label{lemma:estimates}
		Let {\rm\Cref{assum:main1}} hold true, and $q \ge p$.
		Then there exists a constant $C>0$ such that, for all $N \ge 1$, $(\nu,\nu^1,\dots,\nu^N)\in\big(\Pc_{q}(\R^n)\big)^{N+1}$ and $(\alpha^1,\dots,\alpha^N) \in \Ac^N_q(\nu_N)$,
		\[
			\frac{1}{N}\sum_{i=1}^N
			\E^{\P^N_\nu} \bigg[ 
			\sup_{t \in [0,T]}|X^{\alpha,i}_t|^{q}
			\bigg]  
			\le C 
			\bigg( 1+\int_{\R^n}|x^\prime|^{q}\frac{1}{N}\sum_{i=1}^N\nu^i(\mathrm{d}x^\prime)+ \frac{1}{N}\sum_{i=1}^N\E^{\P^N_\nu}\bigg[ \int_0^T \rho(a_0,\alpha^i_t)^{q}\mathrm{d}t \bigg]
			\bigg),
		\]
		and for each $\Pb \in \Pcb_W(\nu)$ $($or $\Pb \in \Pcb_R(\nu)$ when in addition {\rm\Cref{assum:constant_case}} holds$)$,
		we have
		\[
			\E^{\Pb} \bigg[\sup_{t \in [0,T]}|X_t|^{q}
			\bigg]
			+
			\E^{\Pb} \bigg[\sup_{t \in [0,T]}|Y_t|^{q}
			\bigg]
			\le 
			C 
			\bigg(1+\int_{\R^n}|x^\prime|^{q}\nu(\mathrm{d}x^\prime)+ \E^{\Pb} \bigg[ \iint_{[0,T]\times A} \rho(a_0,a)^{q} \Lambda_t(\mathrm{d}a)\mathrm{d}t \bigg] \bigg).
		\]
	\end{lemma}
	
	\begin{remark}
	{\rm
	        Notice that by the existence result (such as in \cite[Theorem A.3.]{djete2019mckean} for the McKean-Vlasov case), we know that: under Assumption \ref{assum:main1}, for all $(\nu,\nu^1,\dots,\nu^N)\in\big(\Pc_{q}(\R^n)\big)^{N+1}$ and $(\alpha^1,\dots,\alpha^N) \in \Ac^N_q(\nu_N),$ $\E^{\P^N_\nu} \big[ \sup_{t \in [0,T]}|X^{\alpha,i}_t|^{q}
			\big]< \infty$ and for each $\Pb \in \Pcb_W(\nu)$ or $\Pb \in \Pcb_R(\nu),$ $\E^{\Pb} \big[\sup_{t \in [0,T]}|X_t|^{q}
			\big]
			+
			\E^{\Pb} \big[\sup_{t \in [0,T]}|Y_t|^{q}
			\big] < \infty$.
		The results in Lemma \ref{lemma:estimates} provide essentially more precise estimations of these quantities.
	}
	\end{remark}

\subsection{{Proof of Theorem \ref{thm:equivalence}}}

	To prove \Cref{thm:equivalence}, the crucial steps consist in first approximating weak control rules by strong control rules, and then relaxed control rules by weak control rules.
	We will provide the two approximation results in \Cref{subsubsec:strong2weak} and \Cref{subsubsec:weak2relaxed} respectively. The subsequent proof of \Cref{thm:equivalence} will then be the object of \Cref{subsubsec:proofofTheorem-equivalence}.

\subsubsection{Approximating weak control rules by strong control rules}
\label{subsubsec:strong2weak}

	This part is devoted to the approximation of weak control rules by strong controls.
	Let {\rm\Cref{assum:main1}} hold true, $\nu \in \Pc_p(\R^n)$ and $\Pb \in \Pcb_W(\nu).$ 
	From the martingale problem in Definition \ref{def:admissible_ctrl_rule} and by using \citeauthor*{stroock2007multidimensional} \cite[Theorem 4.5.2]{stroock2007multidimensional}, 
	on the filtered probability space $(\Omb,\Fb,\Fcb,\Pb)$, $(W, B)$ are standard Brownian motions, 
	$(W,X_0)$ are independent of $(B,\muh)$,
	and there exists a $\Fb$--predictable $A$--valued process $(\alpha_t)_{t \in [0,T]}$, such that, $\Pb\mbox{--a.s.}$,
        \begin{align*}
		X_t&=
		X_0 
		+ 
		\int_0^t b\big(r,X,\mub_r,\alpha_r\big)  \mathrm{d}r 
		+ 
		\int_0^t \sigma \big(r,X,\mub_r,\alpha_r\big) \mathrm{d}W_r 
		+  
		\int_0^t \sigma_0 \big(r,X,\mub_r,\alpha_r\big) \mathrm{d}B_r,\; t\in[0,T], \\
		Y_t&=X_t-\int_0^t \sigma_0 \big(r,X,\mub_r,\alpha_r\big) \mathrm{d}B_r,\; t\in[0,T],
        \end{align*}
        with $\Lambda_t(\mathrm{d}a)\mathrm{d}t =\Lambda^{\alpha}_t(\mathrm{d}a)\mathrm{d}t :=\delta_{\alpha_t}(\mathrm{d}a) \mathrm{d}t$ and
        \begin{align} \label{eq_H-hypothesis}
            \muh_t=\Lc^\Pb\big(X_{t \wedge \cdot},Y_{t \wedge \cdot},\Lambda^t,W \big|B_{t \wedge \cdot},\muh_{t \wedge \cdot} \big)
            =
            \Lc^\Pb\big(X_{t \wedge \cdot},Y_{t \wedge \cdot},\Lambda^t,W\big|B,\muh \big),
            \;
            \mub_t(\mathrm{d} \xb, \mathrm{d}a)
            :=
            \E^{\hat \mu}\big[\delta_{\Xh_{t \wedge \cdot}}(\mathrm{d} \xb)\Lambdah_t(\mathrm{d}a)\big].
        \end{align}
	Let us take a sequence $\big((t^m_i)_{0\leq i\leq m}\big)_{m \ge 1}$ of partitions of $[0,T]$, with $0=t^m_0<t^m_1<\dots<t^m_m=T$, and such that 
	\[\sup_{0\leq i\leq m-1} |t^m_{i+1}-t^m_{i}| \underset{m \to \infty}{\longrightarrow} 0.\]
	For any integer $m \ge 1$, define for simplicity the map $[0,T]\ni t\longmapsto [t]^m:=\sum_{i=0}^{m-1}t^m_i\mathbf{1}_{[t^m_i,t^m_{i+1})}(t)$, as well as $\eps_m:=t_1^m$.
	Let $W^m_\cdot:=W_{\varepsilon_m \vee \cdot}-W_{\varepsilon_m}$ and $B^m_\cdot:=B_{\varepsilon_m \vee \cdot}-B_{\varepsilon_m}$, 
	we define also two filtrations $\Fb^{m}:=(\Fcb^{m}_t)_{t \in [0,T]}$ and $\Gb^{m}=(\Gcb^{m}_t)_{t \in [0,T]}$ by 
	\[
		\Fcb^{m}_t:= \sigma \big(X_{t \wedge \cdot},Y_{t \wedge \cdot},\Lambda^t,W^{m}_{t \wedge \cdot},B^{m}_{t \wedge \cdot}, \muh_{t \wedge \cdot} \big),
		\; \text{\rm and}\;  
		\Gcb^{m}_t:= \sigma \big(B^{m}_{t \wedge \cdot}, \muh_{t \wedge \cdot} \big),\; t \in [0,T].
	\]
	
	\begin{lemma}[Approximation with piecewise constant controls]
	\label{lemma:strongApp_step1}
		In the filtered probability space $(\Omb,\Fbb,\Fcb,\Pb)$,
		there exists a sequence of $\Fbb$--predictable processes $(\alpha^m)_{m \ge 1}$, and a sequence a $\Fbb$--adapted continuous processes $(X^m)_{m \ge 1}$ such that for any $m\geq 1$
		\begin{equation} \label{firstWeak_app}
			\alpha^m_0 = a_0,
			\;
			\alpha^m_t = \alpha^{m}_{[t]^m},\; \text{\rm on}\; [0,T],
			\;
			\lim_{m\rightarrow\infty} \E^{\Pb} \bigg[\int_0^T \rho(\alpha_t,\alpha^m_t)^p \mathrm{d}t \bigg]=0,\; \mbox{\rm and}\;
			\lim_{m\rightarrow\infty} \E^\Pb \bigg[ \sup_{s \in [0,T]} |X_s-X^m_s|^{p} \bigg]=0,
		\end{equation}
		where for each $m \ge 1$, $(X^m_t)_{t \in [0,T]}$ is the unique strong solution of 
		\begin{equation} \label{eq:weakMV_delay}
			X^m_t
			=
			X_0 
			+
			\int_{\eps_m}^{t \vee \eps_m} b\big(r,X^{m}_{r\wedge \cdot},\mub^m_r,\alpha^m_r\big) \mathrm{d}r
			+
			\int_{\eps_m}^{t \vee \eps_m} \sigma \big(r,X^m_{r\wedge \cdot},\mub^m_r,\alpha^m_r\big) \mathrm{d}W^{m}_r
			+  
			\int_{\eps_m}^{t \vee \eps_m} \sigma_0 \big(r,X^m_{r\wedge \cdot},\mub^m_r,\alpha^m_r\big) \mathrm{d}B^{m}_r,
		\end{equation}
		with $\E^{\Pb} \big[\|X^m\|^p \big]< \infty$ and $\mub^m_t:=\Lc^\Pb \big(X^m_{t \wedge \cdot},\alpha^m_t \big| \Gcb^{m}_t\big)$.
		Moreover, if we denote $\Lambda^m_t(\mathrm{d}a)\mathrm{d}t:=\delta_{\alpha^m_t}(\mathrm{d}a)\mathrm{d}t,$ as well as
		\[
			\muh^m_t:=\Lc^\Pb \big(X^m_{t \wedge \cdot},Y^m_{t \wedge \cdot},(\Lambda^m)^t,W^{m} \big| \Gcb^{m}_t\big)\; \mbox{\rm and}\;
			Y^m_{t}:=X^m_t-\int_{\eps_m}^{t \vee \eps_m} \sigma_0 \big(r,X^m_{r\wedge \cdot},\mub^m_r,\alpha^m_r\big) \mathrm{d}B^{m}_r,\; \mbox{\rm for all}~t \in [0,T],
		\]
		then $(X_0,W^{m})$ is $\Pb$--independent of $(B^{m},\muh^m)$,
		$\muh^m_t= \muh^m_T \circ \big(\Xh_{t \wedge \cdot},\Yh_{t \wedge \cdot}, \Lambdah^t, \Wh \big)^{-1}$, 
		and
		\begin{equation} \label{H-property}
			\muh^m_t
			=
			\Lc^\Pb \big(X^m_{t \wedge \cdot},Y^m_{t \wedge \cdot},(\Lambda^m)^t,W^{m} \big| B^{m}, \muh^m\big),\; \Pb\mbox{\rm --a.s.},\; \mbox{\rm for all}\; t \in [0,T].
		\end{equation}
	\end{lemma}
	
    

	\begin{proof}
	First, we claim that for each $m \ge 1$,
	\begin{align} \label{eq:epsilon_muh}
		\muh_t=\Lc^\Pb \big(X_{t \wedge \cdot},Y_{t \wedge \cdot},\Lambda^t,W \big| B^m_{t \wedge \cdot},\muh_{t \wedge \cdot} \big)=\Lc^\Pb \big(X_{t \wedge \cdot},Y_{t \wedge \cdot},\Lambda^t,W \big| B^m,\muh \big),\;\Pb\mbox{\rm--a.s.},\; \mbox{for all}\;t \in [0,T].
	\end{align}
	Indeed, for all $\phi \in C_b(\Cc^n \x \Cc^n \x \M \x \Cc^d)$ and $\psi \in C_b(\Cc^\ell \x C([0,T],\Pc(\Cc^n \x \Cc^n \x \M \x \Cc^d)))$, it follows by \eqref{eq:muh_property} that
	\[
		\E^\Pb \big[\langle \phi,\muh_t \rangle \psi(B^m,\muh) \big]
		=
		\E^\Pb \big[\big\langle \phi,\Lc^\Pb \big(X_{t \wedge \cdot},Y_{t \wedge \cdot},\Lambda^t,W | B^m,\muh \big)\big \rangle \psi(B^m,\muh) \big].
	\]
	This implies \eqref{eq:epsilon_muh} by arbitrariness of $(\phi, \psi)$.  
   	We further observe that $(\Fb^m,\Gb^m)$ satisfies
	\begin{align} \label{General_H-hypotheses}
		\E^{\Pb}[1_D \big| \Gcb^{m}_t]=\E^{\Pb}[1_D \big| \Gcb^{m}_T],\;\mbox{for all}\;D \in \Fcb^{m}_t\lor \sigma(W^m)\;\mbox{and}\; t \in [0,T].
	\end{align}
	
	Next, as $\E^\Pb \big[\int_0^T \rho(\alpha_t,a_0)^p \mathrm{d}t \big] < \infty$,
	it follows (this is a straightforward extension of, for instance, \citeauthor*{lipster1977statistics} \cite[Lemma 4.4]{lipster1977statistics}) that there exists a sequence of piecewise constant and $\Fbb$--predictable process $\alpha^m$ satisfying the first two properties in \Cref{firstWeak_app}.
	Without loss of generality, let us also set $\alpha^m_T := \alpha^m_{t^m_{m-1}}$.
	
\medskip
	
	Then given $\alpha^m$, let $X^m$ be the unique $\Fb^{m}$--adapted solution of the McKean--Vlasov SDE \eqref{eq:weakMV_delay} 
	(see also \cite[Theorem A.3]{djete2019mckean} for its well--posedness),
	with $\mub^m_t:=\Lc^\Pb \big(X^m_{t \wedge \cdot},\alpha^m_t \big| \Gcb^{m}_t\big)$.
	Let $\muh^m$, $\Lambda^m_t(\mathrm{d}a)\mathrm{d}t$ and $Y^m$ be defined as in the statement of Lemma \ref{lemma:strongApp_step1}.

	\medskip
	
	The independence between $(X_0, W^m)$ and $(B^m, \muh^m)$ follows directly from the independence of $(X_0, W)$ and $\Gcb_T$.
	Further, by \Cref{Proposition:property_WeakControl}, $W$ is a Brownian motion under the conditional law of $\Pb$ knowing $\Gcb_T$.
	It follows that, for each $t \in [0,T]$, 
	$(W^m_{t+s} - W^m_t)_{s \in [0, T-t]} =(W_{(t+s) \vee \eps_m} - W_{t \vee \eps_m})_{s \in [0, T-t]}$ and $(X^m_{t \wedge \cdot},Y^m_{t \wedge \cdot},(\Lambda^m)^t,W^{m}_{t \wedge \cdot})$
	are independent under the conditional law of $\Pb$ knowing $\Gcb_t^m$ (or $\Gcb_T^m$).
	Together with \eqref{General_H-hypotheses}, it follows that   
	\[
		\muh^m_t=\Lc^\Pb \big(X^m_{t \wedge \cdot},Y^m_{t \wedge \cdot},(\Lambda^m)^t,W^{m} \big| \Gcb^{m}_t\big)=\Lc^\Pb \big(X^m_{t \wedge \cdot},Y^m_{t \wedge \cdot},(\Lambda^m)^t,W^{m} \big| \Gcb^{m}_T\big),
		\;\mbox{and}\;
		\mub^m_t(\mathrm{d} \xb,\mathrm{d}a)=\E^{\hat \mu^m}\Big[\delta_{\Xh_{t \wedge \cdot}}(\mathrm{d} \xb)\Lambdah_t(\mathrm{d}a)\Big],
	\]
	and therefore
	\[
		\muh^m_t= \muh^m_T \circ \big(\Xh_{t \wedge \cdot},\Yh_{t \wedge \cdot}, \Lambdah^t, \Wh \big)^{-1},\;\Pb\;\mbox{\rm--a.s.,}\; \mbox{for all}\; t \in [0,T].
	\]
	Since $(\muh^m_t)_{t \in [0,T]}$ is a function of $\muh^m_T$,
	and  $(B^{m},\muh^m)$ and $(X_0, W^{m})$ are $\Pb$--independent, it follows by using the definition of $\Gb^{m}$  that \eqref{H-property} holds true.
	 
	\medskip
	
	To conclude, it is enough to prove that $\lim_{m\rightarrow\infty} \E^{\Pb} \big[ \sup_{s \in [0,T]} |X_s-X^m_s|^{p} \big]=0$.
	For any $t \in[ \eps_m,T]$, one has
	\begin{align*}
		X_t-X^m_t
		=&\
		X_{\eps_m} -X_0
		+ \int_{\eps_m}^{t} \big(b(r,X_{r\wedge\cdot},\mub_r,\alpha_r)-b(r,X^m_{r\wedge\cdot},\mub^m_r,\alpha^m_r)\big) \mathrm{d}r \\
		&+ 
		\int_{\eps_m}^{t} \!\! \big(\sigma(r,X_{r\wedge\cdot},\mub_r,\alpha_r)-\sigma(r,X^m_{r\wedge\cdot},\mub^m_r,\alpha^m_r)\big) \mathrm{d}W_r
		+ 
		\int_{\eps_m}^{t} \big(\sigma_0(r,X_{r\wedge\cdot},\mub_r,\alpha_r)-\sigma_0(r,X^m_{r\wedge\cdot},\mub^m_r,\alpha^m_r)\big)\mathrm{d}B_r.
    \end{align*}
	Next, using Jensen's inequality, Burkholder--Davis--Gundy inequality, the Lipschitz property of $(b,\sigma,\sigma_0)$, 
	and the inequality
	\[
		\Wc_p \big(\mub_t,\mub^m_t \big)^p
		=
		\Wc_p \Big(\Lc^\Pb \big(X_{t \wedge \cdot},\alpha_t \big| \Gcb^{m}_t\big),\Lc^\Pb \big(X^m_{t \wedge \cdot},\alpha^m_t \big| \Gcb^{m}_t\big) \Big)^p
		\le
		\E^{\Pb} \bigg[\sup_{s \in [0,t]}\big|X_s-X^m_s \big|^{p} + \rho \big(\alpha^m_t,\alpha_t \big)^p \Big|\Gcb^{m}_t \bigg],
	\]
	there exists a constant $C > 0$, which may vary from line to line, such that
	\begin{align*}
		\E^{\Pb} \bigg[ \sup_{s \in [\eps_m,t]} |X_s-X^m_s|^{p} \bigg]
		\le&\ C 
		\E^{\Pb} \bigg[  |X_{\eps_m}- X_0|^{p} 
		+
		\int_{\eps_m}^{t} \big|(b,\sigma,\sigma_0)(r,X_{r\wedge\cdot},\mub_r,\alpha_r)-(b,\sigma,\sigma_0)(r,X^m_{r\wedge\cdot},\mub^m_r,\alpha^m_r) \big|^p \mathrm{d}r \bigg]
		\\
		\le&\
		C \bigg(\E^{\Pb} \big[  |X_{\eps_m}- X_0|^{p} \big]
		+
		\E^{\Pb} \bigg[ \int_{\eps_m}^{t} \sup_{u \in [\eps_m,r]}\big|X_u-X^m_u \big|^{p} \mathrm{d}r \bigg] + C_m \bigg),
	\end{align*}
	where
	\begin{align*}
		C_m
		:=
		\E^\Pb \bigg[
			\int_0^T \Big(\big|(b,\sigma,\sigma_0)(r,X,\mub_r,\alpha_r)-(b,\sigma,\sigma_0)(r,X,\mub_r,\alpha^m_r)\big|^p
			+ 
			\rho \big(\alpha^m_r,\alpha_r \big)^p\Big)  \mathrm{d}r
		\bigg].
	\end{align*}
 	By Gronwall's lemma (recall that all expectations appearing here are finite), we deduce that for all $t \in [\eps_m,T]$
	\begin{align*}
		\E^{\Pb} \bigg[ \sup_{s \in [\eps_m,t]} |X_s-X^m_s|^{p} \bigg]
		&\le C \Big(\E^{\Pb} \big[  |X_{\eps_m}-X_0|^{p} \big]
		+ C_m \Big),
	\end{align*}
	so that
	\begin{align*}
		\E^{\Pb} \bigg[ \sup_{s \in [0,T]} |X_s-X^m_s|^{p} \bigg]
		&\le C \bigg(\E^{\Pb} \big[  |X_{\eps_m}-X_0|^{p} \big]
		+
		\E^{\Pb} \bigg[ \sup_{r \in [0,\eps_m]}  |X_r-X_0|^{p} \bigg]
		+ C_m \bigg).
	\end{align*}
	By \Cref{assum:main1}, one has, for all $r \in [0,T]$,
	\begin{align*}
		|(b,\sigma,\sigma_0)(r,X,\mub_r,\alpha_r)-(b,\sigma,\sigma_0)(r,X,\mub_r,\alpha^m_r)\big|^p
		\le &\
		C \Big(\big\| X_{r \wedge \cdot} \big\|^p + \E^{\Pb} \Big[\big\| X_{r \wedge \cdot} \big\|^p + \rho \big( a_0,\alpha_r\big)^p \Big| \Gcb_T \Big] + \rho \big( a_0,\alpha_r\big)^p \Big) \\
		&+ C \rho \big(\alpha^m_r,\alpha_r\big)^p.
	\end{align*}
	By dominated convergence and the continuity of coefficients $(b,\sigma,\sigma_0)$, it follows that for all  $K >0$,
	\begin{align*}
		\lim_{m\rightarrow\infty} \E^\Pb \bigg[\int_{0}^{T} \big|(b,\sigma,\sigma_0)(r,X,\mub_r,\alpha_r)-(b,\sigma,\sigma_0)(r,X,\mub_r,\alpha^m_r)\big|^p \mathbf{1}_{\{\rho(\alpha^m_r,\alpha_r) \le K \} } \mathrm{d}r\bigg]=0.
	\end{align*}
	In addition, since $\big(\big\| X_{r \wedge \cdot} \big\|^p + \rho \big( a_0,\alpha_r\big)^p \big)\mathbf{1}_{ \{\rho(\alpha^m_r,\alpha_r) \ge K \}} \le \big(\big\| X_{r \wedge \cdot} \big\|^p + \rho \big( a_0,\alpha_r\big)^p \big)$, which is $\Pb$--integrable, using the uniform integrability of the sequence $(\alpha^m)_{m \ge 1}$, one obtains that
	\begin{align*}
		& \limsup_{K\rightarrow\infty}\; \limsup_{m\rightarrow\infty} \E^\Pb \bigg[\int_{0}^{T} \big|(b,\sigma,\sigma_0)(r,X,\mub_r,\alpha_r)-(b,\sigma,\sigma_0)(r,X,\mub_r,\alpha^m_r)\big|^p \mathbf{1}_{ \{\rho(\alpha^m_r,\alpha_r) > K \}} \mathrm{d}r\bigg]
		\\
		&\le \limsup_{K\rightarrow\infty}\; \limsup_{m\rightarrow\infty}  K\E^\Pb \bigg[\int_{0}^{T} \Big( \big(\big\| X_{r \wedge \cdot} \big\|^p + \rho \big( a_0,\alpha_r\big)^p \big) + \rho \big(\alpha^m_r,\alpha_r\big)^p \Big) \mathbf{1}_{ \{\rho(\alpha^m_r,\alpha_r) > K \} } \mathrm{d}r\bigg]
		\\
		&\le
		\limsup_{K\rightarrow\infty}\; \sup_{m>0} K\E^\Pb \bigg[\int_{0}^{T} \rho \big(\alpha^m_r,\alpha_r\big)^p \mathbf{1}_{ \{\rho(\alpha^m_r,\alpha_r) > K \} } \mathrm{d}r\bigg]=0.
	\end{align*}
	This implies that $\Lim_{m\rightarrow\infty} C_m=0,$ and hence \eqref{firstWeak_app} does hold.
	\end{proof}

	\begin{lemma} \label{lemma:strongApp_step2}
		In the context of {\rm \Cref{lemma:strongApp_step1}}, let $m \ge 1$. In the $($possibly enlarged$)$ filtered probability space $(\Omb,\Fbb,\Fcb,\Pb)$,
		there exists a sequence of i.i.d. random variables $U^m = (U^m_i)_{i \ge 1}$, with uniform distribution on $[0,1]$, and $\Pb$--independent of $(X_0,B^m,W)$,
		together with a $(\sigma(U^m,X_0,W_{t \wedge \cdot},B^m_{t \wedge \cdot} ))_{t \in [0,T]}$--predictable process 
		$(\widetilde{\gamma}^m_t)_{t \in [0,T]}$, such that if we let $(\widetilde{X}^m_t)_{t \in [0,T]}$ be the unique strong solution of
		\begin{align*}
			\widetilde{X}^m_t
			=
			X_0 
			+
			\int_{\eps_m}^{t \vee \eps_m} b \big(r,\widetilde{X}^{m}_{r\wedge\cdot},\overline{\zeta}^m_r,\widetilde{\gamma}^m_r\big) \mathrm{d}r
			+
			\int_{\eps_m}^{t \vee \eps_m} \sigma \big(r,\widetilde{X}^m_{r\wedge\cdot},\overline{\zeta}^m_r,\widetilde{\gamma}^m_r\big) \mathrm{d}W^m_r
			+  
			\int_{\eps_m}^{t \vee \eps_m} \sigma_0 \big(r,\widetilde{X}^m_{r\wedge\cdot},\overline{\zeta}^m_r,\widetilde{\gamma}^m_r\big) \mathrm{d}B^m_r,
		\end{align*}
		with $\overline{\zeta}^m_t:=\Lc^\P \big(\widetilde{X}^m_{t \wedge \cdot},\widetilde{\gamma}^m_t \big| B^m,U^m\big)$,
		and define further $\widetilde{\Lambda}^m_t(\mathrm{d}a)\mathrm{d}t:=\delta_{\widetilde{\gamma}^m_t}(\mathrm{d}a)\mathrm{d}t,$ as well as
		\begin{align*}
			\widetilde{Y}^m_{t}
			:=
			\widetilde{X}^m_t-\int_{\eps_m}^{t \vee \eps_m} \sigma_0 \big(r,\widetilde{X}^m_{r\wedge \cdot},\overline{\zeta}^m_r,\widetilde{\gamma}^m_r\big) \mathrm{d}B^{m}_r,
			\; \mbox{\rm and}\; 
			\widehat{\zeta}^m_t
			:=
			\Lc^\P \big(\widetilde{X}^m_{t \wedge \cdot},\widetilde{Y}^m_{t \wedge \cdot},(\widetilde{\Lambda}^m)^t,W^{m} \big| B^m,U^m\big),			
		\end{align*}
		then, with $(X^m,Y^m, \Lambda^m, W^m,B^m, \muh^m )$ defined in {\rm \Cref{lemma:strongApp_step1}}, we have
		\begin{equation} \label{eq:X2Xh}
			\Lc^\Pb \Big(\widetilde{X}^m,\widetilde{Y}^m, \widetilde{\Lambda}^m, W^m,B^m, \widehat{\zeta}^m_T \Big)
			=
			\Lc^\Pb \Big(X^m,Y^m, \Lambda^m, W^m,B^m, \muh^m \Big).
		\end{equation}

		Finally, when $\ell=0$ and $\muh^m$ is deterministic, then one can take $(\widetilde{\gamma}^m_t)_{t \in [0,T]}$ to be $(\sigma(X_0,W_{t \wedge \cdot}))_{t \in [0,T]}$--predictable.
	\end{lemma}

	\begin{proof}
	Let us fix $m \ge 1$, and introduce $\{ W^m \}_0=\{ B^m \}_0:=0$, and then for $i \in \{1,\dots,m \}$,
	\begin{align*}
		\{ B^m \}_i:=\big(B^{m,(k-1)}\big)_{1 \le k \le i},\;\{ W^m \}_i:=\big(W^{m,(k-1)}\big)_{1 \le k \le i}\;, \;\{ \muh^m \}_i:=\big(\muh^m_{t^m_{k}}\big)_{0 \le k \le i},\;\mbox{and}\;\{ \alpha^m \}_i:=(\alpha^m_k)_{0 \le k \le i},
	\end{align*}
	where $B^{m,(k-1)}_t:=B^m_{(t \vee t^m_{k-1}) \wedge t^m_{k}}-B^m_{t^m_{k-1}}$ and $W^{m,(k-1)}_t:=W^m_{(t \vee t^m_{k-1}) \wedge t^m_{k}}-W^m_{t^m_{k-1}}$, $t \in [0,T].$
    
    \medskip
	$\underline{Step\;1.}$
	For each $i \in \{1,\dots,m \}$, 
	there exists (see \citeauthor*{kurtz2014weak} \cite[Lemma 1.3.]{kurtz2014weak}) a Borel measurable function 
	$G^{\mu}_i:\Cc^\ell \x \Pc(\Cc^n\x \Cc^n \x \M \x \Cc^d)^{i} \x [0,1] \longrightarrow \Pc(\Cc^n\x \Cc^n \x \M \x \Cc^d)$
	such that, for any uniform random variable $U^m_i$ independent of $\big( \{B^m\}_{i}, \{\muh^m \}_{i-1} \big)$ and $\widehat{G}^{\mu}_i :=G^{\mu}_i \big( \{B^m\}_{i}, \{\muh^m \}_{i-1},  U^m_i \big)$, we have
	\begin{align} \label{identity_law}
		\Lc^\Pb \big(X_0,\{B^m\}_{i}, \{W^m\}_{i}, \{\muh^m \}_{i-1}, \muh^m_{t^m_i}\big)
		=
		\Lc^\Pb \big(X_0,\{B^m\}_{i}, \{W^m\}_{i}, \{\muh^m \}_{i-1}, \widehat{G}^{\mu}_i\big),
	\end{align}
	Above, $G^{\mu}_i$ is a function of $(\{B^m\}_{i}, \{\muh^m \}_{i-1},  U^m_i )$ rather than of 
	$(X_0, \{B^m\}_{i}, \{W^m\}_{i}, \{\muh^m \}_{i-1},  U^m_i)$, since $\muh^m_{t^m_i}$ is actually $\Pb$--independent of $(X_0, W^m)$.
	We can apply a similar argument to find a Borel measurable function 
	$G^{\alpha}_i : \R^n \x (\Cc^\ell \x \Cc^d)^{i} \x \Pc(\Cc^n \x \Cc^n \x \M \x \Cc^d)^{(i+1)} \x A^{i} \x \R \longrightarrow A ,$ and  uniform random variable $V^m_i$ independent of the variables $\big( X_0,\{ B^m \}_i , \{ W^m \}_i, \{ \muh^m \}_i, \{ \alpha^m \}_{(i-1)} \big)$
	such that
	\begin{align} \label{identity_control}
		\Lc^\Pb \big(X_0,\{ B^m \}_i , \{ W^m \}_i, \{ \muh^m \}_i, \{ \alpha^m \}_{(i-1)},  \alpha^m_i \big)
		=
		\Lc^\Pb \big(X_0,\{ B^m \}_i , \{ W^m \}_i, \{ \muh^m \}_i, \{ \alpha^m \}_{(i-1)},  \widetilde{G}^\alpha_i \big),
	\end{align}
	where
	\[
		\widetilde{G}^\alpha_i := G^\alpha_i \Big( X_0,\{ B^m \}_i , \{ W^m \}_i,\{ \muh^m \}_i, \{ \alpha^m \}_{(i-1)}, V^m_i\Big).
	\]
	Observe that  one can take $(U^m_1,\dots,U^m_m)$ to be independent of $(V^m_1,\dots,V^m_m).$ 
	We can then find a Borel function $\kappa^d: \R^d \longrightarrow [0,1]$ such that $\Lc^{\Pb} \big( \kappa^d \big(W_{i\eps_m/m}-W_{(i-1)\eps_m/m}\big) \big)$ is a uniform distribution.
	Define next $\widetilde{\gamma}^m_0:=\alpha^m_0 = a_0$,
	$\widehat{\zeta}_0:=\muh_{t^m_0 \wedge \cdot},$ and for any $i\in\{1,\dots,m-1\}$
	\[
		\zetah^{m}_i := G^{\mu}_i \Big( X_0,\{B^m\}_{i}, \{\zetah^m \}_{(i-1)}, U^m_i \Big),
		\;
		\widetilde{\gamma}^m_i
		:=
		G^\alpha_i \Big( X_0, \{B^m\}_{i}, \{W^m\}_{i}, \{\zetah^m\}_{i}, \{\widetilde{\gamma}^m\}_{(i-1)}, \kappa^d \big(W_{i\eps_m/m}-W_{(i-1)\eps_m/m}\big) \Big).
	\]
	Then, for each $i\in\{0,\dots,m\}$, $\zetah^m_i$ is $\sigma (U^m_1,\dots,U^m_i, \{B^m\}_{i} )$--measurable, 
	and
	\begin{equation} \label{measurability_control}
		\widetilde{\gamma}^m_i
		\; \mbox{is}\;
		\sigma \big(X_0,W_{\eps_m \wedge \cdot}, \{W^{m}\}_i,\{\zetah^m\}_{i},\{B^{m}\}_i \big)\mbox{\rm--measurable}.
	\end{equation}    
	
	\medskip
	When $\ell=0$ and $\muh^m$ is deterministic, 
	the previous construction implies that $\{\zetah^m\}_m=\{\muh^m\}_m$ is deterministic and $\widetilde{\gamma}^m_i$ is $\sigma \big(X_0,W_{\eps_m \wedge \cdot}, \{W^{m}\}_i \big)$--measurable.
    
    \medskip
	{$\underline{Step\;2.}$}
	We next prove by induction that, for each $i \in \{0,\dots,m \}$
	\begin{align} \label{general_LawEquality}
		\Lc^\Pb \big(X_0, \{B^m\}_{i}, \{W^m\}_{i}, \{\muh^m\}_{i}, \{\alpha^m\}_{i} \big)
		=
		\Lc^\Pb \big(X_0, \{B^m\}_{i}, \{W^m\}_{i}, \{\zetah^m\}_{i}, \{\widetilde{\gamma}^m\}_{i} \big).
	\end{align}
	When $i=0$, \eqref{general_LawEquality} holds true since $\alpha^m_0$ and $\muh_{t^m_0}$ are deterministic constants. 

	\vspace{0.5em}
	
	Now, assume that \eqref{general_LawEquality} is true for some $i\in\{0,\dots,m-1\}$.
	First, take $\phi \in C_b(\R^n \x \Cc^n \x \Cc^\ell \x \Cc^d \x \Pc(\Cc^n \x \Cc^n \x \M \x \Cc^n)^i \x A^i),$ $\psi \in C_b(\Cc^d)$, 
	$\varphi \in C_b(\Cc^\ell)$ and $h \in C_b(\Pc(\Cc^n \x \Cc^n \x \M \x \Cc^d))$. Using the independence of the increments of the Brownian motion $W^m$, together with \eqref{H-property} and \eqref{identity_law}, we have
	\begin{align} \label{eq:mu2G}
		&\ \E^\Pb \Big[ \phi \big(X_0, B^m_{t^m_i \wedge \cdot},  W^m_{t^m_i \wedge \cdot}, \{\muh^m\}_{i}, \{\alpha^m\}_{i}  \big) \psi \big( W^{m,(i)} \big) \varphi \big( B^{m,(i)} \big) h \big(\muh^m_{t^m_{i+1}} \big) \Big] 
		\nonumber \\
		=&\
		\E^\Pb \Big[ \E^\Pb \Big[ \psi \big( W^{m,(i)} \big) \Big]\E^\Pb \Big[ \phi \big(X_0,B^m_{t^m_i \wedge \cdot}, W^m_{t^m_i \wedge \cdot},\{\muh^m\}_{i}, \{\alpha^m\}_{i}  \big) \Big| B^{m}_{t^m_i \wedge \cdot},\muh^m_{t^m_i \wedge \cdot} \Big] \varphi \big( B^{m,(i)} \big) h \big(\muh^m_{t^m_{i+1}} \big) \Big] 
		\nonumber \\
		=&\
		\E^\Pb \Big[ \E^\Pb \Big[ \psi \big( W^{m,(i)} \big) \Big]\E^\Pb \Big[ \phi \big(X_0,B^m_{t^m_i \wedge \cdot}, W^m_{t^m_i \wedge \cdot},\{\muh^m\}_{i}, \{\alpha^m\}_{i} \big) \Big| B^{m},\muh^m \Big] \varphi \big( B^{m,(i)} \big) h \big(\widehat{G}^\mu_{i+1} \big) \Big] 
		\nonumber \\
		=&\
		\E^\Pb \Big[ \E^\Pb \Big[ \phi \big(X_0,B^m_{t^m_i \wedge \cdot}, W^m_{t^m_i \wedge \cdot}, \{\muh^m\}_{i},  \{\alpha^m\}_{i}  \big) \psi \big( W^{m,(i)} \big) \Big| B^{m},\muh^m \Big] \varphi \big( B^{m,(i)} \big) h \big(\widehat{G}^\mu_{i+1} \big) \Big]. 
	\end{align}
	
	Further, let $\varphi_1 \in C_b(\Cc^\ell \x \Pc(\Cc^n \x \Cc^n \x \M \x \Cc^d)^{i})$ and $\varphi_2 \in C_b([0,1])$,
	using the independence of $U^m_{i+1}$ and that of the increments of the Brownian motions $(B^m, W^m)$,
	and the induction assumption, we obtain
	\begin{align} \label{eq:H_Hypo_deriv}
		&\
		\E^\Pb \Big[ \E^\Pb \Big[ \phi \big(X_0, B^m_{t^m_i \wedge \cdot}, W^m_{t^m_i \wedge \cdot}, \{\muh^m\}_{i}, \{\alpha^m\}_{i}  \big) \psi \big( W^{m,(i)} \big) \Big| B^{m},\muh^m \Big] \varphi_1 \big(B^m_{t^m_{i+1} \wedge \cdot}, \{\muh^m\}_{i}  \big) \varphi_2 (U^m_{i+1}) \Big]
		\nonumber \\
		=&\
		\E^\Pb \Big[ \phi \big(X_0, B^m_{t^m_i \wedge \cdot}, W^m_{t^m_i \wedge \cdot}, \{\muh^m\}_{i}, \{\alpha^m\}_{i}  \big) \psi \big( W^{m,(i)} \big) \varphi_1 \big(B^m_{t^m_{i+1} \wedge \cdot}, \{\muh^m\}_{i}  \big) \Big] \E^\Pb \Big[ \varphi_2 (U^m_{i+1}) \Big]
		\nonumber \\
		=&\
		\E^\Pb \Big[ \phi \big(X_0, B^m_{t^m_i \wedge \cdot}, W^m_{t^m_i \wedge \cdot}, \{\zetah^m\}_{i}, \{\widetilde{\gamma}^m\}_{i}  \big) \psi \big( W^{m,(i)} \big) \varphi_1 \big(B^m_{t^m_{i+1} \wedge \cdot}, \{\zetah^m\}_{i}  \big) \varphi_2 (U^m_{i+1}  ) \Big].
	\end{align}
	Using the arbitrariness of $(\varphi_1,\varphi_2)$,
	and a classical density argument, we can replace 
	$\varphi_1 \big(B^m_{t^m_{i+1} \wedge \cdot}, \{\zetah^m\}_{i}  \big) \varphi_2 (U^m_{i+1}  ) $
	by $\varphi(B^{m,(i)})h \big(G^{\mu}_{i+1}\big(B^m_{t^m_{i+1} \wedge \cdot}, \{\muh^m\}_{i}  ,U^m_{i+1} \big) \big)$, for arbitrary continuous and bounded functions $\varphi$ and $h$,
	in \eqref{eq:H_Hypo_deriv}, and it leads to
	\begin{align*} 
		\eqref{eq:mu2G}
		=
		\E^\Pb \Big[ \phi \big(X_0, B^m_{t^m_i \wedge \cdot}, W^m_{t^m_i \wedge \cdot},\{\zetah^m\}_{i}, \{\widetilde{\gamma}^m\}_{i}  \big) \psi \big( W^{m,(i)} \big) \varphi(B^{m,(i)})h \big(\zetah^m_{i+1} \big) \Big],
	\end{align*}
	and hence 
	\[
		\Lc^\Pb  \Big( X_0, B^m_{t^m_{i+1} \wedge \cdot}, W^m_{t^m_{i+1} \wedge \cdot}, \{\muh^m\}_{(i+1)}, \{\alpha^m\}_{i} \Big)
		=
		\Lc^\Pb  \Big( X_0, B^m_{t^m_{i+1} \wedge \cdot}, W^m_{t^m_{i+1} \wedge \cdot}, \{\zetah^m\}_{(i+1)}, \{\widetilde{\gamma}^m\}_{i} \Big).
	\]
	Together with the result \eqref{identity_control}, and by the independence of $V^m_{i+1}$ w.r.t. the other variables, it follows that
	\[
		\Lc^\Pb \Big(X_0, \{ B^m \}_{(i+1)}, \{ W^m \}_{(i+1)}, \{\muh^m\}_{(i+1)}, \{\alpha^m\}_{i}, \alpha^m_{i+1}\Big)
		=
		\Lc^\Pb \Big(X_0, \{ B^m \}_{(i+1)}, \{ W^m \}_{(i+1)}, \{ \zetah^m\}_{(i+1)}, \{ \widetilde{\gamma}^m \}_{i}, \widetilde{\gamma}^m_{i+1} \Big),
	\]
	which concludes the proof of \eqref{general_LawEquality} by induction.
    
	\medskip
	
	{$\underline{Step\;3.}$}
	Under \Cref{assum:main1},
	the solution of SDE \eqref{eq:weakMV_delay} can be expressed as function of $(X_0,W^m, B^m, (\Lambda^m), \muh^m$).
	More precisely, there exists a Borel function $H^m: [0,T] \x \R^n \x \Cc^d \x \Cc^\ell \x \M \x \Pc(\Cc^n \x \Cc^n \x \M \x \Cc^d) \longrightarrow \Cc^n \x \Cc^n$ such that 
	\[
		(X^m_t,Y^m_t)
		=
		H^m_t \Big(X_0,W^m_{t \wedge \cdot}, B^m_{t \wedge \cdot}, (\Lambda^m)^t, \muh^m_T \Big),
		\; t\in[0,T],\; \Pb\mbox{\rm--a.s.}
	\]
	Moreover, by Lemma \ref{lemma:strongApp_step1},
	the processes $(\muh^m_t)_{t \in [0,T]}$ and $(\mub^m_t)_{t \in [0,T]}$ are actually functions of $\muh^m_{T}.$

	\medskip

	Define $\widetilde{\gamma}^m_t:=\widetilde{\gamma}^m_i$ for $t \in [t^m_i,t^m_{i+1}),$ $i\in\{0,\dots,m-1\}$, 
	$\widetilde{\Lambda}^m_t(\mathrm{d}a)\mathrm{d}t:=\delta_{{\widetilde{\gamma}}^m_t}(\mathrm{d}a)\mathrm{d}t,$ 
	and
	\[
		\zetah^m_t:=\zetah^m_m \circ \big(\Xh_{t \wedge \cdot},\Yh_{t \wedge \cdot}, \Lambdah^t, \Wh \big)^{-1},
		~\mbox{and}~
		\overline{\zeta}^m_t(\mathrm{d} \xb,\mathrm{d}a):=\E^{\widehat \zeta^m_m}\Big[\delta_{\Xh_{t \wedge \cdot}}(\mathrm{d} \xb)\Lambdah_t(\mathrm{d}a)\Big],\;\mbox{for all}\;t \in [0,T],
	\]
	and then
	\[
		(\widetilde{X}^m_t,\widetilde{Y}^m_t)
		:=
		H^m_t\big(X_0,W^m_{t \wedge \cdot}, B^m_{t \wedge \cdot}, (\Lambdat^m)^t, \zetah^m_{T} \big).
	\]
	It follows from \Cref{general_LawEquality} that \eqref{eq:X2Xh} holds true,
	and $\big(\Xt^m,\Yt^m \big)$ satisfies the SDE in the statement of Lemma \eqref{lemma:strongApp_step2}. It remains to prove that
	\begin{equation} \label{eq:repre_zeta}
		\widehat{\zeta}^m_t
		=
		\Lc^\Pb \big(\Xt^m_{t \wedge \cdot},\Yt^m_{t \wedge \cdot},(\Lambdat^m)^t,W^{m} \big| B^m,U^m\big),
		\Pb\mbox{\rm--a.s., for all}\; t \in [0,T].
	\end{equation}
	Recall that $\muh^m_t=\muh^m_T \circ \big( \Xh_{t \wedge \cdot},\Yh_{t \wedge \cdot},\Lambdah^t ,\Wh  \big)^{-1}$ for all $t \in [0, T]$.
	Let  $\phi \in C_b(\Pc(\Cc^n \x \Cc^n \x \M \x \Cc^d))$, 
	$\varphi \in C_b(\Cc^\ell \x C([0,T];\Pc(\Cc^n \x \Cc^n \x \M \x \Cc^d)))$. By \Cref{general_LawEquality}, we have
	\begin{align*}
     		&\
		\E^\Pb \big[ \langle \phi,\zetah^m_t \rangle \varphi \big(B^m,\zetah^m_{T \wedge \cdot} \big) \big]
		=
		\E^\Pb \big[ \langle f,\muh^m_t \rangle \varphi \big(B^m,\muh^m_{T \wedge \cdot} \big) \big]   
		=
		\E^\Pb \big[ f \big( X^m_{t \wedge \cdot},Y^m_{t \wedge \cdot}, (\Lambda^m)^t,W^m \big) \varphi \big(B^m,\muh^m_{T \wedge \cdot} \big) \big] \\
		=&\
		\E^\Pb \big[ f \big( \widetilde{X}^m_{t \wedge \cdot},\widetilde{Y}^m_{t \wedge \cdot}, (\widetilde{\Lambda}^m)^t,W^m \big) \varphi \big(B^m,\zetah^m_{T \wedge \cdot} \big) \big]
		=
		\E^\Pb \big[ \langle f,\Lc^\P \big(\widetilde{X}^m_{t \wedge \cdot},\widetilde{Y}^m_{t \wedge \cdot}, (\widetilde{\Lambda}^m)^t,W^m \big| B^{m}, \zetah^m_{T \wedge \cdot} \big) \rangle \varphi \big(B^m,\zetah^m_{T \wedge \cdot} \big) \big].
	\end{align*}
	This implies that 
	\[
		\zetah^m_{t}
		=
		\Lc^\Pb \big(\widetilde{X}^m_{t \wedge \cdot},\widetilde{Y}^m_{t \wedge \cdot}, (\widetilde{\Lambda}^m)^t,W^m \big| B^{m}, \zetah^m_{T \wedge \cdot} \big).
	\]
	Recall from \eqref{measurability_control} that
	$\widetilde{\gamma}^m_i$ is $\sigma(X_0,W_{\eps_m \wedge \cdot},\{W^m \}_{i}, \{\zetah^m\}_{i},\{B^m \}_{i} )$--measurable,
	$\widehat{\zeta}^{m}_{i}$ is $\sigma(\{B^m \}_{i} , U^m)$--measurable for each $i\in\{0,\dots,m-1\}$, 
	and $U^m$ is independent of $(X_0, B^m, W_{\eps_m \wedge \cdot}, W^m)$ under $\Pb$.
	It follows that \eqref{eq:repre_zeta} holds true.
\end{proof}

	For Proposition \ref{prop:approximation} below, let us 
	denote by $(\alpha_t)_{t \in [0,T]}$ an $A$--valued $\Fb$--predictable process on the canonical space $\Omb$,
	satisfying that $\Lambda_t(\mathrm{d}a)\mathrm{d}t=\delta_{\alpha_t}(\mathrm{d}a) \mathrm{d}t,$ $\overline{\P}$--a.e., for all $\Pb \in \Pcb_W(\nu)$.

	\begin{proposition} \label{prop:approximation}
		Let {\rm\Cref{assum:main1}} hold true, $\nu \in \Pc_p(\R^n)$ and $\Pb \in \Pcb_W(\nu)$. 
		
		\medskip
		$(i)$ When $\ell \neq 0$, there exists a sequence $(\Pb^m)_{m \ge 1} \subset \Pcb_S(\nu)$ such that
		\begin{align} \label{eq:weak_app1}
			\lim_{m \rightarrow\infty} 
			\Lc^{\Pb^m} \big(X,Y,  \Lambda, W,B,\muh,\delta_{(\mub_t,\alpha_t)}(\mathrm{d} \nub, \mathrm{d}a)\mathrm{d}t \big)
			=
			\Lc^{\Pb} \big(X,Y,  \Lambda, W,B,\muh,\delta_{(\mub_t,\alpha_t)}(\mathrm{d} \nub, \mathrm{d}a)\mathrm{d}t \big),\;
			\mbox{\rm in}
			\; \Wc_p.
		\end{align}
		    
		$(ii)$ When $\ell=0,$ there exists a family $(\Pb^m_u)_{u \in [0,1], m \ge 1} \subset \Pcb_S(\nu)$,
		such that $u \longmapsto \Pb^m_u$ is  Borel measurable, and
		\begin{align} \label{eq:weak_app0}
			\lim_{m \rightarrow\infty} \int_0^1 \Lc^{\Pb^m_u} \big(X,Y,  \Lambda, W,B,\muh,\delta_{(\mub_t,\alpha_t)}(\mathrm{d} \nub,\mathrm{d}a)\mathrm{d}t \big) \mathrm{d}u
			=
			\Lc^{\Pb} \Big(X,Y,  \Lambda, W,B,\muh,\delta_{(\mub_t,\alpha_t)}(\mathrm{d} \nub,\mathrm{d}a)\mathrm{d}t \Big),\;
			\mbox{\rm in}
			\; \Wc_p.
		\end{align}
	\end{proposition}
	\begin{proof}
		First, let $(\widetilde{X}^m,\widetilde{Y}^m,B^{m},W^{m},\zetah^m, \overline{\zeta}^m,  \widetilde \gamma^m, \widetilde \Lambda^m )$ be given as in \Cref{lemma:strongApp_step2}. Using \Cref{lemma:strongApp_step1} and \Cref{lemma:strongApp_step2}, we have
		\[
			\lim_{m \to \infty} \Lc^{\Pb}  \Big(\widetilde{X}^m,\widetilde{Y}^m,B^{m},W^{m},\zetah^m_T,\delta_{(\overline{\zeta}^m_t,\widetilde{\gamma}^{m}_t)}(\mathrm{d} \nub,\mathrm{d}a)\mathrm{d}t \Big)
			=
			\Lc^{\Pb}  \Big(X,Y,B,W,\muh,\delta_{(\mub_t,\alpha_t)}(\mathrm{d} \nub,\mathrm{d}a)\mathrm{d}t \Big),
			\;\mbox{\rm in}\;
			\Wc_p.
		\]

	$(i)$ When $\ell \neq 0$, since $B_{\varepsilon_m}$ is independent of $(X_0,W,B^m)$,
	one can take $U^m := \kappa \big(B_{\varepsilon_m} \big)$ for some measurable function $\kappa: \R \longrightarrow [0,1]^m$.
	Consequently, we have
	$\widehat{\zeta}^m_t=\Lc^\Pb \big(\widetilde{X}^m_{t \wedge \cdot},\widetilde{Y}^m_{t \wedge \cdot},(\widetilde{\Lambda}^m)^t,W^{m} \big| B\big),$ $\Pb$--a.s., for all $t \in [0,T]$.
	Let us then define $(\widetilde{S}^{m}_t)_{t \in [0,T]}$ as the unique strong solution of
	\[
		\widetilde{S}^{m}_t
		=
		X_0 
		+
		\int_0^t b\big(r, \widetilde{S}^{m},\overline{\beta}^{m}_r,\widetilde{\gamma}^m_r\big) \mathrm{d}r
		+
		\int_0^t \sigma \big(r, \widetilde{S}^{m},\overline{\beta}^{m}_r,\widetilde{\gamma}^m_r\big)\mathrm{d}W_r
		+  
		\int_0^t \sigma_0\big(r, \widetilde{S}^{m},\overline{\beta}^{m}_r,\widetilde{\gamma}^m_r\big) \mathrm{d}B_r, 
	\]
	with $ \overline{\beta}^{m}_t:= \Lc^\Pb \big(\widetilde{S}^{m}_{t \wedge \cdot}, \widetilde{\gamma}^m_t |B_{t \wedge \cdot}\big) = \Lc^\Pb \big(\widetilde{S}^{m}_{t \wedge \cdot}, \widetilde{\gamma}^m_t |B\big)$.
	Denote, for all $t \in [0,T]$
	\[
		\widetilde{Z}^{m}_t := \widetilde{S}^{m}_t-\int_0^t \sigma_0 \big(r,\widetilde{S}^{m},\overline{\beta}^{m}_r,\widetilde{\gamma}^m_r\big) \mathrm{d}B_r,
		\;\mbox{and}\; 
		\widehat{\beta}^{m}_t := \Lc^{\Pb} \big( \widetilde{S}^{m}_{t \wedge \cdot},\widetilde{Z}^{m}_{t \wedge \cdot}, (\widetilde \Lambda^m)^t, W \big| B \big). 
	\]
	Using almost the same arguments as in the proof of \eqref{firstWeak_app} in \Cref{lemma:strongApp_step1}, 
	we can deduce that \[
	\lim_{m\rightarrow\infty} \E^\Pb \bigg[\sup_{t \in [0,T]}|\widetilde{S}^{m}_t- \widetilde{X}^m_t|^p\bigg]=0,\]
	and moreover
	\begin{align} \label{eq:convergence_WeakStrong}
		&\
		\lim_{m\rightarrow\infty}
		\Lc^\Pb \Big(\widetilde{S}^{m}, \widetilde {Z}^{m}, B, W, \widehat{\beta}^{m}_T, \delta_{(\overline{\beta}^{m}_t,\widetilde{\gamma}^m_t)}(\mathrm{d} \nub,\mathrm{d}a)\mathrm{d}t  \Big)
		=
		\lim_{m\rightarrow\infty} 
		\Lc^\Pb \Big(\widetilde{X}^m,\widetilde{Y}^m, B^m, W^m,\zetah^m_T, \delta_{(\overline{\zeta}^m_t,\widetilde{\gamma}^m_t)}(\mathrm{d} \nub,\mathrm{d}a)\mathrm{d}t  \Big) \nonumber
		\\
		=&\
		\lim_{m\rightarrow\infty} 
		\Lc^\Pb \Big(X^m,Y^m, B^m, W^m,\muh^m_T, \delta_{(\mub^m_t,\alpha^m_t)}(\mathrm{d} \nub,\mathrm{d}a)\mathrm{d}t  \Big)
		=
		\Lc^\Pb \Big(X,Y, B, W,\muh, \delta_{(\mub_t,\alpha_t)}(\mathrm{d} \nub,\mathrm{d}a)\mathrm{d}t  \Big),
		\; \mbox{in}\; \Wc_p.
	\end{align}
	Then it is enough to denote $\Pb^m := \Pb \circ( \widetilde{S}^{m}, \widetilde {Z}^{m},  \widetilde \Lambda^m, W, B, \widehat{\beta}^m_T )^{-1}$ to conclude the proof of $(i)$.
	
\medskip
	
	$(ii)$. When $\ell=0$, so that the process $B$ disappears,
	one has
	$\widehat{\zeta}^m_t=\Lc^\Pb \big(\widetilde{X}^m_{t \wedge \cdot},\widetilde{Y}^m_{t \wedge \cdot},(\widetilde{\Lambda}^m)^t,W^{m} \big| U^m\big),$ $t \in [0,T],$ $\Pb$--a.s.,
	where $U^m$ is independent of $\big(\widetilde{X}^m, \widetilde{Y}^m, \widetilde{\Lambda}^m, W\big)$.
	Let us define $(\widetilde{S}^{m}_t)_{t \in [0,T]}$ as the unique strong solution of
	\[
		\widetilde{S}^{m}_t
		=
		X_0 
		+
		\int_0^t b\big(r,\widetilde{S}^{m},\overline{\beta}^{m}_r, \widetilde{\gamma}^m_r\big) \mathrm{d}r
		+
		\int_0^t \sigma \big(r,\widetilde{S}^{m},\overline{\beta}^{m}_r, \widetilde{\gamma}^m_r\big)\mathrm{d}W_r, 
	\]
	with
	\[
		\overline{\beta}^{m}_t
		:=
		\Lc^\Pb \big(\widetilde{S}^m_{t \wedge \cdot}, \widetilde{\gamma}^m_t |U^m\big)
		=
		\Lc^\Pb \big(\widetilde{S}^m_{t \wedge \cdot}, \widetilde{\gamma}^m_t |U^m\big),
		\;
		\widetilde{Z}^{m}_t := \widetilde{S}^{m}_t,
		\; \mbox{and}\;
		\widehat{\beta}^{m}_t := \Lc^\Pb \big(\widetilde{S}^{m}_{t \wedge \cdot}, \widetilde{Z}^{m}_{t \wedge \cdot}, (\widetilde{\Lambda}^m)^t, W \big| U^m \big),
	\]
	As in $(i)$, we can apply almost the same arguments as in the proof of \Cref{lemma:strongApp_step1} to deduce that
	\[
		\lim_{m\rightarrow\infty}
		\Lc^\Pb \Big(\widetilde{S}^{m}, \widetilde {Z}^{m}, B, W, \widehat{\beta}^{m}_T, \delta_{(\overline{\beta}^{m}_t,\widetilde{\gamma}^m_t)}(\mathrm{d} \nub,\mathrm{d}a)\mathrm{d}t  \Big)
		=
		\Lc^\Pb \Big(X,Y, B, W,\muh, \delta_{(\mub_t,\alpha_t)}(\mathrm{d} \nub,\mathrm{d}a)\mathrm{d}t  \Big),
		\; \mbox{in}\;\Wc_p.
	\]
	
	Beside, it is easy to check that
	\[
		\Lc^\Pb \Big( \widetilde{S}^m, \widetilde{Z}^m, \widetilde{\Lambda}^m,W,B,\widehat{\beta}^m   \Big| U^m\Big) \in \Pcb_S(\nu),  ~\Pb\mbox{--a.s.},
	\]
	which concludes the proof of $(ii)$.
	\end{proof}
	
	\begin{remark}
		 When $\ell =0$, if we assume in addition that $\muh$ is deterministic under $\Pb \in \Pcb_W(\nu)$,
		 we can omit the term $U^m$ in the proof of {\rm\Cref{prop:approximation}.$(ii)$} by {\rm\Cref{lemma:strongApp_step2}},
		 and hence there is no need to consider the conditional law of $(\widetilde{S}^m, \widetilde{Z}^m, \widetilde{\Lambda}^m,W,B,\widehat{\beta}^m)$ knowing $U^m$.
		 It follows that we can find a sequence $(\Pb^m)_{m \ge 1} \subset \Pcb_S(\nu)$ such that \eqref{eq:weak_app1} holds.
	\end{remark}


 	\begin{remark} \label{rem:Gap_StrongWeak}
		In summary, our proof for approximating weak control by strong control rules consists in three main steps:
		
		\begin{itemize}
		\item[$(i)$] approximate the $($weak$)$ control process by piecewise constant processes and freeze the controlled process on $[0, \eps];$
		
		\item[$(ii)$] represent the piecewise constant control process as functionals of the Brownian motions and some independent randomness using the $(H)$--hypothesis type condition \eqref{eq:H_Hypothesis}$;$

		\item[$(iii)$] replace the independent randomness by the increment of the Brownian motions on $[0,\eps]$, so that the control processes becomes functionals of the Brownian motions only.
		\end{itemize}
		This is quite different from the steps in {\rm \citeauthor*{lacker2017limit} \cite{lacker2017limit}} for McKean--Vlasov control problem without common noise,
		and in spirit closer to the technical steps in {\rm \citeauthor*{karoui2013capacities2} \cite[Theorem 4.5]{karoui2013capacities2}}, which approximates weak control rule by strong control rules for classical stochastic control problems.
		In particular, our approach allows to avoid a subtle gap in the proof of {\rm\cite[Lemma 6.7]{lacker2016general}}.
		In that proof, a key technical step uses implicitly the following erroneous argument $($see the paragraph after $(6.19)$ in {\rm\cite{lacker2016general}}$)$: let $W$ and $U$ be two independent random variables on a probability space $(\Om^{\star},\Fc^{\star},\P^{\star})$, and $f: \R \x \R \longrightarrow \R$ be such that $Z := f(W,U)$ is independent of $W$, then $Z$ is measurable with respect to the $($completed$)$ $\sigma$--algebra generated by $U$. 
		For a counter--example, let us consider the case that $W \sim N(0,1)$ and $U \sim \Uc[-1,1]$ and that $W$ is independent of $U$, then $Z:=U \mathbf{1}_{\{W \ge 0\}} - U \mathbf{1}_{\{W < 0\}}$ is independent of $W$, but not measurable w.r.t. $\sigma(U)$.
	\end{remark}

\subsubsection{Approximating relaxed controls by weak control rules}
\label{subsubsec:weak2relaxed}

	We provide here an approximation result of relaxed control rules by weak control rules, when in addition \Cref{assum:constant_case} holds.
	For the classical optimal control problem, such an approximation result is achieved by representing the martingale problems in \Cref{def:admissible_ctrl_rule}  and \Cref{def:relaxed_ctrl_rule} using the notion of martingale measures, as introduced by \citeauthor*{el1990martingale} \cite{el1990martingale} (see Section \ref{proof_Proposition-measurability-check} for a brief reminder on its definition).

\medskip

	Recall that $\Omh := \Cc^n \x \Cc^n \x \M \x \Cc^d$ is defined in \Cref{subsubsec:canonical_space}.
	Let us also introduce an abstract filtered probability space $(\Om^\star, \Fc^{\star}, \F^{\star} := (\Fc^\star_t)_{t \in [0,T]}, \P^{\star})$,
	equipped with $2 (n+d)$ i.i.d. martingale measures $(N^{\star,i})_{1\leq i\leq 2(n+d)}$, with intensity $\nu_0(\mathrm{d}a) \mathrm{d}t$, for some diffuse probability measure $\nu_0$ on $A$,
	and a sequence of i.i.d. standard $d$--dimensional Brownian motions $(W^{\star,i})_{i \ge 1}$.   
	Let us define
	\[
		\Omh^\star := \Omh \x \Om^{\star},
		\;\Fch^\star := \Fch \otimes \Fc^\star,
		\;\Fch^\star_t := \Fch_t \otimes \Fc^\star_t,
	\;\Ph_{\omb} := \muh(\omb) \otimes \P^\star,
		\; \mbox{for all}\; 
		t \in [0,T],\; \mbox{and}\;  \omb \in \Omb.
	\]
	The random elements $(\Xh, \Yh, \Lambdah, \Wh)$ and $(N^\star, W^{\star,i}, i \ge 1)$ can then naturally be extended to $\Omh^\star$.
	Let us first provide an improved version of {\rm\cite[Theorem IV--2]{el1990martingale}}, whose proof is completed in Appendix \ref{proof_Proposition-measurability-check}.

	\begin{proposition} \label{Proposition:diffusion_McKV-relaxed}
		Let $\nu \in \Pc_p(\R^n)$ and $\Pb \in \Pcb_R(\nu)$.
		Then there exists a family of measure--valued processes $(\widehat{N}^{\omb})_{\omb \in \Omb}$ such that,
		for $\Pb$--{\rm a.e.} $\omb \in \Omb$,
		$\widehat{N}^{\omb} = \big( \widehat{N}^{1,\omb}, \dots, \widehat{N}^{d,\omb} \big)$ is an $\big( \Fh^\star,\Ph_{\omb}\big)$--martingale measure with intensity $\Lambdah_t(\mathrm{d}a)\mathrm{d}t$, the martingales $(\widehat{N}^{i,\omb})_{1\leq i\leq d}$ are orthogonal, and satisfy
		\begin{equation}
		\label{eq:X_repres_Nt}
			\Yh_t
			=
			\Xh_0 
			+	
			\iint_{[0,t] \times A} 
			b \big(r, \Xh, \mu(\omb), a \big) \Lambdah_r(\mathrm{d}a)\mathrm{d}r
			+
			\iint_{[0,t] \times A} 
			\sigma\big(r, \Xh, \mu(\omb), a \big) \widehat{N}^{\omb}(\mathrm{d}a,\mathrm{d}r),
			\;
			\Wh_t
			=
			\iint_{[0,t] \times A}  
			\widehat{N}^{\omb}(\mathrm{d}a,\mathrm{d}s),
			\; \Ph_{\omb} \mbox{\rm --a.s.}
		\end{equation}
		Moreover,
		let
		$\widehat{\H}^\star = (\widehat{\Hc}^\star_t)_{t \in [0,T]}$ with $ \widehat{\Hc}^\star_t := \Gcb_t \otimes \Fch^\star_t$ be a filtration on $\Omb \x \Omh^\star$,
		denote by $\Pc^{\widehat{\H}^\star}$ the predictable $\sigma$--algebra on $[0,T] \x \Omb \x \Omh^\star$ with respect to $\widehat{\H}^\star$. 
		Then for all bounded $\Pc^{\widehat{\H}^\star} \otimes \Bc(A)$--measurable function $f:[0,T] \times \Omb \times \Omh^\star \times A  \longrightarrow \R$, 
		one can define the stochastic integral $\iint_{[0,t] \times A}  f^{\omb}(s,a) \widehat{N}^{\omb}(\mathrm{d}s, \mathrm{d}a)$ in such a way that
		\begin{equation} \label{eq:Nt_measurability}
			(t, \omb, \hat \om^\star)
			\longmapsto	
				\bigg(\iint_{[0,t] \times A}  f^{\omb} (s, a) \widehat{N}^{\omb}(\mathrm{d}a,\mathrm{d}s) \bigg)(\hat \om^\star)\; \mbox{\rm is}\; 
			\Pc^{\widehat{\H}^\star}\mbox{\rm --measurable}.
		\end{equation}
	\end{proposition}

	\begin{remark}
		With a fixed probability measure $\Pb$ on $\Om$, one can define a probability measure $\widetilde \P^{\star} := \Pb \otimes \Ph_{\cdot}$ on $\Omb \x \Omh^{\star}$ by
		\[
			\E^{\widetilde \P^{\star}} [\phi] 
			:=
			\int_{\Om \x \Omh^{\star}} \phi(\omb, \hat \om^\star) \Ph_{\omb}(\mathrm{d} \hat \om^{\star}) \Pb(\mathrm{d} \omb),
			~\mbox{for all bounded r.v.}~
			\phi: \Omb \x \Omh^{\star} \to \R.
		\]
		We can also consider the augmented filtration $\widehat{\H}^{\star, \widetilde \P^{\star}}$ of $\widehat{\H}^\star$ under $ \widetilde \P^{\star}$,
		which in particular contains all $ \widetilde \P^{\star}$--null sets in $\Omb \x \Omh^{\star}$.
		At the same time, any $\widehat{\H}^{\star, \widetilde \P^{\star}}$--predictable process is $ \widetilde \P^{\star}$--indistinguishable to a $\widehat{\H}^{\star}$--predictable process (see e.g. \cite[Theorem IV.78]{dellacherie1978probabilities}).
	\end{remark}

	\begin{proposition} \label{Equivalence-Proposition_General}
		Let {\rm \Cref{assum:main1}} and {\rm \Cref{assum:constant_case}} hold, assume that $A \subset \R^j$ for $j \ge 1$,
		and that $\nu \in \Pc_{p^\prime}(\R^n)$ with the constant $p^\prime$ given in {\rm \Cref{assum:main1}}.
		Then for every $\Pb \in \Pcb_R(\nu)$, there exists a sequence $\big( \Pb^m \big)_{m \ge 1} \subset \Pcb_W(\nu)$ such that
		\[
			\lim_{m \to \infty} \Wc_p \big( \Pb^m, \Pb \big) = 0.
		\]
	\end{proposition}
	\begin{proof}

	We only provide here the proof with the additional condition that $\sigma_0$ is a constant, 
	which illustrates better our main ideas. We refer to \Cref{subsec:Equivalence-Proposition} for a proof in the general case.

\medskip
	First, let $\Pb \in \Pcb_R(\nu)$, recall from \Cref{Proposition:diffusion_McKV-relaxed} that on the enlarged filtered space 
	$\big(\Omh^\star, \Fch^\star, \Fh^\star\big)$, 
	we have a family $(\widehat{N}^{\omb})_{\omb \in \Omb}$ such that
	$\widehat{N}^{\omb}$ is a martingale measure with intensity $\Lambda_t(\mathrm{d}a)\mathrm{d}t$ under the probability measure $\Ph_{\omb} := \muh(\omb) \otimes \P^\star$, for $\Pb$--a.e. $\omb \in \Omb$,
	and
	\begin{align*}
		\Xh_t
		&=
		\Xh_0 
		+	
		\iint_{[0,t] \times A}
		b \big(r, \Xh, \mu(\omb), a \big) \Lambdah_r(\mathrm{d}a)\mathrm{d}r
		+
	\iint_{[0,t] \times A}
		\sigma\big(r, \Xh, \mu(\omb), a \big) \widehat{N}^{\omb}(\mathrm{d}a,\mathrm{d}r)
		+
		\sigma_0 B_t(\omb),\; t\in[0,T],\; \Ph_{\omb} \mbox{\rm --a.s.},
		\\
		\Wh_t
		&=
		\iint_{[0,t] \times A} 
		\widehat{N}^{\omb}(\mathrm{d}a,\mathrm{d}s),
		\; t\in[0,T],\; \Ph_{\omb} \mbox{\rm --a.s.}
	\end{align*}
	By \Cref{Lemm:app-Weak_Relaxed} below,
	there exists, on $(\Omh^\star, \Fch^\star)$, a sequence $(\Fh^{\star,m})_{m \ge 1}$ of sub--filtrations of $\Fh^\star$,
	together with a sequence of family of processes $\big(\hat \alpha^m, (\Wh^{\omb, m})_{\omb \in \Omb}, (\Xh^{\omb,m})_{\omb \in \Omb} \big)_{m \ge 1}$, 
	where $\hat \alpha^m$ is an $A$--valued $\Fh^m$--predictable process for each $m \ge 1$,
	and for $\Pb$--a.e. $\omb \in \Omb$, $\Wh^{\omb, m}$ is an $(\Fh^{\star,m},\Ph_{\omb})$--Brownian motion, and with $\Lambdah^m(\mathrm{d}a, \mathrm{d}t) = \delta_{\hat \alpha^m_t}(\mathrm{d}a) \mathrm{d}t$,
	\begin{equation} \label{eq:cvg_conditional_law}
		\Ph_{\omb} \Big[ \lim_{m \to \infty} \Lambdah^m(\mathrm{d}a, \mathrm{d}t) = \Lambdah(\mathrm{d}a, \mathrm{d}t) \Big] = 1,
		\;
		\lim_{m \to \infty}
		\Wc_p \Big(
			\Lc^{\Ph_{\omb}} \big( \Xh^{\omb,m}, \Lambdah^m(\mathrm{d}a, \mathrm{d}t), \Wh^{\omb,m} \big),
			~
			\Lc^{\Ph_{\omb}} \big( \Xh, \Lambdah_t(\mathrm{d}a)\mathrm{d}t, \Wh \big)
		\Big) 
		= 0,
	\end{equation}
	and
	\begin{equation} \label{eq:def_Xt_ombm}
		\Xh^{\omb,m}_t
		=
		\Xh_0 
		+ 
		\int_0^t b\big(r, \Xh^{\omb,m}, \Ph_{\omb} \circ (\Xh^{\omb,m})^{-1}, \hat \alpha^m_r\big) \mathrm{d}r
		+ 
		\int_0^t \sigma \big(r, \Xh^{\omb,m}, \Ph_{\omb} \circ (\Xh^{\omb,m})^{-1}, \hat \alpha^m_r \big)  \mathrm{d} \Wh^{\omb,m}_r + \sigma_0 B_t(\omb),
	\end{equation}
	and for each $m \ge 1$
	\begin{equation} \label{eq:C_adapt}
	    (t, \omb, \hat \om^\star) \longmapsto
		\big( \Xh^{\omb,m}_{t \wedge \cdot}(\hat \om^\star), (\Lambdah^m)^t(\hat \om^\star), \Wh^{\omb,m}_{t \wedge \cdot}(\hat \om^\star) \big)\; \mbox{\rm is}\; 
			\Pc^{\widehat{\H}^\star}\mbox{\rm--measurable},
	\end{equation}
	so that, with the predictable $\sigma$--algebra $\Pc^{\Gb}$ on $[0,T] \x \Omb$ with respect to $\Gb$
	\begin{equation*} 
		(t, \omb) \longmapsto
		\Lc^{\Ph_{\omb}} \big( \Xh^{\omb,m}_{t \wedge \cdot}, (\Lambdah^m)^t, \Wh^{\omb,m} \big)\; \mbox{is}\; 
		\Pc^{\Gb}\mbox{--measurable}.
	\end{equation*}

\medskip	
	Further, let us denote $\Yh^{\omb,m} := \Xh^{\omb,m} - \sigma_0 B(\omb)$ and
	\[
		\Pb^m
		:=
		\int_{\Omb} \Lc^{\Ph_{\omb}} 
		\Big( \Xh^{\omb,m}, \Yh^{\omb,m}, \Lambdah^m, \Wh^{\omb,m}, B(\omb), 
			\Lc^{\Ph_{\omb}} \big(\Xh^{\omb,m}, \Yh^{\omb,m}, \Lambdah^m, \Wh^{\omb,m} \big) 
		\Big) \Pb(\mathrm{d}\omb).
	\]
	It follows by \eqref{eq:cvg_conditional_law} that $\lim_{m \to \infty} \Wc_p (\Pb^m, \Pb) = 0$.
	To conclude, it is enough to show that $\Pb^m \in \Pcb_W(\nu)$.
	Since, by construction, $\Pb^m [ \Lambda \in \M_0]=1$, then it is enough
	to show that $\Pb^m \in \Pcb_R(\nu)$.
	To this end, let us check that $\Pb^m$ satisfies all the conditions in \Cref{prop:eqivalence_def_relaxed}.

\medskip
	It is easy to check that 
	$\Pb^m \big[ \muh \circ (\Xh_0)^{-1} = \nu, X_0=Y_0, W_0=0, B_0=0 \big] = 1$
	and
	$\E^{\Pb^m} \big[ \|X\|^p+	\iint_{[0,T] \times A}  |a - a_0|^p \Lambda_t(\mathrm{d}a) \mathrm{d}t \big]< \infty$.
	Furthermore, for every $\phi \in C_b(\Cc^n \x \Cc^n \x \M \x \Cc^d)$, 
	$\varphi \in C_b(\Pc(\Cc^n \x \Cc^n \x \M \x \Cc^d))$ and $t \in [0,T]$, we have
	\begin{align*}
		\E^{\Pb^m} \big[ \langle \phi,\muh_t \rangle \varphi \big(B,\muh \big) \big]
		&=
		\int_{\Omb} 
			\E^{\Ph_{\omb}} \big[
				\phi \big( 
					\Xh^{\omb,m}_{t \wedge \cdot}, \Yh^{\omb,m}_{t \wedge \cdot}, (\Lambdah^m)^t , \Wh^{\omb,m} 
				\big) 
			\big]
			\varphi \Big( B(\omb), \Lc^{\Ph_{\omb}} \big( \Xh^{\omb,m}, \Yh^{\omb,m}, \Lambdah^{m}, \Wh^{\omb,m} \big) \Big) 
		\Pb(\mathrm{d}\omb)
		\\
		&=
		\E^{\Pb^m} \big[ \phi \big( X_{t \wedge \cdot},Y_{t \wedge \cdot},\Lambda^t,W \big) \varphi \big(B,\muh \big) \big].
	\end{align*}
	This implies that
	\[
		\muh_t
		=
		(\Pb^m)^{\Gcb_T} \circ \big( X_{t \wedge \cdot},Y_{t \wedge \cdot},\Lambda^t,W \big)^{-1},
		\; \Pb^m \mbox{--a.s.}
	\]

	Next, for all $\phi \in C_b(\R^{\ell})$, $\psi \in C_b(\Cc^n \x \Cc^n \x \M \x \Cc^d \x \Cc^\ell \x C([0,T];\Pc(\Cc^n \x \Cc^n \x \M \x \Cc^d)))$, and $s \in[0, t]$,
	we have 
	\begin{align*}
		&\
		\E^{\Pb^m} \big[ \phi (B_t-B_s) \psi \big( X_{s \wedge \cdot},Y_{s \wedge \cdot},\Lambda^s,W_{s \wedge \cdot},B_{s \wedge \cdot},\muh_{s \wedge \cdot}\big) \big]
		\\
		=&\
		\int_{\Omb} 
			\phi \big(B_t(\omb)-B_s(\omb) \big)  
			\E^{\Ph_{\omb}} \Big[ 
				\psi \big( \Xh^{\omb,m}_{s \wedge \cdot}, \Yh^{\omb,m}_{s \wedge \cdot}, (\Lambdah^m)^s, \Wh^{\omb,m}_{s \wedge \cdot}, B_{s \wedge \cdot}(\omb), \Lc^{\Ph_{\omb}}\big( \Xh^{\omb,m}_{s \wedge \cdot}, \Yh^{\omb,m}_{s \wedge \cdot},(\Lambdah^m)^s, \Wh^{\omb,m} \big) \big) 
			\Big] 
		\Pb(\mathrm{d}\omb)
		\\
		=&\
		\E^{\Pb^m} \big[ \phi \big(B_t(\omb)-B_s(\omb) \big)  \big]
		\int_{\Omb} 
			\E^{\Ph_{\omb}} \Big[ 
				\psi \big( \Xh^{\omb,m}_{s \wedge \cdot}, \Yh^{\omb,m}_{s \wedge \cdot}, (\Lambdah^m)^s, \Wh^{\omb,m}_{s \wedge \cdot}, B_{s \wedge \cdot}(\omb), \Lc^{\Ph_{\omb}}\big( \Xh^{\omb,m}_{s \wedge \cdot}, \Yh^{\omb,m}_{s \wedge \cdot},(\Lambdah^m)^s, \Wh^{\omb,m} \big) \big) 
			\Big] 
		\Pb(\mathrm{d}\omb)
		\\
		=&\
		\E^{\Pb^m} \big[ \phi (B_t-B_s) \big]
		\E^{\Pb^m} \big[ \psi \big( X_{s \wedge \cdot},Y_{s \wedge \cdot},\Lambda^s,W_{s \wedge \cdot},B_{s \wedge \cdot},\muh_{s \wedge \cdot}\big) \big],
		\end{align*}
		which implies that $B$ has independent increments with respect to $(\Fb, \Pb^m)$.
		Besides, since $\Pb^m \circ B^{-1}$ is the Wiener measure, 
		it follows that $B$ is an $( \Fb,\Pb^m)$--Brownian motion.
		Also, as $Z = X - Y = \sigma_0 B$, one has immediately that $S^f$ (defined in \eqref{eq:K_process}) is an $( \Fb^\circ,\Pb^m)$--martingale for all $f \in C^2_b(\R^{n + \ell})$.
		Finally, by construction, Condition $(iii)$ in \Cref{prop:eqivalence_def_relaxed} is also satisfied.
		Therefore, $\Pb^m \in \Pcb_R(\nu)$, and hence $\Pb^m \in \Pcb_W(\nu)$.
	\end{proof}

	\begin{lemma} \label{Lemm:app-Weak_Relaxed}
		Let us stay in the context of {\rm\Cref{Equivalence-Proposition_General}}, and assume in addition that $\sigma_0$ is a constant.
		Then on the space $(\Omh^\star, \Fch^\star)$, 
		there exists a sequence $(\Fh^{\star,m})_{m \ge 1}$ of sub--filtrations of $\Fh^\star$,
		together with a sequence of family of processes $\big(\hat \alpha^m, (\Wh^{\omb, m})_{\omb \in \Omb}, (\Xh^{\omb,m})_{\omb \in \Omb} \big)_{m \ge 1}$,
		where $\hat \alpha^m$ is an $A$--valued $\Fh^{\star,m}$--predictable process for each $m \ge 1$,
		and for $\Pb$--a.e. $\omb \in \Omb$, $\Wh^{\omb, m}$ is an $(\Fh^{\star,m},\Ph_{\omb})$--Brownian motion,
		and with $\Lambdah^m(\mathrm{d}a, \mathrm{d}t) = \delta_{\hat \alpha^m_t}(\mathrm{d}a)\mathrm{d}t$
		and $\Xh^{\omb, m}$ be defined in \eqref{eq:def_Xt_ombm}, 
		the convergence and measurability results in \eqref{eq:cvg_conditional_law} and \eqref{eq:C_adapt} hold true.     
	\end{lemma}

	\begin{proof}
	We will adapt the arguments in \cite[Theorem 4.9.]{el1987compactification} to approximate, under each $\Ph_{\omb}$,
	the process
	\[
		\Xh_t
		=
		\Xh_0 
		+
		\iint_{[0,t] \times A}
		b \big(r, \Xh, \mu(\omb), a \big) \Lambdah_r(\mathrm{d}a, \mathrm{d}r)
		+
		\iint_{[0,t] \times A}
		\sigma\big(r, \Xh, \mu(\omb), a \big) \widehat{N}^{\omb}(\mathrm{d}a,\mathrm{d}r)
		+ \sigma_0 B_t(\omb),\; t\in[0,T],
		\;\Ph_{\omb} \mbox{\rm --a.s.},
	\]
	and at the same time check the measurability property at each step.
	
\medskip

	$\underline{Step ~1.}$
	We first show that one can assume w.l.o.g. that $A \subset \R^j$ is a compact set.
	Indeed, for each $e \ge 1$, let us denote $A_e := A \cap [-e, e]^j$, $\pi_e : A \longrightarrow A_e$ the projection from $A$ to $A_e$,
	and then define $\Lambdah^e$ and for all $\omb \in \Omb,$ $\widehat{N}^{\omb, e}$ by
	\[
		\iint_{[0,T] \times A} \phi(s,a) \Lambdah^e (\mathrm{d}a,\mathrm{d}r) := \iint_{[0,T] \times A} \phi_e(s,a) \Lambdah(\mathrm{d}a,\mathrm{d}r),
		\;
		\iint_{[0,T] \times A} \phi(s,a) \Nh^{\omb, e} (\mathrm{d}a,\mathrm{d}r) := \iint_{[0,T] \times A} \phi_e(s,a) \Nh^{\omb}(\mathrm{d}a,\mathrm{d}r),
	\]
	for all $\phi \in C_b([0,T] \x A)$ and $\phi_e (s,a) := \phi(s, \pi_e(a))$.
	Denote also $(b_e, \sigma_e)(t, \xb, \nu, a) := (b, \sigma) (t, \xb, \nu, \pi_e(a))$.
	For $\omb \in \Omb,$ let
	$\Xh^{\omb,e}$ be the unique solution to
	\begin{align*}
		\Xh^{\omb,e}_t
		=&\ 
		\Xh_0 
		+
		\iint_{[0,t] \times A}
		b \big(r, \Xh^{\omb, e}, \mu(\omb), a \big) \Lambdah^e_r(\mathrm{d}a)\mathrm{d}r
		+
		\iint_{[0,t] \times A}
		\sigma \big(r, \Xh^{\omb, e}, \mu(\omb), a \big) \Nh^{\omb, e}(\mathrm{d}a,\mathrm{d}r)
		\\
		&+ \sigma_0 B_t(\omb),\; t\in[0,T],\; 
	\Ph_{\omb} \mbox{\rm --a.s.}
	\end{align*}
	Then by similar arguments as in the proof of \Cref{lemma:strongApp_step1}, it is standard to deduce that, for some constant $C>0$ independent of $e \ge 1$ and $\omb$, and which may change value from line to line
	\[
		\E^{\Ph_{\omb}}  \bigg[\sup_{t \in [0,T]}| \Xh^{\omb, e}_t - \Xh_s |^p  \bigg]
		\le
		C
		 \E^{\Ph_{\omb}}  \bigg[ \iint_{[0,t] \times A}
		 	\big| \big( (b,\sigma) - (b_e, \sigma_e) \big) (t, \Xh, \Ph_{\omb} \circ (\Xh)^{-1},a)\big|^p  \Lambdah_t (\mathrm{d}a)\mathrm{d}t
            \bigg].
	\]
	Using the growth conditions on $(b, \sigma)$ in \Cref{assum:main1},
	we have
	\[
		\iint_{[0,T] \times A} \big| \big( (b,\sigma) - (b_e, \sigma_e) \big) \big( t, \Xh, \Ph_{\omb} \circ (\Xh)^{-1},a \big) \big|^p  \Lambdah_t (\mathrm{d}a)\mathrm{d}t
	        \le
	        C \bigg( \big\| \Xh \big\|^p + \E^{\Ph_{\omb}} \big[ \big\| \Xh \big\|^p \big] \iint_{[0,T] \times A} |a|^p \Lambdah_t (\mathrm{d}a)\mathrm{d}t \bigg).
	\]
	It follows by the dominated convergence theorem that, for $\Pb$--a.e. $\omb \in \Omb$,
	\[
		\lim_{e \to \infty} \E^{\Ph_{\omb}}  \bigg[\sup_{t \in [0,T]} \big| \Xh^{\omb, e}_t - \Xh_s \big|^p  \bigg] = 0.
	\]
	Moreover, as $A$ is a Polish subspace of $\R^j$, then $A$ is closed, and hence $A_e$ is compact. This allows to reduce the problem to the case where $A$ is compact.

	\medskip

	$\underline{Step~2.}$
	We now assume in addition that $A$ is compact and proceed the proof.
	By compactness of $A$, there is a sequence of positive reel numbers $(\delta_e)_{e \ge 1}$ such that $\lim_{e \to \infty} \delta_e = 0$,
	and for each $e \ge 1$, one can find a partition $(A_1^e, \dots, A_e^e)$  of $A$ 
	and $(a_1^e, \dots, a_e^e)$
	satisfying $a_i^e \in A_i^e$ and $|a_i^e-a|< \delta_e$ for all $a \in A_i^e$, $i \in \{1,\dots,e\}$. 
	For $\omb \in \Omb,$ let $\Xh^{\omb,e}$ be the unique solution to the SDE
	\begin{equation} \label{eq:Xtm}
		\Xh^{\omb, e}_t
		=
		\Xh_0 
		+ 
		\sum_{i=1}^e \int_0^t b\big(r, \Xh^{\omb, e},\Ph_{\omb} \circ ( \Xh^{\omb, e})^{-1}, a^e_i\big)  \Lambdah_r(A^e_i) \mathrm{d}r 
		+ 
		\int_0^t \sigma \big(r, \Xh^{\omb,e}, \Ph_{\omb} \circ (\Xh^{\omb, e})^{-1},a^e_i \big)  \mathrm{d} \Nh^{\omb}_r(A^e_i)
		+
		\sigma_0 B_t(\omb).
	\end{equation}
	Using again standard arguments as in the proof of \Cref{lemma:strongApp_step1}, we obtain that, for some constant $C> 0$ (independent of $e$ and $\omb$), which may change from line to line
	\[
		\E^{\Ph_{\omb}} 
		\bigg[\sup_{t \in [0,T]} \big| \Xh^{\omb, e}_t - \Xh_t \big|^p \bigg]
		\le
		C \E^{\Ph_{\omb}}  \bigg[ \sum_{i=1}^e \iint_{[0,T] \times A^e_i} 
			\big|
				(b,\sigma)(r, \Xh, \Ph_{\omb} \circ (\Xh)^{-1},a) -(b,\sigma)(r, \Xh,\Pb_{\omb} \circ (\Xh)^{-1},a^e_i)
			\big|^p \Lambda_r(\mathrm{d}a)\mathrm{d}r 
		\bigg].
	\]
	For every fixed $(r, \mathbf{x}, \nu)$, the map $a \longmapsto (b, \sigma)(r, \mathbf{x}, \nu, a)$ is continuous and hence uniformly continuous. Using dominated convergence,
	it follows that, for $\Pb$--a.e. $\omb \in \Omb$,
	\begin{equation} \label{eq:Xe2X}
		\lim_{e \to \infty}
		\E^{\Ph_{\omb}} 
		\bigg[\sup_{t \in [0,T]} \big| \Xh^{\omb, e}_t - \Xh_t \big|^p \bigg]
		=
		0.
	\end{equation}
	Recall that the space $\Omh^\star$ is equipped with a sequence of i.i.d. Brownian motion $(W^{\star,i})_{i \ge 1}$.
	Let us define, for each $i =1, \dots, e$,
	\[
		\Zh^{\omb, e, i}_t
		:=
		\int_0^t ( q^{e,i}_s )^{-1/2} \mathbf{1}_{\{q^{e,i}_s \neq 0 \}} \mathrm{d} \Nh^{\omb}_s (A^e_i)
		+
		\int_0^t  \mathbf{1}_{ \{ q^{e,i}_s = 0 \}} \mathrm{d}W^{*,i}_s,
		\; \mbox{with}\; 
		q^{e,i}_s :=  \Lambdah_s (A^e_i),
		\;\mbox{for all}\;t \in [0,T].
	\]
	Then it is direct to see that: for $\Pb$--a.e. $\omb \in \Omb,$ $(\Zh^{\omb, e, 1}, \dots, \Zh^{\omb, e, e} )$ is an $e$--dimensional $\big(\Fh,\Ph_{\omb}\big)$--Brownian motion, and one can rewrite \eqref{eq:Xtm}, for any $t\in[0,T]$
	\begin{equation} \label{eq:Xtm2}
		\Xh^{\omb, e}_t
		=
		\Xh_0 
		+ 
		\sum_{i=1}^e \int_0^t b\big(r, \Xh^{\omb, e},\Ph_{\omb} \circ ( \Xh^{\omb, e})^{-1}, a^e_i\big) q^{e,i}_r \mathrm{d}r 
		+ 
		\int_0^t  \sigma \big(r, \Xh^{\omb,e}, \Ph_{\omb} \circ (\Xh^{\omb, e})^{-1},a^m_i \big) \sqrt{q^{e,i}_r} \mathrm{d} \Zh^{\omb, e, i}_r
		+
		\sigma_0 B_t(\omb),
	\end{equation}
	and
	\begin{equation} \label{eq:Wtm2}
		\Wh_t
		=
		\Zh^{\omb,e}_t :=
		\sum_{i=1}^e \int_0^t \sqrt{q^{e,i}_r} \mathrm{d} \Zh^{\omb, e, i}_r,\; t\in[0,T],\; \Ph_{\omb} \mbox{--a.s.}
	\end{equation}
	Furthermore, by considering the process $(q^{e,i}_r, i =1, \dots, e)_{r \in [0,T]}$ as a control process,
	and using \Cref{lemma:strongApp_step1},
	we can assume w.l.o.g. that $q^{e,i}$ is an $\Ft$--predictable process,
	and is in addition constant on each interval $[t_k, t_{k+1})$, for $0 = t_0 < t_1 < \dots < t_K = T$.
	Let $\Lambdah^e (\mathrm{d}a, \mathrm{d}t)  := \sum_{i=1}^e q^{e,i}_t \delta_{a^e_i}(\mathrm{d}a) \mathrm{d}t$,
	it follows by  \eqref{eq:Xe2X} and \eqref{eq:Wtm2} that
	\begin{equation} \label{eq:cvg_conditional_law2}
		\Ph_{\omb} \Big[ \lim_{e \to \infty} \Lambdah^e (\mathrm{d}a, \mathrm{d}t) = \Lambdah(\mathrm{d}a, \mathrm{d}t) \Big] = 1,
		\;
		\lim_{e \to \infty}
		\Wc_p \Big(
			\Lc^{\Ph_{\omb}} \big( \Xh^{\omb, e}, \Lambdah^e (\mathrm{d}a, \mathrm{d}t), \Zh^{\omb,e} \big),
			~
			\Lc^{\Ph_{\omb}} \big( \Xh, \Lambdah_t(\mathrm{d}a)\mathrm{d}t, \Wh \big)
		\Big) 
		= 0.
	\end{equation}
	Moreover, for every $e \ge 1$, it follows by \eqref{eq:Nt_measurability} in \Cref{Proposition:diffusion_McKV-relaxed} that one can choose $(\Zh^{\omb,e,i})_{i=1, \dots,e}$ such that
	\[
		(t, \omb, \hat \om^\star) \longmapsto
		 \big( \Zh^{\omb,e,1}_{t \wedge \cdot}(\hat \om^\star), \dots,  \Zh^{\omb,e,e}_{t \wedge \cdot}(\hat \om^\star)\big)
		~\mbox{is}~
		\Pc^{\widehat{\H}^\star}\mbox{--measurable}.
	\]
	Recall that the solution $\Xh^{\omb,e}$ of SDE \eqref{eq:Xtm2} can be defined by Picard iterations (see e.g. \cite[Theorem A.3.]{djete2019mckean}),
	then by similar arguments as in \cite[Lemma 2.6.]{possamai2015stochastic}, one can choose $\Xh^{\omb,e}$ such that
	\begin{equation} \label{eq:measur_Xe}
		(t, \omb, \hat \om^\star) \longmapsto
		 \big( \Xh^{\omb,e}_{t \wedge \cdot}(\hat \om^\star), (\Lambdah^e)^t(\hat \om^\star), \Zh^{\omb,e,1}_{t \wedge \cdot}(\hat \om^\star), \dots,  \Zh^{\omb,e,e}_{t \wedge \cdot}(\hat \om^\star)\big)
		~\mbox{is}~
		\Pc^{\widehat{\H}^\star}\mbox{--measurable}.
	\end{equation}

	$\underline{Step~3.}$
	We now consider the approximation of 
	$(\Xh^{\omb, e})_{ \omb \in \Omb}$, $\Lambdah^e$ and $\Zh^{\omb, e}$ for a fixed $e \ge 1$.
	For simplicity of presentation, we consider the case $e=2$, $K=2$ and $t_1 = T/2$,
	so that
	\[
		\Lambdah^2(\mathrm{d}a, \mathrm{d}t)  := q^{2,1}_t \delta_{a^2_1}(\mathrm{d}a) \mathrm{d}t + q^{2,2}_t \delta_{a^2_2}(\mathrm{d}a) \mathrm{d}t,
		\;
		q^{2,1}_t + q^{2,2}_t = 1,
		\;
		(q^{2,1}_t, q^{2,2}_t) = 
		\begin{cases}
			(q^{2,1}_0, q^{2,2}_0),\; \mbox{for}~ t \in [0, t_1), \\
			(q^{2,1}_{t_1}, q^{2,2}_{t_1}), \; \mbox{for}~t \in [t_1, T],
		\end{cases}
	\]
	where $(q^{2,1}_0, q^{2,2}_0) \in [0,1]^2$ are two deterministic constants
	and $(q^{2,1}_{t_1}, q^{2,2}_{t_1})$ are $[0,1]$--valued $\Fch^\star_{t_1}$--measurable random variables.

\medskip 
First, we consider a further discretisation of $[0,t_1]$: $0 = t^1_0 < t^1_1 < \dots < t^1_m = t_1$ 
	with $t^1_i := i \Delta t$, $\Delta t := t_1/m$,
	and then
	define two $d$--dimensional processes $(\Wh^{\omb, m, 1}, \Wh^{\omb, m,2})$.
	Let $\Wh^{\omb, m, 1}_0 = \Wh^{\omb, m,2}_0 = 0$,
	and then for each $i=0, \dots, m-1$, let
	\[
		\Wh^{\omb, m,1}_t 
		:=
		\begin{cases} 
			\Wh^{\omb, m,1}_{t^1_i} 
			+ 
			\sqrt{q^{2,1}_0}  \Big( \Zh^{\omb, 2,1}_{t^1_i + (t-t^1_i) /q^{2,1}_0 } - \Zh^{\omb, 2,1}_{t^1_i} \Big) ,
			\; t \in [t^1_i, \theta^1_i], \\[0.8em]
			\Wh^{\omb, m,1}_{\theta^1_i},
			\; t \in (  \theta^1_i, t^\mathbf{1}_{i+1}],
		\end{cases}
		~\mbox{with}~
		\theta^1_i :=   t^1_i + q^{2,1}_0 \Delta t \in [t^1_i, t^\mathbf{1}_{i+1}],
	\]
	and
	\[
		\Wh^{\omb, m,2}_t 
		:=
		\begin{cases} 
			\Wh^{\omb, m,2}_{t^1_i} ,
			\; t \in [t^1_i, \theta^1_i], \\[0.8em]
			\Wh^{\omb, m,2}_{\theta^1_i}
			+ 
			\sqrt{q^{2,2}_0}  \Big( \Zh^{\omb, 2,2}_{t^1_i + (t-\theta^1_i) /q^{2,2}_0 } - \Zh^{\omb, 2,2}_{t^1_i} \Big),
			\; t \in (  \theta^1_i, t^\mathbf{1}_{i+1}].
		\end{cases}
	\]
	Namely, one 'compresses' the increment of the Brownian motion $\Zh^{\omb, 2,1}$ from $[t^1_i, t^\mathbf{1}_{i+1}]$ to $[t^1_i, \theta^1_i]$ to obtain $\Wh^{\omb, m,1}$,
	and 'compresses' the increment of the Brownian motion $\Zh^{\omb, 2,2}$ from $[t^1_i, t^\mathbf{1}_{i+1}]$ to $[\theta^1_i, t^\mathbf{1}_{i+1}]$ to obtain $\Wh^{\omb, m,2}$.
	
\medskip
	Next, on $[t_1, T]$, we take the discretisation $t_1 = t^2_0 < \dots, t^2_m = T$ with $t^2_i := t_1 + i \Delta t$, $\Delta t := t_1/m = (T-t_1)/m$,
	and 
	for each $i =0, \dots, m-1$, let $\theta^2_i := t^2_i + q^{2,2}_{t_1} \Delta t \in [t^2_i, t^2_{i_1}]$.
	Notice that $q^{2,2}_{t_1}$ is an $\Fch^\star_{t_1}$--random variable.
	It follows that the $(\theta^2_i)_{0\leq i\leq m-a}$ are also random.
	By rewriting its definition on $[0,t_1]$ in an equivalent way, 
	we define $(\Wh^{\omb, m,1}, \Wh^{\omb, m,2})$ on $[t_1, T]$ by
	\[
	\begin{cases}
		\Wh^{\omb, m,1}_t 
		:=
			\Wh^{\omb, m,1}_{t^2_i} 
			+ 
			\sqrt{q^{2,2}_0}  \Big( \Zh^{\omb, 2,1}_{t^2_i + (t \wedge \theta^2_i - t^2_i) /q^{2,2}_0 } - \Zh^{\omb, 2,1}_{t^2_i} \Big) ,
			\\[0.8em]
		\Wh^{\omb, m,2}_t 
		:=
		\Wh^{\omb, m,2}_{t^2_i} \mathbf{1}_{\{ t \in [t^2_i, \theta^2_i]\}}
		+
		\Big(
			\Wh^{\omb, m,2}_{\theta^2_i}
			+ 
			\sqrt{q^{2,2}_0}  \Big( \Zh^{\omb, 2,2}_{t^2_i + (t-\theta^2_i) /q^{2,2}_0 } - \Zh^{\omb, 2,2}_{t^2_i} \Big)
		\Big) \mathbf{1}_{\{ t \in (  \theta^2_i, t^2_{i+1}]\}},
	\end{cases}
	\mbox{for}~
	t \in (  t^2_i, t^2_{i+1}].
	\]	
	
	Next, let us define $I^m_1 := \cup_{i=1}^{m-1} ([t^1_i, \theta^1_i) \cup [t^2_i, \theta^2_i) )$
	and $I^m_2 := \cup_{i=0}^{m-1} ([\theta^1_i, t^\mathbf{1}_{i+1}) \cup [\theta^2_i, t^2_{i+1})  )$
	\[
		\Wh^{\omb, m}_t :=  \big(\Wh^{\omb, m,1}_{t - \Delta t} + \Wh^{\omb, m,2}_{t - \Delta t} \big)\mathbf{1}_{\{t \in [\Delta t, T]\}},
		~\mbox{and}~
		\Lambdah^{2,m} (\mathrm{d}a ,\mathrm{d}t) := \delta_{\alpha^m_t}(\mathrm{d}a) \mathrm{d}t,
		~\mbox{with}~
		\alpha^m_t := a^2_1 \mathbf{1}_{\{t \in I^m_1\}} + a^2_2 \mathbf{1}_{\{t \in I^m_2 \}}.
	\]
	Notice that $\Pb$--a.e. $\omb \in \Omb,$ $\Wh^{\omb, m,1}$ and $\Wh^{\omb, m,2}$ are $\Ph_{\omb}$--martingales w.r.t. their natural filtrations with quadratic variation $c^{m,1}_{\cdot}:=\int_0^\cdot \mathbf{1}_{I^m_1}(r) \mathrm{d}r$ and $c^{m,2}_{\cdot}:=\int_0^\cdot \mathbf{1}_{I^m_2}(r) \mathrm{d}r$ respectively.
	Further, with the time shift appearing in its definition, the process $\Wh^{\omb, m}$ is $\Fh^\star$--adapted.
	Moreover,  $\Wh^{\omb, m}$ is a $\Ph_{\omb}$--Brownian motion on $[\Delta t, T]$ with respect to its natural filtration (but not $\Fh^\star$),
	and
	\begin{equation} \label{eq:cvg_Wt12}
		\big( \Lambdah^{2,m}, \widehat c^{m,1}_{\cdot},\widehat c^{m,2}_{\cdot}, \Wh^{\omb, m, 1}_{\cdot}, \Wh^{\omb, m,2}_{\cdot} \big)
		\underset{m\rightarrow\infty}{\longrightarrow}
		\bigg( \Lambdah^2, \int_0^\cdot q^{2,1}_r \mathrm{d}r,\int_0^\cdot q^{2,2}_r \mathrm{d}r,  \int_0^{\cdot} \sqrt{q^{2,1}_r} \mathrm{d} \Zh^{\omb, 2, 1}_r,  \int_0^{\cdot} \sqrt{q^{2,2}_r} \mathrm{d} \Zh^{\omb, 2, 2}_r \bigg),
		~\Ph_{\omb} \mbox{--a.s.}
	\end{equation}
	
	Let us define $\Xh^{\omb,2, m} = (\Xh^{\omb, 2, m}_t)_{t \in [0,T]}$ as the unique solution, under $\Ph_{\omb}$, to
	\begin{align} \label{eq:SDE_Xte}
		\nonumber	
		\Xh^{\omb, 2,m}_t 
		=&\
		\Xh_0 
		+ 
		\int_{\Delta t}^{t \vee \Delta t} b(r, \Xh^{\omb, 2,m}_r, \Ph_{\omb} \circ (\Xh^{\omb, 2,m})^{-1}, \alpha^m_r) \mathrm{d}r
		+
		\int_{ \Delta t}^{t\vee  \Delta t} \sigma(r, \Xh^{\omb, 2,m}_r, \Ph_{\omb} \circ (\Xh^{\omb,2, m})^{-1}, \alpha^m_r) \mathrm{d} \Wh^{\omb, m}_r
		+
		\sigma_0 B_t(\omb)  \\
\nonumber		=&\ 
		\Xh_0
		+ \sum_{i=1}^2
		\Big(
			\int_{\Delta t}^{t \vee \Delta t}  b(r, \Xh^{\omb, 2,m}_r, \Ph_{\omb} \circ (\Xh^{\omb, 2,m})^{-1}, a^2_i) \mathrm{d} \widehat{c}^{m,i}_r
			+
			\int_{ \Delta t}^{t\vee  \Delta t} \!\!\!\! \sigma(r, \Xh^{\omb, 2,m}_r, \Ph_{\omb} \circ (\Xh^{\omb, 2,m})^{-1}, a^2_i) \mathrm{d} \Wh^{\omb, m, i}_{r-\Delta t}
		\Big)
		\\
		&+
		\sigma_0 B_t(\omb).
	\end{align}
	Besides, as in \Cref{lemma:estimates}, it is standard to obtain the following estimate, for some constant $C> 0$
        \[
		\sup_{m \ge 1}
		\E^{\Ph_{\omb}}
		\bigg[\sup_{t \in [0,T]}|\Xh^{\omb,2,m}_t|^{p^{\prime}}
		\bigg]
		\le 
		C 
		\bigg(1+\int_{\R^n}|x|^{p^{\prime}}\nu(\mathrm{d}x) \bigg) 
		< 
		\infty.
	\]
	Using  \cite[Proposition B.1]{carmona2014mean}, 
	it follows that:  for $\Pb$--a.e. $\omb \in \Omb,$
	\[
		\Big( \Lc^{\Ph_{\omb} } \big( \widehat c^{m,1}_{\cdot},\widehat c^{m,2}_{\cdot}, \Wh^{\omb, m, 1}_{\cdot}, \Wh^{\omb, m,2}_{\cdot} ,  \Xh^{\omb, 2, m}_{\cdot} \big) \Big)_{m \ge 1}
		~\mbox{is tight under}~
		\Wc_p.
	\]
	Then along an arbitrary convergent sub-sequence $(m_k)_{k \ge 1}$ (which can potentially depend on $\omb$), one has
	\[
		\Lc^{\Ph_{\omb}} \big( \widehat c^{m_k,1}_{\cdot},\widehat c^{m_k,2}_{\cdot}, \Wh^{\omb, m_k, 1}_{\cdot}, \Wh^{\omb, m_k,2}_{\cdot} ,  \Xh^{\omb, 2, m_k}_{\cdot} \big)
		\underset{k\rightarrow\infty}{\longrightarrow}
		\Lc^{\P^{\star}} \big(  \widehat c^{\star,1}, \widehat c^{\star,2}, \Wh^{ \star, 1}_{\cdot}, \Wh^{ \star,2}_{\cdot} ,  \Xh^{ \star}_{\cdot} \big)
		~\mbox{weakly and under}~\Wc_p,
	\]
	for some random elements $\big( \widehat c^{\star,1}, \widehat c^{\star,2}, \Wh^{\star, 1}_{\cdot}, \Wh^{\star,2}_{\cdot} ,  \Xh^{\star}_{\cdot} \big)$ in $(\Om^\star, \Fc^\star, \P^\star)$.
	By considering the martingale problem associated with the SDE \eqref{eq:SDE_Xte},
	it is standard to check that $\Xh^\star$ satisfies
	\[
		\Xh^{\star}
		=
		\Xh_0
		+ \sum_{i=1}^2
		\int_{0}^{t} b(r, \Xh^{\star}_r, \P^\star \circ (\Xh^{\star})^{-1}, a^2_i) \mathrm{d}\widehat{c}^{\star,i}_r
		+
		\int_{0}^{t} \sigma (r, \Xh^{\star}_r, \P^\star \circ (\Xh^{\star})^{-1},  a^2_i) \mathrm{d} \Wh^{\star ,i}_{r}
		+
		\sigma_0 B_t(\omb),
		~\P^\star\mbox{--a.s.}
	\]
	Besides, by the convergence result in
	\Cref{eq:cvg_Wt12}, one has
	\[
		\Lc^{\P^\star} 
		\big(
			 \Xh_0,\widehat c^{\star,1}, \widehat c^{\star,2}, \Wh^{\star, 1}_{\cdot}, \Wh^{\star,2}_{\cdot}
		\big)
		=
		\Lc^{\Ph_{\omb}}
		\bigg( \Xh_0, \int_0^\cdot q^{2,1}_r \mathrm{d}r,\int_0^\cdot q^{2,2}_r \mathrm{d}r,  \int_0^{\cdot} \sqrt{q^{2,1}_r} \mathrm{d} \Zh^{\omb, 2, 1}_r,  \int_0^{\cdot} \sqrt{q^{2,2}_r} \mathrm{d} \Zh^{\omb, 2, 2}_r \bigg).
	\]
	Then it follows by the strong uniqueness (hence uniqueness in law) of the solution to SDE \eqref{eq:Xtm2} that
        \[
		\Lc^{\P^\star} 
		\big(
			 \widehat c^{\star,1}, \widehat c^{\star,2}, \Wh^{\star, 1}_{\cdot}, \Wh^{\star,2}_{\cdot}, \Xh^\star
		\big)
		=
		\Lc^{\Ph_{\omb}}
		\bigg( \int_0^\cdot q^{2,1}_r \mathrm{d}r,\int_0^\cdot q^{2,2}_r \mathrm{d}r,  \int_0^{\cdot} \sqrt{q^{2,1}_r} \mathrm{d} \Zh^{\omb, 2, 1}_r,  \int_0^{\cdot} \sqrt{q^{2,2}_r} \mathrm{d} \Zh^{\omb, 2, 2}_r,
		\Xh^{\omb, 2} \bigg).
	\]
	Since the limit is unique, and hence does not depend on the sub-sequence,
	we obtain that: for $\Pb$--a.e. $\omb \in \Omb$
	\[
		\Lc^{\Ph_{\omb}} \big( \Lambdah^{2,m}, \Wh^{\omb, m, 1}_{\cdot}, \Wh^{\omb, m, 2}_{\cdot} , \Wh^{\omb, m}_{\cdot} ,  \Xh^{\omb, 2, m}_{\cdot} \big)
	\underset{m\rightarrow\infty}{\longrightarrow}
		\Lc^{\Ph_{\omb}} \bigg( \Lambdah^2,  \int_0^{\cdot} \sqrt{q^{2,1}_r} \mathrm{d} \Zh^{\omb, 2, 1}_r,  \int_0^{\cdot} \sqrt{q^{2,2}_r} \mathrm{d} \Zh^{\omb, 2, 2}_r,
		\Zh^{\omb,2}_{\cdot},
		\Xh^{\omb, 2} \bigg).
	\]
	Further, using \eqref{eq:measur_Xe} and the explicit construction of $\Wh^{\omb,m}$ and the fact that the solution $\Xh^{\omb,2,m}$ of SDE \eqref{eq:SDE_Xte} can be defined by a Picard iteration,
	it follows that one can choose $\Xh^{\omb,2,m}$ such that
	\[
		(t, \omb, \hat \om^\star) \longmapsto  \Big(\Xh^{\omb,2, m}_{t \wedge \cdot}(\hat \om^\star), (\Lambdah^{2,m})^t(\hat \om^\star), \Wh^{\omb,m}_{t \wedge \cdot}(\hat \om^\star) \Big)
		~\mbox{is}~
		\Pc^{\widehat{\H}^\star}\mbox{--measurable}.
	\]
	Finally, we observe that $\Xh^{\omb,2, m}$ is only defined by SDE \eqref{eq:SDE_Xte} on $[\Delta t, T]$, with $\Delta t = t_1/e \longrightarrow 0$ when $e \longrightarrow \infty$. Thus, we can easily extend it to an SDE on $[0,T]$ as \eqref{eq:def_Xt_ombm} and preserve the same convergence and measurability properties.
	\end{proof}

 	\begin{remark} \label{rem:Gap_Weak_Relax}
		Our definition of the relaxed formulation and the proof on the approximation of relaxed control rules by weak control rules is quite different from those used by {\rm \citeauthor*{lacker2017limit} \cite{lacker2017limit}} in the non--common noise context.
		In particular, it allows to fill in a subtle technical gap in {\rm\cite[Proof of Theorem 2.4]{lacker2017limit}},
		where the approximation procedure relies on the erroneous martingale measure approximation result of {\rm\citeauthor*{meleard1992martingale} \cite{meleard1992martingale}}, as explained in {\rm \Cref{foot:meleard}}. Notice however that {\rm\cite[Paragraph before Theorem $2.4$]{lacker2017limit}} does mention the possibility of an alternative proof in the spirit of {\rm\cite{el1987compactification}} and {\rm\cite{karoui2013capacities2}}, but without more details. This is exactly the program we have carried out.
	\end{remark}

\subsubsection{Proof of Theorem \ref{thm:equivalence}}
\label{subsubsec:proofofTheorem-equivalence}
 
	\hspace{2em} $(i.1)$ Let \Cref{assum:main1} hold true, and take $\nu \in \Pc_p(\R^n)$.
	The non--emptiness of $\Pcb_W(\nu)$ follows by a stability result for the martingale problem in \Cref{assum:main1}.
	We provide a detailed proof in \Cref{Theorem:ExistenceSolutionMartingale}.
	
	\vspace{0.5em}
	
	For the convexity of $\Pcb_W(\nu)$, we first prove that $\Pcb_A(\nu)$ is convex. 
	Let us consider $(\Pb_1, \Pb_2) \in \Pcb_A(\nu)\times \Pcb_A(\nu)$, $\theta \in [0,1]$ and $\Pb := \theta \Pb_1 + (1-\theta)\Pb_2$,
	and show that $\Pb \in \Pcb_A(\nu)$.
	First, it is direct to check that $\Pb$ satisfies Conditions $(i)$ and $(iii)$ in \Cref{def:admissible_ctrl_rule}.
	To check Condition $(ii)$ in \Cref{def:admissible_ctrl_rule}, 
	we consider $t \in [0,T]$, $f \in C_b(\Cc^n \x \Cc^n \x \M \x \Cc^d)$, $\psi \in C_b(\Cc^\ell \x \Pc(\Cc^n \x \Cc^n \x \M \x \Cc^d))$, $\varphi \in C_b(\R^n \x \Cc^d)$.
	Notice that under both $\Pb_1$ and $\Pb_2$, $(X_0, W)$ has the same distribution and is independent of $(B, \muh)$,
	it follows that
	\begin{align*}
		\E^{\Pb} \big[ \varphi(X_0, W) \psi \big(B,\muh \big) \big]
		&=
		\theta \E^{\Pb_{1}} \big[ \varphi(X_0, W) \psi \big(B,\muh \big) \big]
		+
		(1- \theta) \E^{\Pb_{2}} \big[ \varphi(X_0) \beta(W) \psi \big(B,\muh \big) \big]
		\\
                &=
		\E^{\Pb} \big[ \varphi(X_0) \beta(W) \big]
		\big(\theta \E^{\Pb_{1}} \big[\psi \big(B,\muh \big) \big]
		+
		(1-\theta) \E^{\Pb_{2}} \big[ \psi \big(B,\muh \big) \big] \big)
		=
		\E^{\Pb}  \big[ \varphi(X_0, W) \big] \E^{\Pb} \big[ \psi \big(B,\muh \big) \big],
	\end{align*}
	which implies the independence of $(X_0, W)$ and $(B, \muh)$ under $\Pb$.
	Furthermore, one has, for each $i\in\{1,2\}$
	\[
		\E^{\Pb_{i}} \big[\langle f, \muh_t \rangle \psi \big(B,\muh \big) \big]
		=
		\E^{\Pb_{i}} \big[ f \big(X_{t \wedge \cdot},Y_{t \wedge \cdot},\Lambda^t,W \big) \psi \big(B,\muh \big) \big],
	\]
	then it is straightforward to obtain that
	\[
		\E^{\Pb} \big[\langle f, \muh_t \rangle \psi \big(B,\muh \big) \big]
		=
		\E^{\Pb} \big[ f \big(X_{t \wedge \cdot},Y_{t \wedge \cdot},\Lambda^t,W \big) \psi \big(B,\muh \big) \big].
	\]
	This implies that for $\Pb$--a.e. $\omb \in \Omb$
	\begin{align*}
		\muh_t (\omb)
		=
		\Pb^{\Gcb_t}_{\omb} \circ (X_{t \wedge \cdot},Y_{t \wedge \cdot},W, \Lambda^t)^{-1}
		=
		\Pb^{\Gcb_T}_{\omb} \circ (X_{t \wedge \cdot},Y_{t \wedge \cdot},W, \Lambda^t)^{-1}.
	\end{align*}
	Then $\Pb$ also satisfies Condition $(ii)$ in \Cref{def:admissible_ctrl_rule}, and hence $\Pb \in \Pcb_A(\nu)$.
	This proves that $\Pcb_A(\nu)$ is convex.
	
	\medskip
	Next, assume in addition that $(\Pb_1, \Pb_2) \in \Pcb_W(\nu)\times \Pcb_W(\nu)$, that is to say $(\Pb_1, \Pb_2) \in \Pcb_A(\nu)\times\Pcb_A(\nu)$ and $\Pb_i \big[ \Lambda \in \M_0 \big]  = 1$ for $i\in\{1,2\}$.
	It follows that $\Pb \in \Pcb_A(\nu)$ and $\Pb \big[ \Lambda \in \M_0 \big]  = 1$, so that $\Pb \in \Pcb_W(\nu)$.

	\medskip

		\hspace{2em} $(i.2)$ Let \Cref{assum:main1} and \Cref{assum:constant_case} hold true, we next show that $\Pcb_R(\nu)$ is convex for $\nu \in \Pc_p(\R^n)$.
	Let $(\Pb_1, \Pb_2) \in \Pcb_R(\nu)\times \Pcb_R(\nu)$, $\theta \in [0,1]$, and $\Pb :=\theta \Pb_1 + (1-\theta)\Pb_2.$
	Then $\Pb \in \Pcb_A(\nu)$ since $(\Pb_1, \Pb_2) \in \Pcb_R(\nu)\times\Pcb_R(\nu) \subset \Pcb_A(\nu)\times \Pcb_A(\nu)$. Let also $\varphi \in C_b(\R^n \x \R^d),$ $0\leq s \le t$ and $\zeta: \Omh \longrightarrow \R$ a bounded $\Fch_s$--measurable variable,
	then
	\[
                \E^{\Pb} \big[ \big|\E^{\hat \mu} \big[\widehat{S}^{\varphi,\mu}_t \zeta \big]-\E^{\hat \mu}\big[\widehat{S}^{\varphi,\mu}_s \zeta \big] \big|\big]
		=
		\theta \E^{\Pb_1} \big[ \big|\E^{\hat \mu} \big[\widehat{S}^{\varphi,\mu}_t \zeta \big]-\E^{\hat \mu}\big[\widehat{S}^{\varphi,\mu}_s \zeta \big] \big|\big]
		+
		(1-\theta) \E^{\Pb_2} \big[ \big|\E^{\hat \mu} \big[\widehat{S}^{\varphi,\mu}_t \zeta \big]-\E^{\hat \mu}\big[\widehat{S}^{\varphi,\mu}_s \zeta \big] \big|\big]
		= 0.
	\]
	By considering a countable dense family of $\varphi$, $0\leq s \le t$ and $\zeta$,
	it follows that for $\Pb$--a.e. $\omb \in \Omb$, 
	$(\widehat{S}^{\varphi,\mu(\omb)}_t)_{t \in [0,T]}$ is an $(\Fh,\muh(\omb))$--martingale for all $\varphi \in C_b^2(\R^n \x \R^d)$.
	This proves that $\P\in \Pcb_R(\nu)$.

\medskip

	\hspace{2em} 	$(i.3)$ Take $\nu \in \Pc_{p}(\R^n)$,
	we now show that $\Pcb_R(\nu)$ is closed under the $\Wc_p$--topology.
	First, from \Cref{lemma:estimates}, we have $\Pcb_R(\nu) \subset \Pc_p(\Omb).$ 
	Let $(\Pb_m)_{m \ge 1} \subset \Pcb_R(\nu)$, and $\Pb \in \Pc(\Omb)$ be such that
	$\lim_m \Wc_p (\Pb_m,\Pb)=0.$
	Then $\Pb \in \Pc_p(\Omb)$.
	
\medskip
	
	Let $f \in C_b^2(\R^n \x \R^n \x \R^d \x \R^\ell)$ and $\varphi \in C_b^2(\R^n \x \R^d)$,
	by \Cref{assum:main1},
	there exists some constant $C>0$ such that for all $(\omb,\hat \om) \in \Omb \x \Omh$ and $t \in [0,T]$ 
	\begin{align} \label{eq:Smartingale_growth}
		\big| \Sb^{f}_t(\omb) \big|
		\le
		C\bigg( 1 + \| X_{t \wedge \cdot}(\omb) \|^p + \int_{\Cc^n} \|\xb\|^p \mu(\omb)(\mathrm{d} \xb) + \iint_{[0,T]\times A} \rho(a_0,a)^p \Lambda_r(\omb) (\mathrm{d}a) \mathrm{d}r \bigg),
	\end{align}
	and
	\begin{align} \label{eq:Mmartingale_growth}
		\big| \widehat{S}^{\varphi,\mu(\omb)}_t(\hat \om) \big|
		\le
		C\bigg( 1 + \big\| \Yh_{t \wedge \cdot}(\hat \om) \big\|^p + \int_{\Cc^n} \|\xb\|^p \mu(\omb)(\mathrm{d} \xb) + \iint_{[0,T]\times A} \rho(a_0,a)^p \Lambdah_r(\tilde \om) (\mathrm{d}a) \mathrm{d}r \bigg).
	\end{align}
	Let $0\leq s \le t$, $\zeta: \Cc^n \x \Cc^n \x \M \x \Cc^d \longrightarrow \R$ and $\phi: \Cc^n \x \Cc^n \x \M \x \Cc^d \x \Cc^\ell \x \Pc(\Omb) \longrightarrow \R$ be two bounded continuous functions. Using the regularity of the coefficient functions $(b,\sigma,\sigma_0)$, together with \eqref{eq:Smartingale_growth} and \eqref{eq:Mmartingale_growth}, it follows that
	\begin{align*}
		0&=\lim_{m\rightarrow\infty} \E^{\Pb_m} \big[ \big|\E^{\hat \mu} \big[\widehat{S}^{\varphi,\mu}_t \zeta \big(\Xh_{s \wedge \cdot},\Yh_{s \wedge \cdot}, \Lambdah^s, \Wh_{s \wedge \cdot} \big) \big]-\E^{\hat \mu}\big[\widehat{S}^{\varphi,\mu}_s \zeta\big(\Xh_{s \wedge \cdot},\Yh_{s \wedge \cdot}, \Lambdah^s, \Wh_{s \wedge \cdot} \big) \big] \big|\big]
		\\
		&=
		\E^{\Pb} \big[ \big|\E^{\hat \mu} \big[\widehat{S}^{\varphi,\mu}_t \zeta \big(\Xh_{s \wedge \cdot},\Yh_{s \wedge \cdot}, \Lambdah^s, \Wh_{s \wedge \cdot} \big) \big]-\E^{\hat \mu}\big[\Mh^{\varphi,\mu}_s \zeta\big(\Xh_{s \wedge \cdot},\Yh_{s \wedge \cdot}, \Lambdah^s, \Wh_{s \wedge \cdot} \big) \big] \big|\big],
	\end{align*}
	and 
	\begin{align*}
		0&=\lim_{m\to\infty} \big|\E^{\Pb_m}[\Sb^f_t\phi \big(X_{s \wedge \cdot},Y_{s \wedge \cdot},\Lambda^s, W_{s \wedge \cdot}, B_{s \wedge \cdot}, \muh_s \big)]
		-
		\E^{\Pb_m}[\Sb^f_s\phi \big(X_{s \wedge \cdot},Y_{s \wedge \cdot}, \Lambda^s, W_{s \wedge \cdot}, B_{s \wedge \cdot}, \muh_s \big)]
		\big|
		\\
		&=
		\big|\E^{\P}[\Sb^f_t\phi \big(X_{s \wedge \cdot},Y_{s \wedge \cdot}, \Lambda^s, W_{s \wedge \cdot}, B_{s \wedge \cdot}, \muh_s \big)]
		-
		\E^{\Pb}[\Sb^f_s\phi \big(X_{s \wedge \cdot},Y_{s \wedge \cdot}, \Lambda^s, W_{s \wedge \cdot}, B_{s \wedge \cdot}, \muh_s \big)]
		\big|.
	\end{align*}
	This implies that for $\Pb$--a.e. $\omb \in \Omb$, 
	$(\Mh^{\varphi,\mu(\omb)}_t)_{t \in [0,T]}$ is an $(\Fh,\muh(\omb))$--martingale for all $\varphi \in C_b^2(\R^n \x \R^d)$,
	and $(\Sb^f_t)_{t \in [0,T]}$ is an $(\Fb,\Pb)$--martingale for all $f \in C_b^2(\R^n \x \R^n \x \R^d \x \R^\ell)$.

\medskip
	
	Finally, it is straightforward to check all the other conditions in \Cref{def:relaxed_ctrl_rule}, and we 
	can conclude that $\Pb \in \Pcb_R(\nu)$.
	
\medskip
 
	$(ii)$ Fix $\nu \in \Pc_{p}(\R^n)$. 
	First, one has clearly $V_S(\nu) \le V_W(\nu)$.
	Furthermore, for any $\Pb \in \Pcb_W(\nu)$, by \Cref{prop:approximation} and under condition $\ell \ge 1$,
	there is a sequence of probability measures $(\Pb^m)_{m \ge 1} \subset \Pcb_S(\nu)$ such that
	\begin{equation} \label{eq:Strong2Weak_var}
		\lim_{m \rightarrow\infty} 
		\Lc^{\Pb^m} \big(X,Y,  \Lambda,  W,B,\muh,\delta_{(\mub_t,\alpha_t)}(\mathrm{d} \nub, \mathrm{d}a)\mathrm{d}t \big)
		=
		\Lc^{\Pb} \big(X,Y,  \Lambda, W,B,\muh,  \delta_{(\mub_t,\alpha_t)}(\mathrm{d} \nub, \mathrm{d}a)\mathrm{d}t \big),
	\end{equation}
	in $\Pc_p\big( \Omb \x \M(\Pc(\Cc^n \x A) \x A)\big)$ under $\Wc_p$.
	This implies in particular that $\Pb^m \longrightarrow \Pb$ in $\Pc_p(\Omb)$ under $\Wc_p$.

\medskip
	Besides, although $\Pb \longmapsto J(\Pb)$ is not continuous in general (see \Cref{rem:nonContinuityJ}),
	the convergence in \eqref{eq:Strong2Weak_var} is stronger than the convergence $\Pb^m \longrightarrow \Pb$.
	With the growth and lower semi--continuity conditions of $L$ and $g$ in \Cref{assum:main1},
	and by a slight extension of \cite[Lemma 4.1]{lacker2017limit}, 
	the convergence \eqref{eq:Strong2Weak_var} implies that
	\[
		V_S(\nu) 
		\ge
		\lim_{m \to \infty} J(\Pb^m) 
		\ge
		J(\Pb).
	\]
	It follows that $V_S(\nu) = V_W(\nu)$.
	
\medskip
	
	When $\ell = 0$, using \Cref{prop:approximation}, 
	it is enough to  consider a convex combination of strong control rules and apply the same argument as above to conclude the proof.

\medskip

	$(iii)$ We assume here that $A \subset \R^j$,  $\nu \in \Pc_{p^\prime}(\R^n)$.
	It is enough to use \Cref{Equivalence-Proposition_General} to deduce that $\Pcb_W(\nu)$ is dense in $\Pcb_R(\nu)$ with respect to $\Wc_p$.
	Next, under \Cref{assum:constant_case}, together with the growth condition of $L$ and $g$ in \Cref{assum:main1},
	$\Pb \longmapsto J(\Pb)$ is lower semi--continuous (see \Cref{rem:ContinuityJ}) on $\Pc_p(\Omb)$.
	This is enough to prove that $V_W(\nu) = V_R(\nu)$. 
	
\medskip
	
	Finally, when $L$ and $g$ are continuous, under \Cref{assum:main1} and \Cref{assum:constant_case}, 
	$\Pb \longmapsto J(\Pb)$ is continuous on $\Pc_p(\Omb)$.
	Let $(\Pb^m)_{m \ge 1} \subset \Pcb_R(\nu)$ be a sequence such that
	\[
		\lim_{m \to \infty} J(\Pb^m) = V_R(\nu) < \infty.
	\]
	The coercivity condition \eqref{eq:cond_coercive} in \Cref{assum:main1} ensures that 
	$(\Pb^m)_{m \ge 1}$ is relatively compact w.r.t. $\Wc_p$ (see also \Cref{Propostion:convergence} below for a more detailed argument).
	By the closedness of $\Pcb_R(\nu)$,
	it follows that there exists $\Pb \in \Pcb_R(\nu)$, such that $\Wc_p(\Pb^m, \Pb) \longrightarrow 0$, possibly along a subsequence.
	Together with the continuity of $J: \Pc_p(\Omb) \longrightarrow \R$, this implies that $\Pb$ is an optimal relaxed control rule.
	\qed

\subsection{Proof of Theorem \ref{thm:limit} and Proposition \ref{Proposition:Continuity-Existence}}
	Based on the equivalence result and the closedness property of $\Pcb_R(\nu)$ in \Cref{thm:equivalence},
	we can provide the proof of the limit theory result in \Cref{thm:limit}
	and the continuity result in \Cref{Proposition:Continuity-Existence}.

\subsubsection{Approximation of McKean--Vlasov SDEs by large population SDEs}
    
	We show in this section that, for any control $\alpha \in \Ac_p(\nu)$ and the controlled process $X^{\alpha}$ defined in \eqref{eq:MKV_SDE},
	they can be approximated by a large population controlled SDE $(X^{\alpha,1}, \dots, X^{\alpha, N})$ as in \eqref{eq:N-agents_StrongMV_CommonNoise}. 
	Let us enforce \Cref{assum:main1}, and assume that $A \subset \R^j$ for some $j \ge 1$.
	
\medskip
	
	Recall from \Cref{subsec:strong_form} that $\Om := \R^n \x \Cc^d \x \Cc^{\ell}$ is equipped with the canonical element $(X_0, W, B)$, the canonical filtration $\F$ and a sub--filtration $\G$.
	We consider a probability measure $\P_\star$, under which $X_0$, $W$, $B$ are mutually independent, $(W,B)$ is an $\F$--Brownian motion, and $X_0 \sim \Uc[0,1]$.
	In particular, the probability space $(\Om, \Fc_0, \P_\star)$ is rich enough to support an $\R^n$--valued random variable of any distribution.
	Let $\xi$ be an $\Fc_0$--measurable random variable such that $\E[|\xi|^{p}] < \infty$, 
	$\alpha$ be an $\F$--predictable process satisfying the integrability condition \eqref{eq:alpha_Hp}.
	We denote by $X^{\xi, \alpha}$ the unique strong solution of the controlled McKean--Vlasov SDE
	\begin{align} \label{eq:initial_SDE}
		X^{\xi, \alpha}_t
		=
		\xi
		+ 
		\int_0^t \!\! b \big(r,X^{\xi, \alpha},\mub^{\xi, \alpha}_r, \alpha_r \big) \mathrm{d}r 
		+ 
		\int_0^t \!\! \sigma \big(r,X^{\xi, \alpha},\mub^{\xi, \alpha}_r, \alpha_r \big) \mathrm{d}W_r 
		+ 
		\int_0^t \!\! \sigma_0 \big(r,X^{\xi, \alpha},\mub^{\xi, \alpha}_r, \alpha_r \big) \mathrm{d}B_r,
		~\P_\star\mbox{--a.s.},
	\end{align}
	with $\mub^{\xi, \alpha}_r := \Lc^{\P_\star}(X^{\xi, \alpha}_{r \wedge \cdot}, \alpha_r) | \Gc_r)$, $\P_\star$--a.s. and satisfying $\E^{\P_\star} \big[ \|X^{\xi,\alpha}\|^p \big]< \infty.$
	As for \eqref{eq:MKV_SDE}, $X^{\xi, \alpha}$ is an $\F^\star$--adapted continuous process.

\medskip
Given in addition a $\G$--optional $\Pc(\Cc^n \x A)$--valued process $\mub = (\mub_t)_{t \in [0,T]}$ satisfying the integrability condition 
	\begin{equation} \label{eq:cond_mub}
		\E^{\P_\star} \bigg[ \iiint_{[0,T]\times\Cc^n \x A} ( \| \mathbf{x} \|^p + \| a - a_0 \|^p ) \mub_t(\mathrm{d} \mathbf{x}, \mathrm{d} a) \mathrm{d}t \bigg] < \infty,
	\end{equation}
	we denote by $X^{\xi, \mub, \alpha}$ the unique solution of the standard SDE
	\begin{equation} \label{eq:SDE_mu}
		X^{\xi,\mub, \alpha}_t
		=
		\xi
		+ 
		\int_0^t b \big(r, X^{\xi,\mub, \alpha}, \mub_r, \alpha_r \big) \mathrm{d}r 
		+ 
		\int_0^t \sigma \big(r, X^{\xi,\mub, \alpha}, \mub_r, \alpha_r \big) \mathrm{d}W_r
		+ 
		\int_0^t \sigma_0 \big(r, X^{\xi,\mub, \alpha}, \mub_r, \alpha_r \big) \mathrm{d}B_r,
		\; \P_\star\mbox{--a.s.},
	\end{equation}
	with $\E^{\P_\star} \big[ \|X^{\xi,\mub,\alpha}\|^p \big]< \infty.$
	In above, $X^{\xi, \mub, \alpha}$ is defined as an $\F$--adapted process with continuous paths, $\P_\star$-a.s.
	In particular, one has $X^{\xi,\mub^{\xi,\alpha}, \alpha}=X^{\xi, \alpha}$, $\P_\star$--a.s. and
	\[
		\Lc^{\P_\star} ( X^{\xi, \alpha},W,B) = \Lc^{\P_\star} (X^{\xi^\prime, \alpha},W,B),\; \mbox{and}\;
		\Lc^{\P_\star} ( X^{\xi, \mub, \alpha},W,B) = \Lc^{\P_\star} (X^{\xi^\prime, \mub,\alpha},W,B),\; \mbox{whenever}\; 
		\Lc^{\P_\star} (\xi) = \Lc^{\P_\star} (\xi^\prime).
	\]

	\begin{lemma} \label{lemm:continuity_initial-variable}
		Let $(\xi^m)_{m \ge 0}$ be a sequence of $\Fc_0$--measurable random variables such that 
		\[
			\lim_{m \to \infty} \Wc_p \big(\P_\star \circ (\xi^m)^{-1}, \P_\star \circ (\xi^0)^{-1} \big) 
			=
			 0,
		\] 
		and $\sup_{m \ge 0} \E^{\P_\star} [|\xi^m |^{p^\prime}]< \infty.$
		Let $\phi: [0,T] \x \R^n \x \Cc^d \x \Cc^\ell \longrightarrow A$ be a bounded continuous function,
		and $(\alpha^m)_{m \ge 0}$ be defined by $\alpha^m_t := \phi(t, \xi^m, W_{t \wedge \cdot}, B_{t \wedge \cdot})$ for all $t \in [0,T]$.	
		Then, for each $t \in [0,T]$, we have
		\[
			\lim_{m \to \infty}
			\E^{\P_\star} \Big[
				\Wc_p \Big( 
					\Lc^{\P_\star} \big( X^{\xi^m, \alpha^m}_{t \wedge \cdot}, \alpha^m_t \big | \Gc_t \big),
					\Lc^{\P_\star} \big(X^{\xi^0, \alpha^0}_{t \wedge \cdot}, \alpha^0_t \big| \Gc_t \big) 
				\Big)
			\Big]
			=
			0,
		\]
		and, for any fixed $\mub = (\mub_t)_{t \in [0,T]}$ satisfying \eqref{eq:cond_mub},
		\[
			\lim_{m \to \infty}
			\E^{\P_\star} \Big[
				\Wc_p \Big( 
					\Lc^{\P_\star} \big( X^{\xi^m, \mub, \alpha^m}_{t \wedge \cdot}, \alpha^m_t \big | \Gc_t \big),
					\Lc^{\P_\star} \big(X^{\xi^0, \mub, \alpha^0}_{t \wedge \cdot}, \alpha^0_t \big| \Gc_t \big) 
				\Big)
			\Big]
			=
			0.
		\]
	\end{lemma}
	\begin{proof}
	We will only prove the first convergence result, since the second follows by almost the same arguments.

\medskip
	
	First, without loss of generality, one can use Skorokhod's representation theorem and assume that $\lim_{n \to \infty} \xi^n = \xi^0$, $\P_\star$--a.s. Then, using the Lipschitz properties and the polynomial growth of the coefficient functions, we have using classical arguments (see notably $Step \; 1$ of the proof of \Cref{Lemm:app-Weak_Relaxed}), that there exists a constant $C>0$ such that, for $m\ge 1$,
	\begin{align} \label{ineq:McKV_initialValue-dependence}
		\E^{\P_\star} \bigg[
		\sup_{t \in [0,T]}|X_t^{\xi^m,\alpha^m}-X_t^{\xi^0,\alpha^0}|^p
		\bigg]
		\le 
		C \bigg(
		\E^{\P_\star} \big[|\xi^m-\xi^0|^p \big]
		+
		\E^{\P_\star} \bigg[ \int_0^T  \big|\alpha^m_t -\alpha^0_t \big|^p \mathrm{d}t\bigg] + C_m
		\bigg),
	\end{align}
	where
	\begin{align*}
		C_m:= 
		\E^{\P_\star} \bigg[\int_0^T \big|(b,\sigma,\sigma_0)\big(r,X^{\xi^0,\alpha^0},\mub^{\xi^0,\alpha^0}_r,\alpha^m_r \big)-(b,\sigma,\sigma_0)\big(r,X^{\xi^0,\alpha^0},\mub^{\xi^0,\alpha^0}_r,\alpha^0_r \big) \big|^p \mathrm{d}r \bigg].
	\end{align*}
	Next, since $\sup_m \E^{\P_\star}[|\xi^m|^{p^\prime}]< \infty$, for some $p^\prime > p$,
	then $(|\xi^m-\xi^0|^p)_{m \ge 1}$ is $\P_\star$--uniformly integrable and it follows that $\lim_{m \to \infty} \E^{\P_\star} \big[|\xi^m-\xi^0|^p \big]=0$.
	Moreover, since $\phi: [0,T] \x \R^n \x \Cc^d \x \Cc^\ell \longrightarrow A$ is bounded continuous,
	we obtain that
	\[
		\lim_{m \to \infty}  |\alpha^m_t - \alpha^0_t| = \lim_{m \to \infty}  C_m  = 0,
		\;\mbox{and hence}\;
		\lim_{m \to \infty}
		\E^{\P_\star} \bigg[
		\sup_{t \in [0,T]} \big|X_t^{\xi^m,\alpha^m}-X_t^{\xi^0,\alpha^0} \big|^p
		\bigg]
		=
		0.
	\]
	To conclude, it is enough to notice that, as $m \longrightarrow 0$,
	\[
		\E^{\P_\star} \Big[
		\Wc_p \Big( 
			\Lc^{\P_\star} \big( X^{\xi^m, \alpha^m}_{t \wedge \cdot}, \alpha^m_t \big | \Gc_t \big),
			\Lc^{\P_\star} \big(X^{\xi^0, \alpha^0}_{t \wedge \cdot}, \alpha^0 \big| \Gc_t \big) 
		\Big)
		\Big]
		\le
	        \E^{\P_\star} \Big[
			\big| X_{t\wedge \cdot}^{\xi^m,\alpha^m}-X_{t\wedge \cdot}^{\xi^0,\alpha^0} \big|^p
		\Big]^{1/p}
		+
		\E^{\P_\star} \big[  \big|\alpha^m_t - \alpha^0_t \big|^p \big]^{1/p}
		\longrightarrow 0.
	\]
	\end{proof}

	To proceed, let us consider, for each $N \ge 1$, the space $\Om^N := (\R^n)^N \x (\Cc^d )^N \x \Cc^{\ell}$ defined in \Cref{subsec:N_controlpb}, equipped with canonical elements $(X^1_0, \dots, X^N_0, W^1, \dots, W^N)$ and canonical filtration $\F^N$.
	On $\Om^N$, we also introduce a sub--filtration
	\[
		\G^N := (\Gc^N_t)_{t \in [0,T]},\; 
		\mbox{with}\;
		\Gc^N_t := \sigma (B_s:s \in [0,t]).
	\]
	Given $\nu \in \Pc_p(\R^n)$ and a sequence $(\nu^i)_{i \ge 1} \subset \Pc_p(\R^n)$, 
	we take the first $N$ elements to define $\P^N_{\nu}$ on $\Om^N$, under which $X^i_0 \sim \nu^i$, and $B$, $W^i$ are standard Brownian motions, and $(X^1_0, \dots, X^N_0, W^1, \dots, W^N, B)$ are mutually independent.

\medskip
	
	Further, in Lemma \ref{lemm:continuity_initial-variable},
	we keep using the bounded continuous function $\phi$ to define the control process $\alpha$.
	Together with an initial random variable $\xi \sim \nu$, one obtain a $\G$--optional process $\mub^{\xi, \alpha}$ in $\Om$.
	Notice that in $\Om$, the process $\mub^{\xi, \alpha}$ is a functional of the common noise process $B$,
	one can then extend it as a $\G^N$--optional process in $\Om^N$ while keeping the same notation for simplicity.

\medskip
	
	Finally, with the same bounded continuous function $\phi: [0,T] \x \R^n \x \Cc^d \x \Cc^\ell \longrightarrow A$ in \Cref{lemm:continuity_initial-variable},
	we introduce the control processes $(\alpha^1, \dots, \alpha^N)$ by 
	$\alpha^i_t := \phi(t,X^i_0,W^i_{t \wedge \cdot}, B_{t \wedge \cdot})$,
	and then define a sequence of processes $\Xb^{\alpha^i,i}$, $i=1, \dots, N$, by
	\begin{equation} \label{eq:def_Xi}
		\Xb^{\alpha^i, i}_t
		= 
		X^i_0
		+ 
		\int_0^t b\big(r, \Xb^{\alpha^i,i}, \mub^{\xi,\alpha}_r,  \alpha^i_r \big) \mathrm{d}r 
		+ 
		\int_0^t \sigma \big(r, \Xb^{\alpha^i,i}, \mub^{\xi,\alpha}_r,  \alpha^i_r \big) \mathrm{d}W^i_r 
		+ 
		\int_0^t  \sigma_0 \big(r, \Xb^{\alpha^i,i}, \mub^{\xi,\alpha}_r,  \alpha^i_r \big) \mathrm{d}B_r,\;\P^N_{\nu}\mbox{--a.s.}
	\end{equation}
	Notice that the above SDE is almost the same as \eqref{eq:SDE_mu}, except that we use here $(X^i_0, \mub^{\xi, \alpha}, W^i)$ instead of $(\xi, \mub, W)$.

	\begin{lemma} \label{lemma of chaos-1}
		Assume that $\nu$ and $(\nu^i)_{i \ge 1}$ satisfy
		\[
			\lim_{N \to \infty}
			\Wc_{p} \bigg(\frac{1}{N}\sum_{i=1}^N \nu^i, \nu \bigg)
			=
			0,\; \mbox{\rm and}\;
			\sup_{N \ge 1} 
			\frac{1}{N}\sum_{i=1}^N \int_{\R^n} | x |^{p^{\prime}} \nu^i(\mathrm{d} x ) 
			< 
			\infty.
		\]
		Then
		\begin{equation} \label{eq:empir_dist_cvg}
			\lim_{N \to \infty} \E^{\P^N_{\nu}} 
			\bigg[
			\int_0^T \Wc_p \big( 
				\overline \varphi^{N}_t,
				\mub^{\xi,\alpha}_t
			\big) \mathrm{d}t
			\bigg]
			= 0,
			\;\mbox{\rm with}\;\;
			\overline \varphi^{N}_t(\mathrm{d} \xb, \mathrm{d}a)
			:=
			\frac{1}{N}\sum_{i=1}^N \delta_{ \big(\Xb^{\alpha^i,i}_{t \wedge \cdot},  \alpha^i_t \big)}(\mathrm{d} \xb, \mathrm{d}a).
		\end{equation}
	\end{lemma}

	\begin{proof} 
	Notice that to prove \eqref{eq:empir_dist_cvg}, it is enough to prove that,  in the space $ (\M \big(\Pc(\Cc^n \x A) \x \Pc(\Cc^n \x A) \big),\Wc_{p})$,
	\[
		\Lambdab^N(\mathrm{d} \nub,\mathrm{d} \nub^{\prime},\mathrm{d}t)
		:=
		\E^{\P^N_\nu} \Big[ \delta_{\big(\overline \varphi^{N}_t, \mub^{\xi,\alpha}_t \big)}(\mathrm{d} \nub,\mathrm{d} \nub^{\prime})\mathrm{d}t \Big]
		\underset{N \to \infty}{\longrightarrow}
		\Lambdab^0(\mathrm{d}\nub,\mathrm{d}\nub^{\prime},\mathrm{d}t)
		:=
		\E^{\P_\star} \Big[ \delta_{\mub^{\xi,\alpha}_t}(\mathrm{d} \nub)\delta_{\mub^{\xi,\alpha}_t}(\mathrm{d} \nub^{\prime})\mathrm{d}t \Big].
	\]
	
	First, by a trivial extension of \Cref{lemma:estimates}, there exists a constant $C$ independent of $i \ge 1$, s.t.
	\begin{align*}
		\E^{\P^N_\nu} \bigg[
		\frac{1}{N}\sum_{i=1}^N \sup_{[0,T]} \big|\Xb^{\alpha^i,i}_t\big|^{p^\prime}
		+
		\int_0^T |a_0 -\alpha^i_t|^{p^\prime} \mathrm{d}t
		\bigg] 
		&\le 
		\frac{1}{N}\sum_{i=1}^N 
		C
		\bigg(1+\E^{\P^N_\nu}[|X^i_0|^{p^\prime}] + \E^{\P_\star}[|\xi|^{p^\prime}] +\E^{\P^N_\nu} \bigg[\int_0^T |a_0 - \alpha^i_t|^{p^\prime} \mathrm{d}t \bigg] \bigg)
		\\
		&\le C\bigg(1+\int_{\R^n}|x|^{p^\prime}\frac{1}{N}\sum_{i=1}^N\nu^i(\mathrm{d} x) \bigg) < \infty,
	\end{align*}
	where the second inequality follows by the fact that $\phi$ is bounded.
	Since $p^\prime > p$, it follows by \cite[Proposition-A.2.]{carmona2014mean} and \cite[Proposition-B.1.]{carmona2014mean}
	that $(\Lambdab^N)_{N \in \N}$ is relatively compact in $(\M \big(\Pc(\Cc^n \x A) \x \Pc(\Cc^n \x A) \big),\Wc_p)$.
	
\medskip

	Let $(N_m)_{m \ge 1}$ be a subsequence such that 
	$\Lambdab^{N_m} \longrightarrow_{m\to\infty} \Lambdab^{\infty}$ under $\Wc_p$.
	We only need to show that $\Lambdab^{\infty} = \Lambdab^0$, or equivalently (see \Cref{Prop-identification_probability}), that
	for every $k \ge 1$, 
	$g_1,\dots,g_k \in C_b(\Cc^n \x A)$, $f \in C_b([0,T] \x \Pc(\Cc^n \x A))$, we have
	\begin{equation} \label{eq:Lambda_cvg}
		\int_0^T \int_{\Pc(\Cc^n \x A)^2} \prod_{i=1}^k \langle g_i, \nub \rangle f(t,\nub^{\prime}) \Lambdab^{\infty} (\mathrm{d}\nub,\mathrm{d}\nub^{\prime},\mathrm{d}t)
		=
		\int_0^T \int_{\Pc(\Cc^n \x A)^2} \prod_{i=1}^k \langle g_i, \nub \rangle f(t,\nub^{\prime}) \Lambdab^0 (\mathrm{d} \nub,\mathrm{d} \nub^{\prime},\mathrm{d}t).
	\end{equation}
	In the following, we provide the proof of \eqref{eq:Lambda_cvg} for the case $k=2$, since the proof for the general case is identical.
	
\medskip

	Notice that $\mub^{\xi,\alpha}$ is $\G^N$--adapted, 
	and $X^{\alpha^i, i}$ depends only on $(X^i_0, W^i, B)$. It therefore follows that 
	$(\Xb^{\alpha^i,i}_{t \wedge \cdot},\alpha^i_t)$ and $(\Xb^{\alpha^j,j}_{t \wedge \cdot},\alpha^j_t)$ are conditionally independent given the $\sigma$--algebra $\Gc^N_t,$ for all $t \in [0,T].$ 
	Thus for $i \neq j$,
	\[
		\E^{\P^N_{\nu}} 
		\big[
			g_1 \big(\Xb^{\alpha^i,i}_{t \wedge \cdot},\alpha^i_t \big) 
			g_2 \big( \Xb^{\alpha^j,j}_{t \wedge \cdot},\alpha^j_t \big) 
			f \big(t,\mub^{\xi,\alpha}_t \big)
		\big]
		=
		\E^{\P^N_\nu}
		\Big[ 
			\E^{\P^{N}_\nu} \big[g_1(\Xb^{\alpha^i,i}_{t \wedge \cdot},\alpha^i_t) \big| \Gc^N_t \big]
			\E^{\P^{N}_\nu} \big[g_2(\Xb^{\alpha^j,j}_{t \wedge \cdot},\alpha^j_t) \big| \Gc^N_t \big]
			f \big( t,\mub^{\xi,\alpha}_t \big)
		\Big].
	\]
	Since $f$, $g_1$, and $g_2$ are bounded, it follows that
	\begin{align*}
		&\
		\int_0^T \int_{\Pc(\Cc^n \x A)^2} \langle g_1, \nub \rangle \langle g_2, \nub \rangle f(t,\nub^{\prime})  \Lambdab^{\infty} (\mathrm{d}\nub,\mathrm{d}\nub^{\prime},\mathrm{d}t) \\
		=&\
		\lim_{m \to \infty} 
		\int_0^T \frac{1}{N^2_m}\sum_{i,j=1}^{N_m} 
			\E^{\P^{N_m}_\nu} \Big[
				g_1 \big(\Xb^{\alpha^i,i}_{t \wedge \cdot},\alpha^i_t \big) 
				g_2 \big(\Xb^{\alpha^j,j}_{t \wedge \cdot},\alpha^j_t \big) 
				f \big(t,\mub^{\xi,\alpha}_t \big) 
			\Big]
		\mathrm{d}t
		\\
		=&\
		\lim_{m \to \infty} 
		\int_0^T \frac{1}{N^2_m}\sum_{i,j=1}^{N_m}
			\E^{\P^{N_m}_\nu}\Big[ 
				\E^{\P^{N_m}_\nu} \big[g_1(\Xb^{\alpha^i,i}_{t \wedge \cdot},\alpha^i_t) \big|\Gc^N_t \big]
				~\E^{\P^{N_m}_\nu} \big[g_2(\Xb^{\alpha^j,j}_{t \wedge \cdot},\alpha^j_t) \big| \Gc^N_t \big]
				~f(t,\mub^{\xi,\alpha}_t) 
			\Big]
		\mathrm{d}t
		\\
		=&\
		\lim_{m \to \infty} 
		\int_0^T \int_{\Pc(\Cc^n \x A)^2} 
			\langle g_1, \nub \rangle 
			\langle g_2, \nub \rangle 
			f(t, \nub^{\prime}) 
			\E^{\P^{N_m}_\nu}\Big[
				\delta_{\big(\frac{1}{N_m}\sum_{i=1}^{N_m}
					\Lc^{\P^{N_m}_\nu}(\Xb^{\alpha^i,i}_{t \wedge \cdot},\alpha^i_t|\Gc^N_t),\mub^{\xi,\alpha}_t \big)}
				(\mathrm{d}\nub,\mathrm{d} \nub^{\prime}) 
			\Big]
		\mathrm{d}t.
	\end{align*}
	Let $U^N$ be a random variable on $(\Om, \Fc_0, \P_\star)$ such that $\Lc^{\P_\star}(U^N)=\frac{1}{N}\sum_{i=1}^N \nu^i.$ If we note $\alpha^{\star,N}_t:=\phi(t,U^N,W_{t \wedge \cdot},B_{t \wedge \cdot}),$ we have, from \Cref{lemm:continuity_initial-variable} that, for all $t \in [0,T]$
	\begin{align*}
		&\ 
		\lim_{m \to \infty}
		\E^{\P^{N_m}_{\nu}} \bigg[
			\Wc_p \bigg(
				\frac{1}{N_m}\sum_{i=1}^{N_m}
					\Lc^{\P^{N_m}_\nu} \big(\Xb^{\alpha^i,i}_{t \wedge \cdot},\alpha^i_t \big|\Gc^N_t \big),
				\mub^{\xi,\alpha}_t 
			\bigg)
		\bigg]
		\\
		&=
		\lim_{m \to \infty}
		\E^{\P_\star} \Big[
			\Wc_p \Big( 
				\Lc^{\P_\star} \big(X^{U^N, \mub^{\xi,\alpha}, \alpha^{\star,N}}_{t \wedge \cdot},\alpha^{\star,N}_t \big|\Gc_t \big),
				~\Lc^{\P_\star} \big(X^{\xi, \mub^{\xi,\alpha}, \alpha}_{t \wedge \cdot}, \alpha_t \big| \Gc_t \big) 
				\Big)
		\Big]
		=
		0.
	\end{align*}
	Consequently
	\[
		\int_0^T \int_{\Pc(\Cc^n \x A)^2} \langle g_1, \nub \rangle \langle g_2, \nub \rangle f(t,\nub^{\prime})  \Lambdab^{\infty} (\mathrm{d}\nub,\mathrm{d}\nub^{\prime},\mathrm{d}t) 
		=
		\int_0^T \int_{\Pc(\Cc^n \x A)^2} \langle g_1, \nub \rangle \langle g_2, \nub \rangle f(t,\nub^{\prime})  \Lambdab^0 (\mathrm{d}\nub,\mathrm{d}\nub^{\prime},\mathrm{d}t),
	\]
	and the proof is concluded.
	\end{proof}

\vspace{0.5em} 

	Given a probability measure $\nu \in \Pc_p(\R^n)$ and a sequence $(\nu^i)_{i \ge 1} \subset \Pc_p(\R^n)$,
	we consider the probability spaces $(\Om, \Fc, \P_{\nu})$ and $(\Om^N, \Fc^N, \P^N_{\nu})$, introduced respectively in  \Cref{subsec:strong_form} and \Cref{subsec:N_controlpb}.	
	Let us fix a bounded continuous function $\phi: [0,T] \x \R^n \x \Cc^d \x \Cc^\ell \longrightarrow A$,
	and define a control process $\alpha := (\alpha_t)_{t \in [0,T]}$ on $(\Om,\Fc)$,
	and control processes $(\alpha^1, \dots, \alpha^N)$ on $(\Om^N,\Fc^N)$ by
	\begin{equation} \label{eq:def_alpha_phi}
		\alpha_t := \phi(t, X_0, W_{t \wedge \cdot}, B_{t \wedge \cdot}),
		\;
		\alpha^i_t := \phi(t, X^i_0, W^i_{t \wedge \cdot}, B_{t \wedge \cdot}),
		\;t \in [0,T],\; i = 1, \dots, N.
	\end{equation}
	Using the control process $\alpha$, $(X^{\alpha}, \mub^{\alpha})$ is defined by \eqref{eq:MKV_SDE} under $\P_{\nu}$.
	In particular, in the probability space $(\Om, \Fc, \P_\star)$, let $\xi \sim \nu$, and $(X^{\xi, \alpha}, \mub^{\xi, \alpha})$ be defined 
	by \eqref{eq:initial_SDE}. We have $\P_\star \circ (\mub^{\xi, \alpha})^{-1} = \P_{\nu} \circ (\mu^{\alpha})^{-1}$.
	Next, let $\xi$ be a random variable on $(\Om, \Fc, \P_\star)$ satisfying $\P_\star \circ \xi^{-1} = \nu$. We also naturally extend the $\G$--optional process $\mub^{\xi, \alpha}$ on $\Om$ into a $\G^N$--optional process on $\Om^N$. 
	Then with the bounded control processes $(\alpha^1, \dots, \alpha^N)$, 
	$(X^{\alpha,i} )_{i=1, \dots, N}$ is defined by \eqref{eq:N-agents_StrongMV_CommonNoise} under $\P^N_{\nu}$,
	and $(\Xb^{\alpha^i,i})_{i =1, \dots, N}$ is defined by \eqref{eq:def_Xi}.
	Recall also that
	\[
		\varphi^{N,X}_{t}(\mathrm{d} \xb) := \frac{1}{N}\sum_{i=1}^N \delta_{(X^{\alpha,i}_{t \wedge \cdot})}(\mathrm{d} \xb),\;
		\varphi^{N}_{t}(\mathrm{d} \xb, \mathrm{d}a) := \frac{1}{N}\sum_{i=1}^N \delta_{(X^{\alpha,i}_{t\wedge \cdot},\alpha^i_t )}(\mathrm{d} \xb, \mathrm{d}a),
		~\mbox{and}~
		\overline \varphi^{N}_t(\mathrm{d} \xb, \mathrm{d}a)
			:=
			\frac{1}{N}\sum_{i=1}^N \delta_{ \big(\Xb^{\alpha^i,i}_{t \wedge \cdot},  \alpha^i_t \big)}(\mathrm{d} \xb, \mathrm{d}a).
	\]	

	\begin{proposition} \label{propogation of chaos-2}
		Let $\alpha$ and $(\alpha^i)_{1\leq i\leq N}$ be defined in \eqref{eq:def_alpha_phi}, together with the Borel measurable function $\phi: [0,T] \x \R^n \x \Cc^d \x \Cc^\ell \longrightarrow A$.
		Assume that
		\[
			\alpha \in \Ac_p(\nu),
			~
			\sup_{N \ge 1} \frac{1}{N}\sum_{i=1}^N \int_{\R^n} | x |^{p^\prime} \nu^i(\mathrm{d} x) < \infty,
			~\mbox{\rm and}~
			\lim_{N \to \infty} \Wc_{p} \bigg(\frac{1}{N}\sum_{i=1}^N \nu^i,\nu \bigg)=0.
		\]
		Then
		\begin{equation}  \label{res_propogation_of_chas}
			\lim_{N \to \infty}
			\E^{\P^N_{\nu}} \bigg[\int_0^T\Wc_p(\varphi_t^{N}, \mub^{\xi, \alpha}_t) \mathrm{d}t \bigg]
			=
			0,
			~\mbox{\rm and}~
			\lim_{N \to \infty}
			\Lc^{\P^N_{\nu}} \big(\delta_{\varphi^{N}_{t} }(\mathrm{d} \nub)\mathrm{d}t, \varphi^{N,X} \big)
			= 
			\Lc^{\P_\nu} \big(\delta_{\mub^\alpha_t }(\mathrm{d} \nub)\mathrm{d}t, \mu^{\alpha} \big)
			~\mbox{\rm under}~\Wc_{p}.
		\end{equation}
            	Consequently
		\[
			V_S(\nu) 
			\le
			\liminf_{N \to \infty} V_S^N(\nu^1,\dots,\nu^N).
		\]
	\end{proposition}
	\begin{proof}
	$(i)$ 		
	Using \Cref{assum:main1}, together with Burkholder--Davis--Gundy inequality and Gronwall's lemma,
	it follows by classical arguments that
	there exist positive constants $K$, and $K^\prime$ such that for all $N \ge 1$, $i=1, \dots, N$ and $t \in [0,T]$
	\[
		\E^{\P_{\nu}^N} \bigg[ \sup_{r \in [0,t]} \big|X^{\alpha,i}_r-\Xb^{\alpha^i,i}_r \big|^p \bigg]
		\le
		K\E^{\P_{\nu}^N}\bigg[\int_0^t \Wc_p(\varphi_r^{N},\mub^{\xi,\alpha}_r)^p \mathrm{d}r\bigg]
		\le
		K^\prime\E^{\P^N_{\nu}} 
		\bigg[\int_0^t \Big(
			\Wc_p(\varphi_r^{N}, \overline \varphi_r^{N})^p
			+
			\Wc_p(\overline \varphi_r^{N},\mub^{\xi,\alpha}_r)^p
		\Big)
		\mathrm{d}r\bigg].
	\]
	Further, notice that
	\[
		\E^{\P^N_{\nu}}\big[ \Wc_p(\varphi_t^{N}, \overline \varphi_t^{N})^p \big]
		\le
		\frac{1}{N} \sum_{i=1}^N\E^{\P^N_{\nu}}\bigg[ \sup_{r \in [0,t]} |X^{\alpha,i}_r-\Xb^{\alpha^i,i}_r|^p \bigg]
		\le
		K\E^{\P^N_{\nu}} 
		\bigg[\int_0^t \Big(
			\Wc_p(\varphi_r^{N}, \overline \varphi_r^{N})^p
			+
			\Wc_p( \overline \varphi_r^{N},\mub^{\xi,\alpha}_r)^p
		\Big)
		\mathrm{d}r\bigg],
	\]
	it follows by  Gronwall's lemma and then by \Cref{lemma of chaos-1} that
	\[
		\lim_{N \to \infty}
		\E^{\P^N_{\nu}}\big[ \Wc_p(\varphi_t^{N}, \overline \varphi_t^{N})^p \big]
		\le
		\lim_{N \to \infty}
		K\E^{\P^N_{\nu}} \bigg[\int_0^t
		\Wc_p(\overline \varphi_r^{N},\mub^{\xi,\alpha}_r)^p \mathrm{d}r\bigg]
		=0,
		~\mbox{and thus}~
		\lim_{N \to \infty}
		\E^{\P^N_{\nu}}\bigg[ \int_0^T \Wc_p(\varphi_t^{N},\mub^{\xi,\alpha}_t)^p \mathrm{d}t\bigg]
		=
		0.
	\]
	As an immediate consequence, we also have
	\[
		\lim_{N \to \infty}
			\Lc^{\P^N_{\nu}} \big(\delta_{\varphi^{N}_{t} }(\mathrm{d} \nub)\mathrm{d}t, \varphi^{N,X} \big)
			= 
			\Lc^{\P_\nu} \big(\delta_{\mub^\alpha_t }(\mathrm{d} \nub)\mathrm{d}t, \mu^{\alpha} \big),
			~\mbox{under}~\Wc_{p}.
	\]

	$(ii)$ Let us now consider an arbitrary control process $\alpha \in \Ac_p(\nu)$, 
	so that there exists a Borel measurable function $\phi: [0,T] \x \R^n \x \Cc^d \x \Cc^\ell \longrightarrow A$ 
	such that $\alpha_t=\phi(t,\xi,W_{t \wedge},B_{t \wedge \cdot})$ for all $t \in [0,T],$ $\P_\nu$--a.s.
	Then there exists (see e.g. \cite[Proposition C.1.]{carmona2014mean}) 
	a sequence of bounded continuous functions $(\phi^m)_{m \ge 1} :  [0,T] \x \R^n \x \Cc^d \x \Cc^\ell \longrightarrow A$ 
	such that 
	\[
		\lim_{m \to \infty}
		\alpha^m_t 
		:=
		\lim_{m \to \infty}
		\phi^m(t,\xi,W_{t \wedge},B_{t \wedge \cdot})
		=
		\phi(t,\xi,W_{t \wedge},B_{t \wedge \cdot})
		=
		\alpha_t,\;\mathrm{d}\P_\nu \otimes \mathrm{d}t\;\mbox{--a.e.}
	\]
	Then, in the probability space $(\Om,\F,\Fc,\P_\nu)$, 
	it follows by standard arguments (see e.g.  the proof of \Cref{prop:approximation} or \Cref{Lemm:app-Weak_Relaxed}) that
	\[
		\Lim_{m \to \infty} 
		\E^{\P_\nu}\bigg[
			\sup_{t \in [0,T]} \big| X_t^{\alpha^m} - X_t^{\alpha} \big|^p 
		\bigg]
		=
		0,
		~\mbox{and}~
		\Lim_{m \to \infty}
		\E^{\P_{\nu}} \bigg[
			\int_0^T \Wc_p \big( \mub^{\alpha^m}_t, \mub^\alpha_t \big)^p \mathrm{d}t 
		\bigg]
		= 
		0.
	\]
	Finally, for each $m \ge 1$, consider the bounded continuous function $\phi^m$.
	For each $N \ge 1$, on the space $(\Om^N, \Fc^N, \P^N_{\nu})$,
	we can define control processes $(\alpha^{m,i})_{1\leq i\leq N}$ by 
	$\alpha^{m,i}_t:=\phi^m(t,X^i_0,W^i_{t \wedge \cdot},B_{t \wedge \cdot})$, $t\in[0,T]$, $i\in\{1,\dots,N\}$,
	and then define $(X^{\alpha^m,1},\dots,X^{\alpha^m,N})$ as the unique solution of \Cref{eq:N-agents_StrongMV_CommonNoise} 
	with control processes $(\alpha^{m,i})_{i =1, \dots, N}$.
	
	\medskip
	Define then 
	\[\varphi^{m, N, X}_{t}(\mathrm{d} \xb) := \frac{1}{N}\sum_{i=1}^N \delta_{(X^{\alpha^m,i}_{t \wedge \cdot})}(\mathrm{d} \xb)\;\mbox{and}\;\varphi^{m, N}_{t}(\mathrm{d} \xb, \mathrm{d}a) := \frac{1}{N}\sum_{i=1}^N \delta_{(X^{\alpha^m,i}_{t \wedge \cdot},\alpha^{m,i}_t )}(\mathrm{d} \xb, \mathrm{d}a).\]
	We have, thanks to \Cref{res_propogation_of_chas},
	\[\Lim_{N \to \infty} \Lc^{\P_{\nu}^N} \big( \delta_{\varphi^{m,N}_t}(\mathrm{d}m)\mathrm{d}t, \varphi^{m, N, X} \big) 
		= \Lc^{\P_{\nu}} \big( \delta_{\mu^{\alpha^m}_t}(\mathrm{d}m)\mathrm{d}t, \mu^{\alpha^m}  \big),\; \mbox{under}\; \Wc_p.\]
	It follows then	
	\begin{align*}
		J(\alpha)
		&=
		\E^{\P_{\nu}} \bigg[ \int_0^T \big \langle L(t,\cdot,\mub^\alpha_t),\mub^\alpha_t \big \rangle \mathrm{d}t 
		+
		\big\langle g(\cdot,\mu^\alpha_T), \mu^\alpha_T \big \rangle
		\bigg]\\
		&\le
		\lim_{m \to \infty}
		\E^{\P_{\nu}} \bigg[ \int_0^T \big \langle L(t,\cdot,\mub^{\alpha^m}_t), \mub^{\alpha^m}_t \big \rangle \mathrm{d}t 
		+
		\big \langle g(\cdot,\mu^{\alpha^m}_T), \mu^{\alpha^m}_T \big \rangle
		\bigg]
		\\
		&\le
		\lim_{m \to \infty} \lim_{N \to \infty}
		\E^{\P^N_{\nu}} \bigg[ \int_0^T \big \langle L(t,\cdot,\varphi_t^{m, N}),\varphi_t^{m,N} \big \rangle \mathrm{d}t 
		+
		\big \langle g(\cdot, \varphi_T^{m,N,X}), \varphi_T^{m,N,X} \big \rangle
		\bigg]
		\\
		&\le
		\lim_{m \to \infty} \lim_{N \to \infty}
		\frac{1}{N} \sum_{i=1}^N
		\E^{\P^N_{\nu}} \bigg[ 
			\int_0^T L \big(t,X^{\alpha^m,i},\alpha^{m,i}_t,\varphi_t^{m, N} \big) \mathrm{d}t 
			+
			g \big(X^{\alpha^m,i},\varphi_T^{m, N,X} \big)
		\bigg] 
		\le
		\Liminf_{N \to \infty} 
		V_S^N(\nu^1,\dots,\nu^N).
	\end{align*}
	By arbitrariness of $\alpha \in \Ac_p(\nu)$, it follows that
	$V_S(\nu) \le \Liminf_{N\to\infty} V_S^N(\nu^1,\cdot,\nu^N)$.
	\end{proof}

	Using exactly the same arguments and \Cref{lemm:continuity_initial-variable} 
	we can obtain the following result, whose proof is therefore omitted.

	\begin{proposition}   \label{propogation of chaos-2p}
		Assume that
		\[
			\sup_{m \ge 1} \int_{\R^n} | x |^{p^{\prime}} \nu^m(\mathrm{d} x) < \infty,
			~\mbox{\rm and}~
			\lim_{m \to \infty} \Wc_{p}(\nu^m,\nu)=0.
		\]
		Then with the control process $\alpha$ defined in \eqref{eq:def_alpha_phi},
		we have
		\[
			\lim_{m \to \infty}
			\Lc^{\P_{\nu^m}} \big(\delta_{\mub^{\alpha}_t}(\mathrm{d} \nub)\mathrm{d}t, \mu^{\alpha} \big)
			=
			\Lc^{\P_{\nu}} \big(\delta_{\mub^{\alpha}_t}(\mathrm{d} \nub) \mathrm{d}t, \mu^{\alpha} \big),
			~\mbox{\rm under}~\Wc_{p},
			~\mbox{\rm and consequently}~
			V_S(\nu) \le \liminf_{m \to \infty} V_S(\nu^m).
		\]
	\end{proposition}

\subsubsection{Tightness of the optimal control rules}

	Let us now stay in the context of \Cref{thm:limit} and prove that the set of optimal or $\varepsilon$--optimal control rules is tight.
	Recall that \Cref{assum:main1} and \Cref{assum:constant_case} hold true, 
	$A \subset \R^j$ for some $j \ge 1$, and both $L$ and $g$ are continuous in all their arguments.
	Let $N \ge 1$, $(\nu, \nu^1, \dots, \nu^N) \subset \Pc_p(\R^n)$, $\alpha \in \Ac(\nu)$ and $(\alpha^1, \dots, \alpha^N) \in \Ac^N(\nu_N)$.
	$\P^N(\alpha^1, \dots, \alpha^N)$ is a probability measure on $\Omb$ defined by \eqref{eq:def_PN}.

	\begin{proposition} \label{Propostion:convergence}
    
		$(i)$ In the context of {\rm\Cref{thm:limit}},
		Let $(\nu^i)_{i \ge 1} \subset \Pc_p(\R^n)$ satisfy
		$\sup_{N \ge 1} 
			\frac{1}{N}\sum_{i=1}^N \int_{\R^n}| x|^{p^\prime} \nu^i(\mathrm{d} x) < 
			\infty
		$
		and
		$(\Pb^N)_{N \ge 1} \subset \Pc_p(\Omb)$ satisfy \eqref{eq:eps_optimal_ctrl},
		then both $(\frac1N \sum_{i=1}^N \nu^i)_{ N \ge 1}$ and $(\Pb^N)_{N \ge 1}$ are relatively compact under $\Wc_{p}$.
		{\color{black}
		Moreover, for any converging subsequence $(\Pb^{N_m})_{m \ge 1}$, we have
		\[
			\lim_{m \to \infty} \Wc_p \bigg( \frac{1}{N_{m}}\sum_{i=1}^{N_{m}} \nu^i , \nu \bigg) = 0,
			~\mbox{\rm for some}\; \nu \in \Pc_p(\R^n),
			~\mbox{\rm and}~
			\lim_{m \to \infty} \Wc_{p} \big( \Pb^{N_{m}}, \Pb^\infty \big)=0,
			~\mbox{\rm for some}
			~\Pb^\infty \in  \Pcb_R(\nu).
		\]
		}

		\medskip

		$(ii)$ In the context of {\rm\Cref{Proposition:Continuity-Existence}}, 
		let  $(\varepsilon_m)_{m \ge 1} \subset \R_+$ be such that $\lim_{m \to \infty} \varepsilon_m = 0$,
		$(\Pb^m)_{m \ge 1}$ be a sequence such that
		\[
			\Pb^m \in \Pcb_R(\nu^m),
			~\mbox{\rm and}~
			J(\Pb^m) \ge V_S(\nu^m) - \varepsilon_m.
		\]
		Then the sequence $(\Pb^m)_{m \ge 1}$ is relatively compact,
		and moreover, any cluster point of $(\Pb^m)_{m \ge 1}$ belongs to $\Pcb_R(\nu)$.
	\end{proposition}
	\begin{proof}
	We will only consider $(i)$, since the proof of $(ii)$ is identical.
	
\medskip
	
	\underline{\it Tightness}: 
	To prove the tightness of $(\Pb^N)_{N \ge 1}$ under $\Wc_p$, we adapt the proof of \cite[Proposition 3.5.]{lacker2017limit} to our context.
	First, let us define control processes $(\alpha^{0,i})_{i \ge 1}$ by $ \alpha^{0,i}_t \equiv  a_0$ for all  $t \in [0,T]$ and $i \ge 1$, and denote  $\Pb^N_0:=\P^N(\alpha^{0,1}, \dots, \alpha^{0,N})$.
	By Lemma \ref{lemma:estimates},  there exist some constants $K$, $K^\prime>0$, such that for all $N \ge 1$
	\[
		J\big(\Pb^N_0\big) 
		\ge
		-K
		\bigg(1+\E^{\Pb^N_0} \bigg[ \sup_{t \in [0,T]}|X_t|^p \bigg]
		\bigg) 
		= 
		-K
		\bigg(
			1+\frac{1}{N}\sum_{i=1}^N\E^{\P^N_\nu} \bigg[\sup_{t \in [0,T]} \big| X^{\alpha^0,i}_t \big|^p \bigg]
		\bigg)
		\ge
		-K^\prime\bigg( 
			1+ \frac{1}{N}\sum_{i=1}^N \int_{\R^n} |x|^p \nu^i(\mathrm{d}x)  
		\bigg).
	\]
	Since by \eqref{eq:eps_optimal_ctrl} 
	\[ J\big(\Pb^N\big) \ge V^N_S(\nu^1, \dots, \nu^N) -\varepsilon_N \ge J\big(\Pb^N_0\big)-\varepsilon_N,\] 
	it follows that
	$J\big(\Pb^N\big) \ge -C$, for some constant $C$ independent of $N$.
	Using again \Cref{lemma:estimates}, the coercivity condition \eqref{eq:cond_coercive}, and the growth conditions in \Cref{assum:main1},
	it follows that 
	\[
		J\big(\Pb^N\big)
		\le 
		K
		\bigg( 
			1 + \int_{\R^n}|x^\prime|^p \frac{1}{N}\sum_{i=1}^N \nu^i(\mathrm{d}x^{\prime})
			+ 
			\frac{1}{N}\sum_{i=1}^N\E^{\P^N_\nu}\bigg[ \int_0^T  \big| \alpha^{i,N}_t - a_0 \big|^{p}\mathrm{d}t \bigg]  
		\bigg)
		- 
		C_L \frac{1}{N}\sum_{i=1}^N\E^{\P^N_\nu}\bigg[ \int_0^T \big| \alpha^{i,N}_t - a_0 \big|^{p^{\prime}} \mathrm{d}t \bigg].
        \]
	Then, there exists some constant $C >0$, independent of $N$, such that
	\[
		C_L \frac{1}{N}\sum_{i=1}^N\E^{\P^N_\nu}\bigg[ \int_0^T  \big| \alpha^{i,N}_t - a_0 \big|^{p^{\prime}}\mathrm{d}t \bigg] 
		-
		K\frac{1}{N}\sum_{i=1}^N \E^{\P^N_\nu}\bigg[ \int_0^T  \big| \alpha^{i,N}_t - a_0 \big|^{p} \mathrm{d}t \bigg] 
		< 
		C.
	\]
	Since $p^\prime > p$, it follows that
	\begin{equation} \label{eq:integ_alpha_inter}
		\sup_{N \ge 1}  \frac{1}{N}\sum_{i=1}^N\E^{\P^N_\nu}\bigg[ \int_0^T  \big| \alpha^{i,N}_t - a_0 \big|^{p^{\prime}}\mathrm{d}t \bigg] 
		< 
		\infty.
	\end{equation}
	With the condition $\sup_{N\geq 1} \frac{1}{N}\sum_{i=1}^N \int_{\R^n} | x |^{p^\prime} \nu^i(\mathrm{d}x) < \infty$,
	and by similar arguments as in \cite[Proposition 3.5.]{lacker2017limit}, 
	it is easy to deduce that both $(\frac1N \sum_{i=1}^N \nu^i)_{ N \ge 1}$ and $(\Pb^N)_{N \ge 1}$ are relatively compact under $\Wc_{p}$.
	 
\medskip

	\underline{\it Identification of the limit}:
	Up to a subsequence, let us assume w.l.o.g. that
	\[
	{\color{black}
		\Lim_{N \to \infty} \Wc_{p} \big( \Pb^{N}, \Pb \big)=0,
		~\mbox{for some}~
		\Pb \in \Pc_p(\Omb),
		~\mbox{so that}~
		\lim_{N \to \infty} \Wc_p \bigg( \frac{1}{N}\sum_{i=1}^{N} \nu^i , \nu \bigg) = 0,
		~\mbox{with}~\nu := \Pb \circ X_0^{-1} \in \Pc_p(\R^n),		
	}
	\]
	and then prove that $\Pb \in \Pcb_R(\Omb)$.
 	To this end, it is enough, by \Cref{prop:eqivalence_def_relaxed}, to prove that $\Pb$ satisfies the following properties
	\begin{itemize}
	\item[$(i)$] $\Pb \big[ \muh \circ (X_0)^{-1}=\nu,X_0=Y_0, W_0=0, B_0=0 \big] = 1$;
	
	\item[$(ii)$] $\E^{\Pb} \big[ \|X\|^p + \int_{[0,T]\x A} \big( \rho(a_0, a) \big)^p \Lambda_t(\mathrm{d}a) \mathrm{d}t \big]  < \infty$;
	
	\item[$(iii)$] $\muh$ satisfies \eqref{eq:muh_property} under $\Pb$;
	\item[$(iv)$] $(B_t)_{t \in [0,T]}$ is an $(\Fb,\Pb)$--Brownian motion;
	\item[$(v)$] the process $( S^{f}_t)_{t \in [0,T]}$ (defined in \eqref{eq:K_process}) is an $(\Fb^\circ,\Pb)$--martingale w.r.t. the filtration 
	$\Fb^\circ=(\Fc^{\circ}_t)_{t \in [0,T]}$ defined by $\Fcb^{\circ}_t:=\sigma (X_{t \wedge \cdot},Y_{t \wedge \cdot},B_{t \wedge \cdot}, \mu_t)$ for all $f \in C^2_b(\R^n \x \R^\ell)$;
	\item[$(vi)$] finally, for $\Pb$--a.e. $\omb \in \Omb$,
	$\big(\Mh^{\varphi,\mu(\omb)}_t\big)_{t \in [0,T]}$ (defined in \eqref{eq:Mvarphi})
	is an $\big(\Fh,\muh(\omb) \big)$--martingale for all $\varphi \in C^2_b(\R^n \x \R^d)$.
\end{itemize}

	First, let us consider two bounded continuous functions $h^1$, $h^2$ in $C_b(\R^n)$, we have
	\begin{align*}
		\E^{\Pb}
		\big[ 
			\langle h^1 ,\muh \circ (X_{0})^{-1} \rangle
			\langle h^2 ,\muh \circ (X_{0})^{-1} \rangle
		\big]
		&=
		\lim_{N\to\infty} \frac{1}{N^2}\sum_{i,j=1}^N \E^{\P^N_{\nu}}\big[  h^1(X^{i}_0) h^2(X^{j}_0)\big] 
		\\
		&=
		\lim_{N\to\infty} \frac{1}{N^2}\sum_{i=1}^N\langle h^1 h^2,\nu^i \rangle
		+
		\lim_{N\to\infty} \frac{1}{N^2}\sum_{i \neq j}^N\langle h^1,\nu^i \rangle \langle h^2,\nu^j \rangle\\
		&=
		\lim_{N\to\infty}
		\bigg \langle h^1, \frac{1}{N}\sum_{i=1}^N \nu^i \bigg \rangle \bigg \langle h^2, \frac{1}{N}\sum_{i=1}^N \nu^i \bigg \rangle
		=
		\langle h^1,\nu \rangle
		\langle h^2,\nu \rangle.
        \end{align*}
	Using similar arguments, we can deduce that for all $k \ge 1$ and bounded continuous functions $h^1, \dots, h^k \in C_b(\R^n)$
	\[
		\E^{\Pb}
		\big[ 
			\Pi_{i=1}^i \langle h^i ,\muh \circ (X_{0})^{-1} \rangle
		\big]
		=
		\Pi_{i=1}^k \langle h^i,\nu \rangle,
		~\mbox{and hence}~
		\Pb[ \muh \circ (X_0)^{-1} = \nu ] = 1.
	\]
	Besides, with the definition of $\P^N_{\nu}$ in \Cref{subsec:N_controlpb},
	and then by \eqref{eq:integ_alpha_inter}, it is easy to deduce that
	\[
		\Pb \big[ X_0=Y_0, W_0=0, B_0=0 \big]=1,
		~\mbox{and}~
		\E^{\Pb} \bigg[\|X\|^p + \iint_{[0,T]\times A} (\rho(a_0,a))^p \Lambda_t(\mathrm{d}a)\mathrm{d}t \bigg] < \infty.
	\]
        
	Next, notice that, for all $\phi \in C_b(\Cc^n \x \Cc^n \x \M \x \Cc^d)$ and  $\psi \in C_b(\Cc^{\ell} \x \Pc(\Omh))$,
	\begin{align*}
	    &\E^{\Pb}
		\big[ \phi \big(X_{t \wedge \cdot}, Y_{t \wedge \cdot}, \Lambda^t,W \big) \psi \big(B,\muh \big) \big]
		\\
		&=
		\Lim_{N\to\infty} \E^{\Pb^N}
		\big[ \phi \big(X_{t \wedge \cdot}, Y_{t \wedge \cdot}, \Lambda^t,W \big) \psi \big(B,\muh \big) \big]
		=\Lim_{N\to\infty} \frac{1}{N} \sum_{i=1}^N \E^{\P^N_{\nu}}
		\big[ \phi \big(X^i_{t \wedge \cdot}, Y^i_{t \wedge \cdot}, (\delta_{\alpha^{i,N}_s}(\mathrm{d}a)\mathrm{d}s)^t,W^i \big) \psi \big(B,\overline{\varphi}_{N}  \big) \big]
		\\
		&=\Lim_{N\to\infty} \E^{\P^N_{\nu}}
		\Big[ \E^{\overline{\varphi}_N} \big[ \phi \big(\Xh_{t \wedge \cdot}, \Yh_{t \wedge \cdot}, (\Lambdah)^t,\Wh \big) \big] \psi \big(B,\overline{\varphi}_N  \big) \Big]
		=
		\E^{\Pb}
		\Big[ 
			\E^{{\hat \mu}}\big[ \phi \big( \Xh_{t \wedge \cdot}, \Yh_{t \wedge \cdot},\Lambdah^t, \Wh \big) \big] 
			\psi \big(B,\muh \big) 
		\Big],
	\end{align*}
	which implies that $\muh$ satisfies \eqref{eq:muh_property} under $\Pb$
	that is, for $\Pb$--a.e. $\omb \in \Omb$, 
	\[
		\muh_t(\omb)
		=
		\Pb^{\Gcb_T}_{\omb} \circ \big(X_{t \wedge \cdot},Y_{t \wedge \cdot},\Lambda^t,W \big)^{-1}.
	\]

	We next show that $(B_t)_{t \in [0,T]}$ is an $(\Fb,\Pb)$--Brownian motion.
	First, since $\P^N_{\nu} \circ B^{-1}$ is the Wiener measure, 
	it is clear that $\Pb \circ B^{-1}$ is also the Wiener measure.
	Next, let $\phi \in C_b(\Omb)$, for all $s \in [0,T]$, we define the random variables
	\[
		\Phi_s 
		:= 
		\phi \big(X_{s \wedge \cdot}, Y_{s \wedge \cdot}, \Lambda^s, W_{s \wedge \cdot}, B_{s \wedge \cdot}, \muh_{s \wedge \cdot} \big)
		~\mbox{on}~\Omb,
		~\mbox{and}~
		\Phi^i_s
		:=
		\phi \big(X^{\alpha,i}_{s \wedge \cdot}, Y^{\alpha,i}_{s \wedge \cdot}, (\delta_{\alpha^{i,N}_t}(\mathrm{d}a)\mathrm{d}t)^s, W^i_{s \wedge \cdot}, B_{s \wedge \cdot},\overline{\varphi}^N_{s \wedge \cdot} \big)
		~\mbox{on}~(\Om^N,\Fc^N).
	\]
	On $(\Om^N,\Fc^N)$, we introduce the $\sigma$--algebra $\Fc^{N,W} := \sigma \{ W^1,\dots,W^N\}$.
	Then, for all $\psi \in C_b(\R^{\ell})$ and $t \ge s$
	\begin{align*}
		\E^{\Pb} \big[ \psi (B_t-B_s) \Phi_s \big]
		&=
		\lim_{N\to \infty} \frac{1}{N} \sum_{i=1}^N \E^{\P^N_\nu} \big[ \psi(B_t-B_s) \Phi^{i}_s \big]
		=
		\lim_{N\to \infty} \frac{1}{N} \sum_{i=1}^N 
			\E^{\P^N_\nu} \Big[ \E^{\P^N_\nu} \big[ \psi(B_t-B_s)  \Phi^{i}_s \big| \Fc^{N,W} \big] \Big] \\
		&=
		\lim_{N\to \infty} \frac{1}{N} \sum_{i=1}^N 
			\E^{\P^N_\nu} \Big[ \E^{\P^N_\nu} \big[ \psi(B_t-B_s)\big| \Fc^{N,W} \big] 
				\E^{\P^N_\nu} \big[ {\Phi}^{i}_s \big| \Fc^{N,W} \big] 
			\Big]\\
		&=
		\lim_{N\to \infty}  
			\frac{1}{N} \sum_{i=1}^N
			\E^{\P^N_\nu} \big[ \psi(B_t-B_s) \big]
			\E^{\P^N_\nu} \big[ {\Phi}^{i}_s \big] 
		=
		\E^{\Pb} \big[ \psi(B_t-B_s)\big] \E^{\Pb} \big[\Phi_s \big].
	\end{align*}
	This implies that $B$ is an $(\Fb,\Pb)$--Brownian motion.

\medskip	
	We finally consider the two martingale problems in \Cref{prop:eqivalence_def_relaxed},
	for which we can adapt the proofs in \cite[Proposition 5.1.]{lacker2017limit}.
	Let $\varphi \in C_b^2 (\R^{n}\x \R^d)$, 
	$f \in C_b(\R^n \x \R^\ell)$,
	$\psi \in C_b(\Omh)$, 
	$\phi \in C_b(\Pc(\Omh))$
	and $\beta \in C_b(\Cc^n \x \Cc^n \x \Cc^\ell \x \Pc(\Cc^n))$.
	In addition, on $(\Om^N,\Fc^N)$, we define the processes $\widehat{S}^{\varphi,N,i}$ for $i =1,\dots,N$ by
	\[
		\widehat{S}^{\varphi,N,i}_t
		:=
		\varphi \big(Y^{\alpha,i}_t, W^i_t \big)-\varphi(Y^{\alpha,i}_0, W^i_0)
		-
		\int_0^t  \widehat \Lc_s \varphi\big(X^{\alpha,i}, Y^{\alpha,i}, W^i, \alpha^{i,N}_s, \varphi^{N,X}_s \big) \mathrm{d}s,
		~t \in [0,T],
	\]
	where $\widehat \Lc$ is defined in \eqref{eq:conditionnal_generator}.
	Then $(\widehat{S}^{\varphi,N,i})_{i \in \{1,\dots,N\}}$ are $(\P^N_\nu,\F^N)$--orthogonal martingales with quadratic variation
	\[
		\bigg(\int_0^t \big|\sigma \big(r, X^{\alpha,i}, \varphi^{N,X}_r,\alpha^{i,N}_s \big) \nabla \varphi(X^{\alpha,i}_r) \big|^2 \mathrm{d}r \bigg)_{t \in [0,T]},
		~i=1, \dots, N.
	\]
	Denote
	\[
		\big \langle \big(\Mh^{\varphi,\mu}_t - \Mh^{\varphi,\mu}_r\big) \Psi_r,\muh \big \rangle
		:=
		\E^{\muh} \Big[ 
			\big(\Mh^{\varphi,\mu}_t- \Mh^{\varphi,\mu}_r \big) 
			\psi(\Xh_{r \wedge \cdot}, \Yh_{r \wedge \cdot},\Lambdah^r,\Wh_{r \wedge \cdot})  
		\Big],
	\]
	it follows by direct computation that, for some constant $C>0$ whose value may vary from line to line
	\begin{align*}
		&\
		\Big|\E^{\Pb}
		\big[
			\phi(\muh) ~ \big\langle \big(\Mh^{\varphi,\mu}_t - \Mh^{\varphi,\mu}_r \big) \Psi_r, \muh \big \rangle 
		\big] \Big|
		=
		\lim_{N\to\infty}
		\Big| \E^{\Pb^{N}}
		\big[
			\phi(\muh) ~\big\langle \big( \Mh^{\varphi,\mu}_t - \Mh^{\varphi,\mu}_r \big) \Psi_r, \muh \big\rangle 
		\big] \Big|
		\\
		\le&\
		\limsup_{N\to\infty}
		\E^{\Pb^{N}}\big[ \big| \phi(\muh) \big|^2 \big]^{1/2}
		\E^{\Pb^{N}}\Big[
			\big| \big\langle \big( \Mh^{\varphi,\mu}_t- \Mh^{\varphi,\mu}_r \big) \Psi_r, \muh \big \rangle \big|^2 
		\Big]^{1/2}
		\\
		=&\
		\limsup_{N\to\infty}
		C\E^{\Pb}\big[ \big| \phi(\muh) \big|^2 \big]^{1/2}
		\E^{\P^N_\nu}\bigg[
			\bigg|\frac{1}{N} \sum_{i=1}^N ( \widehat{S}^{\varphi,N,i}_t - \widehat{S}^{\varphi,N,i}_r)\psi \big(X^{\alpha,i}_{r \wedge \cdot},Y^{\alpha,i}_{r \wedge \cdot},(\Lambda^i)^r,W^i_{r \wedge \cdot} \big)\bigg|^2 
		\bigg]^{1/2}
		\\
		\le&\
		\limsup_{N\to\infty}
		C\E^{\Pb}\big[ \big|\phi(\muh) \big|^2 \big]^{1/2}
		\bigg(\frac{1}{N^2} \sum_{i=1}^N\E^{\P^N_\nu}\bigg[
		\int_r^t \big|\sigma \big(s, X^{\alpha,i}, \varphi^{N,X}_s,\alpha^{i,N}_s \big) \nabla \varphi(X^{\alpha,i}_s) \big|^2 \mathrm{d}s 
		\bigg] \bigg)^{1/2}
		\\
		 \le &\
		\limsup_{N\to\infty}
		C\E^{\Pb}\big[ \big|\phi(\muh) \big|^2 \big]^{1/2}
		\bigg(\frac{1}{N^2} \sum_{i=1}^N\E^{\P^N_\nu}\bigg[
		\int_r^t \big|X^{\alpha,i}_{s \wedge \cdot} \big|^p +\rho (a_0,\alpha^{i,N}_s)^p \mathrm{d}s 
		\bigg] \bigg)^{1/2}
		\le
		\limsup_{N\to\infty}
		\frac{C}{\sqrt{N}}
		=
		0.
        \end{align*}
	This implies that, for  $\Pb$--a.e. $\omb \in \Omb$
	\[
		\big \langle \big(\Mh^{\varphi,\mu(\omb)}_t- \Mh^{\varphi,\mu(\omb)}_r \big) \Psi_r,\muh(\omb) \big \rangle
		=
		0.
	\]

	Similarly,  with $\hat a_0$ defined in equation \eqref{eq:def_hat_a_0} and  $Z^i:=X^{\alpha,i}-Y^{\alpha,i}$,
	let us introduce 
	$(S^{f,N, i}_t)_{t \in [0,T]}$ on $\Om^N$ by
	\[
		S^{f,N, i}_t
		:=
		f \big(Z^i_t, B_t \big)-\varphi(Z^i_0, B_0)
		-
		\int_0^t \frac{1}{2}\mathrm{Tr}\big[ \hat a_0(s, X^{\alpha,i}, \varphi^{N,X}_s) \nabla^2 \varphi(Z^i_s,B_s)\big]\mathrm{d}s,
		~ t \in [0,T].
	\]
	Denoting $\Lambda^i(\mathrm{d}a)\mathrm{d}t:=\delta_{\alpha^{i,N}_t}(\mathrm{d}a)\mathrm{d}t$,
	and applying the same arguments as above, it follows that
	\[
		\E^{\Pb} \big[ 
			\big(S^f_t-S^f_r \big) \beta \big( X_{r \wedge \cdot}, Y_{r \wedge \cdot}, B_{r \wedge \cdot}, \mu_{r} \big) 
		\big]
		=
		\Lim_{N\to \infty} 
		\frac{1}{N} \sum_{i=1}^N
		\E^{\P^N_\nu}\Big[ \big(S^{f,N, i}_t-S^{f,N, i}_r \big) \beta \big( X^{\alpha,i}_{r \wedge \cdot}, Y^{\alpha,i}_{r \wedge \cdot}, B_{r \wedge \cdot}, \varphi^{N,X}_{r} \big) \Big]
		=
		0.
	\]
	Finally, by considering $(r,t, \psi, \phi)$ in a countable dense subset of 
	$[0,T] \x [0,T] \x C_b(\Omh) \x C_b(\Pc(\Omh))$, 
	it follows that the process $\big( S^{f}_t\big)_{t \in [0,T]}$ is an $( \Fb^\circ,\Pb)$--martingale 
	for all $f \in C^2_b(\R^n \x \R^\ell)$,
	and for $\Pb$--a.e. $\omb \in \Omb$,
	$(\Mh^{\varphi,\mu(\omb)}_t)_{t \in [0,T]}$
	is an $\big(\Fh,\muh(\omb) \big)$--martingale for all $\varphi \in C^2_b(\R^n \x \R^d)$.
	We then conclude that $\Pb \in \Pcb_R(\nu)$.
	\end{proof}

\subsubsection{Proof of Proposition  \ref{Proposition:Continuity-Existence}}

	Let $\nu \in  \Pc_{p^{\prime}}(\R^n)$ and $(\nu^m)_{m \ge 1} \subset \Pc_{p^{\prime}}(\R^n)$ be such that
	$\sup_{m \ge 1} \int_{\R^n}\|x^\prime\|^{p^\prime} \nu^m(\mathrm{d}x^\prime) < \infty$ and  $\lim_{m \to \infty} \Wc_{p}(\nu^m,\nu)=0$.  
	
\medskip
	
	We first consider two sequences $(\eps_m)_{m \ge 1} \subset \R_+$ and $(\Pb^m)_{m \ge 1}$ such that
	\[
		\lim_{m \to \infty} \eps^m=0,
		\;
		\Pb^m \in \Pcb_R(\nu^m),
		~\mbox{and}~ 
		J(\Pb^m) \ge V_R(\nu^m) - \eps^m,
		~\mbox{for all}~m \ge 1.
	\]
	It follows by \Cref{Propostion:convergence} that $(\Pb^m)_{m \in \N}$ is relatively compact under $\Wc_p$.
	Via a subsequence, let us assume that $\Pb^m \to \Pb^{\infty}$ under $\Wc_p$,
	so that $\Pb^{\infty} \in \Pcb_R(\nu)$.
	Using the continuity and growth conditions of $(L,g)$ in \Cref{assum:main1} and \Cref{assum:constant_case}, 
	it follows that $\lim_{m \to \infty} J(\Pb^m) = J(\Pb^{\infty})$, and therefore
	\[
		\limsup_{m \to \infty} V_R(\nu^{m}) 
		\le
		\lim_{m \to \infty} J(\Pb^{m})
		=
		J(\Pb^\infty)
		\le 
		V_R(\nu)
		=
		V_S(\nu).
	\]
	Together with the inequality from \Cref{propogation of chaos-2p},
	we then conclude the proof.
	\qed

\subsubsection{Proof of Theorem \ref{thm:limit}}
    
	$(i)$
	By \Cref{Propostion:convergence},
	the sequence $(\Pb^N)_{N \ge 1}$ is relatively compact under $\Wc_p$. 
	Further, for any convergent subsequence  $(\Pb^{N_m})_{m \ge 1}$, one has
	\[
		\lim_{m \to \infty} \Wc_p \bigg( \frac{1}{N_m}\sum_{i=1}^{N_m} \nu^i , \nu \bigg) = 0,
		~\mbox{for some}~\nu \in \Pc_p(\R^n),
		~\mbox{and}~
		\Lim_{m \to \infty} \Wc_{p} \big( \Pb^{N_m}, \Pb^{\infty} \big)=0,
		~\mbox{for some}~
		\Pb^{\infty} \in \Pcb_R(\nu).
	\]
	Moreover,
	under \Cref{assum:main1} and \Cref{assum:constant_case}, 
	it follows  by \eqref{eq:eps_optimal_ctrl}  and \Cref{rem:ContinuityJ} that
	\[
		\limsup_{N \to \infty} V_S^N(\nu^1,\dots,\nu^N) 
		\le
		\lim_{m \to \infty} J(\Pb^{N_m})
		=
		J(\Pb^{\infty})
		\le
		V_R(\nu)
		=
		V_S(\nu).
	\]
	Together with \Cref{propogation of chaos-2}, one obtains that
	\begin{equation} \label{eq:VN2V_inter}
		\lim_{N \to \infty} V_S^N(\nu^1,\dots,\nu^N)
		=
		J(\Pb^{\infty})
		=
		V_R(\nu)
		=
		V_S(\nu),
	\end{equation}
	and hence $\Pb^{\infty} \in \Pcb_R^\star(\nu)$.

\medskip
	
	{\color{black} $(ii)$ The second item is in fact a direct consequence of \Cref{prop:approximation}, \Cref{Equivalence-Proposition_General} and \Cref{propogation of chaos-2}.}

\medskip
	
	$(iii)$ 
	Finally, let  $(N_m)_{m \in \N}$ be a sequence such that
	\[
		\limsup_{N \to \infty} \bigg|V_S^N(\nu^1,\dots,\nu^N)-V_S \bigg( \frac{1}{N}\sum_{i=1}^{N}\nu^i \bigg) \bigg|
		=
		\lim_{m \to \infty} \bigg|V_S^{N_m}(\nu^1,\dots,\nu^{N_m})-V_S \bigg( \frac{1}{N}\sum_{i=1}^{N_m}\nu^i \bigg) \bigg|.
	\]
	One more time, through a subsequence, we can assume that 
	\[
		\frac{1}{N_m}\sum_{i=1}^{N_m}\nu^i \underset{m\to\infty}{\longrightarrow} \nu,
		~\mbox{under}~\Wc_p,
		~\mbox{for some}~
		\nu \in \Pc_p(\R^n).
	\]
	Using \eqref{eq:VN2V_inter} and \Cref{Proposition:Continuity-Existence}, we obtain that
	\[
		\limsup_{N \to \infty} \bigg|V_S^N(\nu^1,\dots,\nu^N)-V_S \bigg( \frac{1}{N}\sum_{i=1}^{N}\nu^i \bigg) \bigg|
		\le
		\lim_{m \to \infty} \bigg|V_S^{N_m}(\nu^1,\dots,\nu^{N_m})-V_S (\nu) \bigg|
		+
		\lim_{m \to \infty} \bigg|V_S(\nu) -V_S \bigg( \frac{1}{N}\sum_{i=1}^{N_m}\nu^i \bigg) \bigg|
		=
		0,
	\]
	and thus \eqref{eq:continuity_VNS_VS} holds true.
	\qed

\begin{appendix}   
\section{Some technical results} 

\subsection{Existence of weak solution to the McKean-Vlasov equations}
    
	Exceptionally, we do not impose Assumption \ref{assum:main1} in this subsection,
	but consider a weaker condition to prove the existence of weak solution to the McKean-Vlasov equation with initial distribution $\nu \in \Pc(\R^n)$
	(see Definition \ref{def:admissible_ctrl_rule}).
	
	\begin{assumption} \label{assum:weak-existence}
		The coefficient functions $(b, \sigma, \sigma_0)$ are continuous in $(\xb, \nub, a)$, and satisfy one of the following two conditions
		\begin{itemize}
			\item $p = 0$ and $(b, \sigma, \sigma_0)$ are bounded$;$
		
			\item $p \ge (1 \vee \hat p)$, for some $\hat p \in [0,2]$, and for all $(t,\xb, \nub, a) \in [0,T] \x \Cc^n \x \Pc(\Cc^n \x A)\x A $
			\begin{align*}
				| b(t,\xb, \nub, a)|
				&\le 
				C \bigg ( 1+ \| \xb \| + \Big( \int_{\Cc^n \x A}\big( \| \xb^\prime\|^p + \rho(a_0,a^\prime)^p\big) \nub(\mathrm{d}\xb^\prime,\mathrm{d}a^\prime) \Big)^{\frac{1}{p}} + \rho(a_0,a) \bigg ),\\
				| (\sigma,\sigma_0)(t, \xb, \nub, a)|^2
				&\le 
				C \bigg ( 1+ \| \xb \|^{\hat p} + \Big( \int_{\Cc^n\x A} \big(\| \xb^\prime\|^p + \rho(a_0,a^\prime)^p\big) \nub(\mathrm{d}\xb^\prime,\mathrm{d}a^\prime) \Big)^{\frac{\hat p}{p}} + \rho(a_0,a)^{\hat p} \bigg ).
			\end{align*}
		\end{itemize}

	\end{assumption}

	\begin{theorem} \label{Theorem:ExistenceSolutionMartingale}
		Let {\rm\Cref{assum:weak-existence}} hold true. Then
		
		\begin{itemize}
			\item when $p = 0$, then there exists $\Pb \in \Pcb_W(\nu)$ for all $\nu \in \Pc(\R^n);$
		
			\item when $p \ge 1$, assume in addition that $\nu \in \Pc(\R^n)$ for some  $p^{\prime} \in (p, \infty) \cup [2, \infty)$. 
			Then there exists $\Pb \in \Pcb_W(\nu)$ and it holds that $\E^{\Pb}\big[ \|X\|^{p^{\prime}} \big] < \infty$.
		\end{itemize}
	\end{theorem}
	
	\begin{proof}
	Without loss of generality, we can assume that $A$ is a singleton given by $A = \{a_0\}$ (otherwise, we can use a constant control process equals to $a_0$).

	\vspace{0.5em}

	First, recall that the filtered probability space $(\Om, \F, \Fc, \P_{\nu})$ is defined in \Cref{subsec:strong_form} and equipped with an initial random variable $X_0$, together with Brownian motions $(W, B)$.
	For each $m \ge 1$,
	we consider the solution  $(X^m_t)_{t \in [0,T]}$ of the Euler scheme of McKean--Vlasov equation \eqref{eq:MKV_SDE}, 
	that is
	\[
		X^m_t
		:= 
		X_0 + \int_0^t b\big(r,X^m_{[r]^m \wedge \cdot },\mub^m_{[r]^m},a_0\big)\mathrm{d}r + \int_0^t \sigma \big(r,X^m_{[r]^m \wedge \cdot },\mub^m_{[r]^m},a_0\big) \mathrm{d}W_r  + \int_0^t\sigma_0 \big(r,X^m_{[r]^m \wedge \cdot },\mub^m_{[r]^m},a_0\big) \mathrm{d}B_r,\; \P_{\nu}\mbox{--a.s.},
	\]
	where $\mub^m_t:=\Lc^{\P_{\nu}}(X^m_{t \wedge \cdot}|\Gc_t) \otimes \delta_{a_0}$,
	$[t]^m=i T 2^{-m}$ for all $ t \in \big[ i T 2^{-m}, (i+1)T 2^{-m} \big)$ and $i \in \{0,\dots,2^m-1\}$. 
	Under Assumption \ref{assum:weak-existence},
	it is straightforward to check that for each $m \in \N^\star$
	\begin{itemize}
		\item when $p =0$, $\E^{\P_{\nu}}\big[\sup_{t \in [0,T]}|X^m_t-X_0|^{q} \big] < \infty$ for all $q \ge 0;$
		
		\item when $p \ge 1$, $\E^{\P_{\nu}}\big[\sup_{t \in [0,T]}|X^m_t|^{p^\prime} \big] < \infty$. 
	\end{itemize}
	
	As in  \Cref{lemma:estimates}, 
	by classical arguments using Gronwall's lemma as in \citeauthor*{lacker2017limit} \cite[Lemma 3.1]{lacker2017limit},
	it follows that, for some constant $C > 0$ independent of $m$
	\begin{itemize}
		\item when $p=0$
		\begin{align} \label{eq:estimates-exis-bounded}
		\E^{\P_{\nu}}\bigg[\sup_{t \in [0,T]}|X^m_t-X_0|^{p^\prime} \bigg] 
		\le
		C;
	\end{align}
			
		\item when $p \ge 1$
	\begin{align} \label{eq:estimates-exis}
		\E^{\P_{\nu}}\bigg[\sup_{t \in [0,T]}|X^m_t|^{p^\prime} \bigg] 
		\le
		C \bigg(1+ \int_{\R^n}|x |^{p^\prime} \nu(\mathrm{d}x) \bigg).
	\end{align}
	\end{itemize}
	
	Let us denote 
	$Y^m_{\cdot}:=X^m_\cdot-\int_0^\cdot\sigma_0 \big(r,S^m_{[r]^m \wedge \cdot },\mub^m_{[r]^m},a_0\big) \mathrm{d}B_r$,
	and
	\[
		\Pb^m
		:=
		\P_{\nu} \circ \Big (X^m,Y^m,\Lambda^\circ,W,B,\muh^m \Big)^{-1},\;\mbox{with}\;\Lambda^\circ_t(\mathrm{d}a)\mathrm{d}t:=\delta_{a_0}(\mathrm{d}a)\mathrm{d}t,\;\mbox{and}\;\muh^m:=\Lc^{\P_{\nu}} \big(X^m,Y^m,\Lambda^\circ,W \big| \Gc_{T} \big).
	\]
	When $p=0$,
	the sequence $(\Pb^m)_{m \ge 1}$ is relatively compact under the weak convergence topology, by  \eqref{eq:estimates-exis-bounded}.
	When $p \ge 1$, by  \eqref{eq:estimates-exis} and \cite[Proposition A.2]{carmona2014mean},
	the sequence $(\Pb^m)_{m \ge 1}$ is relatively compact under $\Wc_{p}$.
	By possibly taking a subsequence, we can assume that,  
	for some $\Pb \in \Pc(\Omb)$, $\Pb^m \longrightarrow_{m\to\infty} \Pb$ under the weak convergence topology or $\Wc_p$, according to the value of $p$.
	
	\medskip
	
	We next show that $\Pb \in \Pcb_W(\nu)$.
	Recall that, for each $\varphi \in  C^2_b(\R^n \x \R^n \x \R^d \x \Cc^\ell)$,
	the process $\Sb^{\varphi}$ is defined on $\Omb$ by \eqref{eq:associate-martingale}.
	Similarly, we define the processes $\Sb^{\varphi, m} = (\Sb^{\varphi, m}_t)_{t \in [0,T]}$ on $\Omb$ by
	\[
		\Sb^{\varphi,m}_t 
		:=
		\varphi (X_t, Y_t, W_t, B_t)
		-
		\int_0^t
		\Lc_r \varphi \big(
			X_{[r]^m \wedge \cdot}, 
			Y_{[r]^m \wedge \cdot},
			W_{[r]^m \wedge \cdot},
			B_{[r]^m \wedge \cdot},
			a_0,
			\mub_{[r]^m}
		\big) \mathrm{d}r.
	\]
	By the continuity of the coefficient functions $(b, \sigma, \sigma_0)$, then uniform continuity on a compact set, it is straightforward to  check that 
	on each compact subset $\Omb_c \subset \Omb$, one has 	
	\begin{align} \label{eq_compacitness}
		\lim_{m \to \infty} \sup_{\omb\;\in\; \Omb_c} 
		\big| \Sb^{\varphi,m}_t(\omb)- \Sb^{\varphi}_t(\omb)\big|=0,\;\mbox{for every}\;t \in [0,T].
	\end{align}
	Further, whatever the case $p=0$ or $p \ge 1$ in Assumption \ref{assum:weak-existence},
	one has
	\begin{align} \label{eq:growth-martingale}
		\sup_{m \ge 1} \E^{\Pb^m} \big[ \big| \Sb^{\varphi, m}_t \big|^{p^{\prime}} \big] < \infty,
		~\mbox{for all}~t \in [0,T], ~\varphi \in  C^2_b(\R^n \x \R^n \x \R^d \x \Cc^\ell).
	\end{align}
	Now, as $(\Pb^m)_{m \in \N^{\star}}$ is relatively compact, for each $\eps > 0$, we can find a compact subset $\Omb_{\eps} \subset \Omb$ such that 
	$\Pb^m[ \Omb_{\eps}] \ge 1- \eps$ for all $m \ge 1$.
	For any bounded continuous function $\phi \in C_b( \Cc^n \x \Cc^n \x \Cc^d \x \Cc^\ell \x \Pc(\Cc^n \x \Cc^n \x \Cc^d) )$
	and $s \le t$, we denote $\Phi_s := \phi \big(X_{s \wedge \cdot},Y_{s \wedge \cdot},W_{s \wedge \cdot},B_{s \wedge \cdot}, \muh_s \big)$, with $(b,\sigma,\sigma_0)$ bounded or with $(b, \sigma, \sigma_0)$ satisfying Assumption \ref{assum:weak-existence} and $\nu \in \Pc_{p^{\prime}}(\R^n),$
	it follows that
	\begin{align*}
		\big|\E^{\Pb}\big[\big( \Sb^{\varphi}_t - \Sb^{\varphi}_s \big)\Phi_s \big]  \big|
		&=
		\lim_{m\to \infty}
		\big|\E^{\Pb^m}\big[\big( \Sb^{\varphi}_t - \Sb^{\varphi}_s \big) \Phi_s \big]  \big| 
		\\
		&\le
		\limsup_{m \to \infty} 
		\big|\E^{\Pb^m}\big[\big( \Sb^{\varphi}_t- \Sb^{\varphi}_s \big)\Phi_s \mathbf{1}_{\Omb_{\eps}}\big]  \big|
		+
		\limsup_{m \to \infty} 
		\big|\E^{\Pb^m}\big[\big( \Sb^{\varphi}_t- \Sb^{\varphi}_s\big)\Phi_s \mathbf{1}_{\Omb_{\eps}^c}\big]  \big|
		\\
		&\le
		\limsup_{m \to \infty} 
		\Big[
			\big|\E^{\Pb^m}\big[\big(\Sb^{\varphi,m}_t- \Sb^{\varphi,m}_s \big)\Phi_s \big]  \big|
			+
			\big|\E^{\Pb^m}\big[\big(\Sb^{\varphi,m}_t- \Sb^{\varphi,m}_s \big)\Phi_s \mathbf{1}_{\Omb_{\eps}^c}\big]  \big| 
			+
			\big|\E^{\Pb^m}\big[\big(\Sb^{\varphi}_t- \Sb^{\varphi}_s \big)\Phi_s \mathbf{1}_{\Omb_{\eps}^c}\big]  \big| 
		\Big]
		\\
		&\le  
		C \eps^{\frac{p^{\prime}-1}{p^{\prime}}},
	\end{align*}
	where the last inequality follows by H\"older's inequality together with \eqref{eq:growth-martingale} and the fact that $\Pb^m[\Omb_{\eps}^c] \le \eps$, for all $m \ge 1$. 
	Let $\eps \longrightarrow 0$, it follows that
	$\Sb^{\varphi}$ is an $(\Fb,\Pb)$--martingale for all $\varphi \in  C^2_b(\R^n \x \R^n \x \R^d \x \Cc^\ell)$. Further, by almost the same arguments as in the proof of \Cref{Propostion:convergence}, 
we have for all  $t \in [0,T]$
	\[
		\muh_t
		=
		\Lc^{\Pb} (X_{t \wedge \cdot},Y_{t \wedge \cdot },\Lambda^t,W \big| \Gcb_T)
		=
		\Lc^{\Pb} (X_{t \wedge \cdot},Y_{t \wedge \cdot },\Lambda^t,W \big| \Gcb_t),\; \Pb\mbox{--a.s.}
	\]
	Moreover, it is easy to see that $\Pb \circ (X_0)^{-1}=\nu$ and $(B,\muh)$ is independent of $(W,X_0)$ under $\Pb$, so that $\Pb \in \Pcb_W(\nu)$.
	
	\medskip

	Finally, for the case $p \ge 1$ in Assumption \ref{assum:weak-existence},
	using  \eqref{eq:estimates-exis}, the fact that $\nu \in \Pc_{p^{\prime}}(\R^n)$ together with Fatou's lemma, 
	it follows that $\E^{\Pb}\big[ \|X\|^{p^{\prime}} \big] < \infty$.
	\end{proof}

\subsection{Characterisation of probability measures on a set of probability measures}

	\begin{proposition} \label{Prop-identification_probability}
		Let $E$ be a Polish space, $(\Upsilon_1,\Upsilon_2) \in \Pc(\Pc(E)) \times  \Pc(\Pc(E))$ be such that
		\begin{equation} \label{eq:Up1_Up2}
			\int_{\Pc(E)} \Prod_{i=1}^k \langle \varphi_i,\nu \rangle \Upsilon_1(\mathrm{d}\nu)
			=
			\int_{\Pc(E)} \Prod_{i=1}^k \langle \varphi_i,\nu \rangle \Upsilon_2(\mathrm{d}\nu),
			~\mbox{\rm for all}~ 
			k \ge 1,
			~\mbox{\rm and}~ (\varphi_i)_{i \in \{1,\dots,k\}} \subset C_b(E).
		\end{equation}
		Then $\Upsilon_1=\Upsilon_2.$
	\end{proposition}
    
	\begin{proof}
	First, using \eqref{eq:Up1_Up2}, we have, for all $k \ge 1$, for every family of polynomial
	functions $(\psi^i)_{i \in \{1,\dots,k\}}$, and every $(\varphi_i)_{i \in \{1,\dots,k\}} \subset C_b(E;\R)$,
	\begin{equation} \label{eq:Up1_Up2p}
		\int_{\Pc(E)} \Prod_{i=1}^k \psi^i(\langle \varphi_i,\nu \rangle) \Upsilon_1(\mathrm{d}\nu)
		=
		\int_{\Pc(E)} \Prod_{i=1}^k \psi^{i}(\langle \varphi_i,\nu \rangle)  \Upsilon_2(\mathrm{d}\nu).
	\end{equation}
Since we can approximate any continuous function by polynomial functions, uniformly on compact sets,
	it follows that \eqref{eq:Up1_Up2p} still holds true for all
	$k \ge 1$, $(\psi^i)_{i \in \{1,\dots,k\}} \subset C_b(\R;\R)$ and $(\varphi_i)_{i \in \{1,\dots,k\}} \subset C_b(E;\R)$.
	This further implies that, for all $(r_1,\dots,r_k) \in  \R^k$
	\begin{align*}
		\int_{\Pc(E)} \Prod_{i=1}^k \mathbf{1}_{\{\nu:\langle \varphi_i,\nu \rangle < r^i \}} \Upsilon_1(\mathrm{d}\nu)
		=
		\int_{\Pc(E)} \Prod_{i=1}^k \mathbf{1}_{\{\nu:\langle \varphi_i,\nu \rangle < r^i \}} \Upsilon_2(\mathrm{d}\nu).
	\end{align*}
	In other words, $\Upsilon_1[A] = \Upsilon_2[A]$ for all $A \in \Psi$, where
	\[
		\Psi
		:=
		\Big\{
			A[r_1, \dots,r_m; \varphi_1, \dots,\varphi_m]
			: m \ge 1,\;(r_1,\dots,r_m) \in \R^m\;\mbox{and}\;(\varphi_1,\dots,\varphi_m) \in C_b(E)^m   
		\Big \},
	\]
	with 
	\[
		A[r_1, \dots,r_m; \varphi_1, \dots,\varphi_m]
		:=
		\big\{
			\lambda \in \Pc(E) : 
			\langle \varphi_i,\lambda \rangle < r_i, 
			~i = 1, \dots,m 
		\big \}.
	\]
	Notice that  the weak convergence topology on $\Pc(E)$ is generated by the open sets in $\Psi$,
	it follows by the monotone class theorem that $\Upsilon_1= \Upsilon_2$ on the Borel $\sigma$--field of $\sigma(\Psi)$.
	\end{proof}

\section{Proof of some technical results} 

	We finally provide here the proof of the approximation result (of relaxed control by weak control rules) in \Cref{Equivalence-Proposition_General},
	and some related technical results.
	In this Section, \Cref{assum:main1} and \Cref{assum:constant_case} are imposed,
	in particular, $A$ is a subset of $\R^j$ for some $j \ge 1$.
	
\subsection{An equivalent reformulation for relaxed control rules}
 
	On $\Omb$, let us introduce a filtration $\Fb^\circ=(\Fc^{\circ}_t)_{t \in [0,T]}$ 
	and a process $S^f = (S^f_t)_{t \in [0,T]}$, for every $f \in C^2_b(\R^{n + \ell})$, by
	\begin{align} \label{eq:K_process}
		\Fcb^{\circ}_t:=\sigma (X_{t \wedge \cdot},Y_{t \wedge \cdot},B_{t \wedge \cdot}, \mu_t),
		~\mbox{and}~
		S^{f}_t
		:=
		f \big(Z_t, B_t \big)-\varphi(Z_0, B_0)
		-
		\int_0^t \frac{1}{2}\mathrm{Tr}\big[  a_0(s, X, \mu) \nabla^2 \varphi(Z_s,B_s)\big]\mathrm{d}s,
		~t \in [0,T],
	\end{align}
	where $Z := X-Y$ and
	\begin{align} \label{eq:def_hat_a_0}
		 a_0 \big(t, \xb, \nu \big)
		:=
		\begin{pmatrix} 
			\sigma_0(t,\xb,\nu)  \\ 
			\mathrm{I}_{\ell} 
		\end{pmatrix}
		\begin{pmatrix} 
			\sigma_0(t,\xb,\nu)  \\ 
			\mathrm{I}_{\ell}
		\end{pmatrix}^{\top},
		~\mbox{for each}~
		(t,\xb,\nu) \in [0,T] \x \Cc^n \x \Pc(\Cc^n).
	\end{align}

    	\begin{proposition} \label{prop:eqivalence_def_relaxed}
		Let $\nu \in \Pc(\R^n)$, then a probability measure $\Pb \in \Pc(\Omb)$ belongs to $\Pcb_R(\nu)$ if and only if 
		\begin{itemize}
			\item[$(i)$] $\Pb \big[ \muh \circ (\Xh_0)^{-1} =\nu, Y_0=X_0,W_0=0,B_0=0 \big] = 1$, 
			$\E^{\Pb} \Big[ \| X\|^p + \int_{[0,T]\x A} \big( \rho(a_0, a) \big)^p \Lambda_t(\mathrm{d}a) \mathrm{d}t \Big]  < \infty$, 
			and
			\[
			    \muh_t (\omb)
			    =
			    \Pb^{\Gcb_T}_{\omb} \circ (X_{t \wedge \cdot},Y_{t \wedge \cdot},W, \Lambda^t)^{-1},\;\mbox{for}\;\Pb\;\mbox{\rm--a.e.}\;\omb \in \Omb,
			    ~\mbox{for all}~t \in [0,T];
			\]
			
			\item[$(ii)$] $(B_t)_{t \in [0,T]}$ is an $(\Fb,\Pb)$--Brownian motion, 
			and
			the process $\big( S^{f}_t\big)_{t \in [0,T]}$ is an $(\Fb^\circ,\Pb)$--martingale for all $f \in C^2_b(\R^n \x \R^\ell)$;

			\item[$(iii)$] 
			for $\Pb$--{\rm a.e.} $\omb \in \Omb$,
			the process $ \big( \Mh^{\varphi, \mu(\omb)}_t\big)_{t \in [0,T]}$ is an $\big(\Fh,\muh(\omb) \big)$--martingale for all $\varphi \in C^2_b(\R^n \x \R^d)$.
		\end{itemize}
    	\end{proposition}
    	
    	\begin{proof}
	First, let $\Pb \in \Pcb_R(\nu)$,
	then $\Sb^{\varphi}$ (recall  \eqref{eq:associate-martingale}) is an $(\Fb,\Pb)$--martingale for all $\varphi \in C^2_b \big(\R^n \x \R^n \x \R^d \x \R^{\ell} \big)$,
	which implies immediately that $(B_t)_{t \in [0,T]}$ is an $(\Fb,\Pb)$--Brownian motion and
	$S^f$ is an $(\Fb^\circ,\Pb)$--martingale for all $f \in C^2_b(\R^n \x \R^\ell)$.
	It follows that $\Pb$ satisfies Conditions $(i)$--$(ii)$--$(iii)$ in the statement.

	\medskip

	Next, let $\Pb \in \Pc(\Omb)$ satisfying Conditions $(i)$--$(ii)$--$(iii)$ in the statement.
	It is immediate to check that $\Lc^{\Pb} \big(X_0,W,B,\muh \big)=\Lc^{\Pb} (X_0 ) \otimes \Lc^{\Pb} (W ) \otimes \Lc^{\Pb} (B,\muh)$,
	i.e. $X_0$, $W$ and $(B,\muh)$ are mutually independent under $\Pb$.
	Then by comparing Conditions $(i)$--$(ii)$--$(iii)$ in the statement 
	with Definitions \ref{def:admissible_ctrl_rule} and \ref{def:relaxed_ctrl_rule}, it is enough to prove that $\Sb^{\varphi}$ is an $(\Fb,\Pb)$--martingale for all
	$\varphi \in C_b^2(\R^n \x \R^n \x \R^d \x \R^{\ell})$ to conclude that $\Pb \in \Pcb_R(\nu)$.
	
	\medskip
	
	To this end, let us introduce, for every $\varphi \in C^2_b(\R^n \x \R^d)$, 
	a process $\widetilde{S}^{\varphi} = (\widetilde{S}^{\varphi}_t)_{t \in [0,T]}$ on $(\Omb, \Fcb)$ by
	(recall also the definitions of $\widehat \Lc$ and $\Mh^{\varphi, \nu}$ in \eqref{eq:conditionnal_generator} and \eqref{eq:Mvarphi})
	\begin{equation}\label{eq:Mvarphi_canonical}
		\widetilde{S}^{\varphi}_t
		:=
		\varphi \big(Y_t, W_t \big)- \varphi(Y_0, W_0)
		-
		\iint_{[0,t]\x A}  \widehat \Lc_s \varphi \big(X, Y, W, \mu_s, a \big) \Lambda_s(\mathrm{d}a)\mathrm{d}s.
	\end{equation} 
	Since $B$ is an $(\Fb,\Pb)$--Brownian motion, 
	we have, for all $\theta \in \R^\ell$, $0\leq s \le t$, $\phi \in C_b(\Cc^n \x \Cc^n \x \Cc^d \x \M)$, and $\psi \in C_b(\Cc^\ell \x \Pc(\Omh))$, 
	\begin{align*}
		&\ \E^{\Pb} \Big[ 
			\widetilde{S}^{\varphi}_t 
			\exp\big(\theta \cdot B_t -  |\theta|^2 t/2  \big)
			\phi \big(X_{s \wedge \cdot}, Y_{s \wedge \cdot}, W_{s \wedge \cdot}, \Lambda^s \big) 
			\psi \big( B_{s \wedge \cdot}, \muh_s \big) 
		\Big]
		\\
		=&\
		\E^{\Pb} \Big[ 
			\E^{\widehat \mu} \Big[ \Mh^{\varphi,\mu}_t \phi \big(\Xh_{s \wedge \cdot}, \Yh_{s \wedge \cdot}, \Wh_{s \wedge \cdot}, \Lambdah^s \big) \Big] 
			\exp\big(\theta \cdot B_t -  |\theta|^2 t/2  \big)
			\psi \big( B_{s \wedge \cdot}, \muh_s \big) \Big]
		\\
		=&\
		\E^{\Pb} \Big[ 
			\E^{\widehat \mu} \Big[ \Mh^{\varphi,\mu}_s \phi \big(\Xh_{s \wedge \cdot}, \Yh_{s \wedge \cdot}, \Wh_{s \wedge \cdot}, \Lambdah^s \big) \Big] 
			\exp\big(\theta \cdot B_s -  |\theta|^2 s/2  \big)
			\psi \big( B_{s \wedge \cdot}, \muh_s \big) \Big]
		\\
		=&\
		\E^{\Pb} \Big[ 
			\widetilde{S}^{\varphi}_s
			\exp\big(\theta \cdot B_s -  |\theta|^2 s / 2  \big)
			\phi \big(X_{s \wedge \cdot}, Y_{s \wedge \cdot}, W_{s \wedge \cdot}, \Lambda^s \big) 
			\psi \big( B_{s \wedge \cdot}, \muh_s \big) 
		\Big].
	\end{align*}
	In other words, 
	$(\widetilde{S}^{\varphi}_t \exp(\theta B_s - |\theta_s|^2 s/2) )_{t \in [0,T]}$ is an $(\Fb,\Pb)$--martingale for any $\varphi \in C^2_b(\R^{n + d})$ and $\theta \in \R^\ell$.
	From Condition $(ii)$ in the statement, we know that $B$ is an $(\Fb,\Pb)$--Brownian motion,
	and 
	\[
		Y_{\cdot}=X_\cdot-\int_0^\cdot \sigma_0(s, X, \mu ) \mathrm{d}B_s,\; \Pb\text{\rm--a.s.}
	\]
	Then it follows by \cite[Theorems 4.2.1 and 8.1.1]{stroock2007multidimensional} that 
	$\Sb^{\varphi}$ is an $(\Fb, \Pb)$--martingale for all $\varphi \in C^2_b \big(\R^n \x \R^n \x \R^d \x \R^{\ell} \big)$.
    	\end{proof}

\subsection{Proof of Proposition \ref{Proposition:diffusion_McKV-relaxed}} 
\label{proof_Proposition-measurability-check}


	We first recall the definition of the martingale measures (see e.g. \citeauthor*{el1990martingale} \cite{el1990martingale}), but in a special context, and then discuss the associated stochastic integration and some measurability issues.
	Let us consider the Polish space $A$, and an abstract filtered probability space $(\Om^{\star}, \Fc^{\star}, \F^{\star}, \P^{\star})$, equipped with a random measure $\nu_t(\mathrm{d}a) \mathrm{d}t$ on $[0,T] \x A$, where $t \longmapsto \nu_t(\mathrm{d}a)$ is $\Pc(A)$--valued predictable process.
	Let $\Pc^{\F^{\star}}$ denote the predictable $\sigma$--field w.r.t. the filtration $\F^{\star}$,
	and $\Mc(A)$ the space of all Borel signed measure on $A$.

	\begin{definition} \label{def:notion_martingale-measure}
		An $\Mc(A)$--valued process $(N_t(\mathrm{d}a))_{t \in [0,T]}$ is called an $(\F^{\star},\P^{\star})$--martingale measure of intensity $\nu_t(\mathrm{d}a) \mathrm{d}t$ if

		\medskip
		$(i)$ for all $B \in \Bc(A)$, $(N_t(B))_{t \in [0,T]}$ is a $(\F^{\star},\P^{\star})$--martingale with quadratic variation $\int_0^{\cdot} \nu_s(B) \mathrm{d}s$, and with $N_0(B)=0;$

		\medskip
		$(ii)$ let $B_1, B_2 \in \Bc(A)$ be such that $B_1 \cap B_2=\emptyset$, then $(N_t(B_1))_{t \in [0,T]}$ and $(N_t(B_2))_{t \in [0,T]}$ are two orthogonal martingales.
	\end{definition}

	Given an $(\F^{\star},\P^{\star})$--martingale measure $(N_t(\mathrm{d}a))_{t \in [0,T]}$ of intensity $\nu_t(\mathrm{d}a) \mathrm{d}t$,
	and a $\Pc^{\F^{\star}} \otimes \Bc(A)$--measurable function $f:[0,T] \x \Om^{\star} \x A \longrightarrow \R$ such that 
	\[
		\E^{\P^{\star}} \bigg[\iint_{[0,T] \times A} |f(s,a)|^2 \nu_s(\mathrm{d}a) \mathrm{d}s \bigg]< \infty,
	\]
	one can first approximate $f$ by a sequence $(f^m)_{m \ge 1}$ of simple functions of the form $f^m(s, a) := \sum_{k=1}^m f_k^m \mathbf{1}_{(s^m_k, t^m_k]}(s) \mathbf{1}_{B^m_k}(a)$,
	where $B^m_k \subset \Bc(A)$,
	\[
		s^m_k < t^m_k,
		~
		f^m_k ~\mbox{is}~\Fc^\star_{s^m_k}\mbox{--measurable, for all}~k=1, \dots,m,~\mbox{and}~
		\Lim_{m\to \infty} \E^{\P^{\star}} \bigg[\iint_{[0,T] \times A} \big| f(s,a)-f^m(s,a) \big|^2 \nu_s(\mathrm{d}a) \mathrm{d}s \bigg]=0.
	\]
	Then one can define the stochastic integral, for $t \in \Q$,
	\[
		N_t(f)
		=
		\iint_{[0,t] \times A} f(s,a)N(\mathrm{d}a,\mathrm{d}s)
		:=
		\Lim_{m\to \infty} N_t(f^m)
		:=
		\Lim_{m\to \infty} \sum_{k=1}^m f_k^m \big( N_{t^m_{k}\wedge t}(B_k^m)- N_{s^m_k \wedge t}(B_k^m) \big),\;\mbox{with the limit in}\;\L^2,
	\]
	and then, for all $t \in [0,T]$
	\[
		N_t(f) 
		=
		\iint_{[0,T] \times A} f(s,a)N(\mathrm{d}a,\mathrm{d}s)
		:=
		\limsup_{\Q \ni s \nearrow t} N_s(f).
	\]

	Notice that $(N_t(f))_{t \in [0,T]}$ is an $(\F^{\star}, \P^{\star})$--continuous martingale with quadratic variation $\int_0^\cdot \int_A f(s,a) \nu_s(\mathrm{d}a)\mathrm{d}s$,
	and it is in fact independent of the approximating sequence $(f^m)_{m \ge 1}$ (see e.g. \cite[Section 1]{el1990martingale}).

	\vspace{0.5em}
	
	Let us now consider another abstract measurable space $(E, \Ec)$,
	a family of probability measures $(\P^{\star}_e)_{e \in E}$ on $(\Om^{\star}, \Fc^{\star}, \F^{\star})$ under which $N$ is a martingale measure with intensity $\nu_t(\mathrm{d}a) \mathrm{d}t$,
	and the random measure $\nu_t(\mathrm{d}a) \mathrm{d}t$ has the same distribution under each $\P^\star_e$. 
	In addition, the family $(\P^{\star}_e)_{e \in E}$ satisfies that,
	for all bounded Borel function $\varphi: \Om^\star \x E \longrightarrow \R$,
	\begin{align} \label{eq:cond-measure}
		E \ni e \longmapsto \int_{\Om^\star} \varphi(\om^\star,e) \P^\star_e(\mathrm{d}\om^\star) \in \R\;\mbox{is}\;\Ec\mbox{--measurable}.
	\end{align}
	Let $f: [0,T] \x \Om^\star \x A \x E \to \R$ be $\Pc^{\F^{\star}} \otimes \Bc(A) \x \Ec$--measurable function such that
	\begin{equation} \label{eq:fe_integ}
		\E^{\P^{\star}_e} \bigg[ \iint_{[0,T] \times A} |f^e(s,a)|^2 \nu_s(\mathrm{d}a) \mathrm{d}s \bigg]< \infty,
		~\mbox{for each}~e \in E.
	\end{equation}

	\begin{lemma} \label{eq:StochInt_N}
		One can construct a family of processes $ \{ (N_t(f^e) )_{t \in [0,T]} \big\}_{ e \in E}$ such that
		\begin{equation} \label{eq:meas_sto_int_N}
			(t, \om^\star, e) \longmapsto N_t(f^e, \om^{\star}) ~\mbox{\rm is}~\Pc^{\F^\star} \otimes \Ec \mbox{\rm --measurable},
		\end{equation}
		and
		\begin{equation} \label{eq:sto_int_Ne}
			N_t(f^e, \om^{\star}) = \bigg(\iint_{[0,t] \times A} f^e(s,a)N(\mathrm{d}a,\mathrm{d}s) \bigg)(\om^\star), ~t \in [0,T],
			~\P^\star_e\mbox{\rm --a.s. for each}~e \in E.
		\end{equation}
	\end{lemma}
	\begin{proof}
		Let us first consider the simple functions $f: [0,T] \x \Om^\star \x A \x E \to \R$ in form $f(s, a) := \sum_{k=1}^m f_k \mathbf{1}_{(s_k, t_k]}(s) \mathbf{1}_{B_k}(a)$,
		where for each $k=1, \dots, m$,
		\[
			s_k < t_k,
			~
			f_k: \Om^\star \x E \to \R ~\mbox{is}~\Fc^\star_{s_k}\otimes \Ec\mbox{--measurable, and}~
			B_k \in \Bc(A).
		\]
		Then it is clear that
		\[
			(t, \om^\star, e) \longmapsto N_t(f^e, \om^{\star}) 
			:= 
			\sum_{k=1}^m f_k(\om^{\star}, e) \big( N_{t_{k}\wedge t}(\om^{\star}, B_k)- N_{s_k \wedge t}(\om^{\star}, B_k) \big)
			~\mbox{is}~\Pc^{\F^\star} \otimes \Ec \mbox{--measurable}.
		\]
		Next, let $f_1, f_2: [0,T] \x \Om^\star \x A \x E \to \R$ be two bounded $\Pc^{\F^{\star}} \otimes \Bc(A) \x \Ec$--measurable functions.
		Assume that both $f_1$ and $f_2$, one can construct the stochastic integrals satisfying \eqref{eq:meas_sto_int_N} and \eqref{eq:sto_int_Ne},
		then it is clear that for $f := f_1 \pm f_2$, $N_t(f) := N_t(f_1) \pm N_t(f_2)$ satisfies also \eqref{eq:meas_sto_int_N} and \eqref{eq:sto_int_Ne}.
		
		\medskip
		
		Further, let $(f_m)_{ m \ge 1}$ be a sequence of positive bounded functions increasely converging to bounded function $f$ pointwisely, 
		all $f, f_m$ are $\Pc^{\F^{\star}} \otimes \Bc(A) \x \Ec$--measurable functions,
		and for each $m \ge 1$, one can construct $N_t(f_m)$ satisfying \eqref{eq:meas_sto_int_N} and \eqref{eq:sto_int_Ne}.
		Then it is clear that for each $e \in E$,
		\[
			N_t(f_m^e) \longrightarrow \iint_{[0,t] \x A} f^e(s,a) N(\mathrm{d}a,\mathrm{d}s) ~\mbox{in}~\L^2(\P^\star_e),~\mbox{as}~ m \longrightarrow \infty.
		\]
		Following \cite[Lemma 3.2.]{neufeld2014measurability}, toegether with Condition \eqref{eq:cond-measure}, one can find a family of sub-sequence $(m_k(e))_{k \ge 1, e \in E}$ which is $\Ec$--measurable and
		\[
			N_t(f^e) := \limsup_{k \to \infty} N_t(f_{m_k(e)}^e) = \iint_{[0,t] \x A} f^e(s,a) N(\mathrm{d}a,\mathrm{d}s),~\P^\star_e \mbox{--a.s., for each}~e \in E.
		\]
		In other words, one can choose a version $N_t(f)$ satisfying  \eqref{eq:meas_sto_int_N} and \eqref{eq:sto_int_Ne}.
		By the monotone class theorem, 
		it follows that the statement holds true for all bounded functions $f: [0,T] \x \Om^\star \x A \x E \to \R$ which is $\Pc^{\F^{\star}} \otimes \Bc(A) \x \Ec$--measurable.
		
	\medskip
		
		{\color{black}
		Finally, let $f: [0,T] \x \Om^\star \x A \x E \to \R$ be a $\Pc^{\F^{\star}} \otimes \Bc(A) \x \Ec$--measurable function satisfying \eqref{eq:fe_integ}. For each $m \in \N^{\star},$ define  $f_m:=f\mathbf{1}_{|f| \le m},$  then $(f_m)_{m \in \N^{\star}}$ is a sequence  of $\Pc^{\F^{\star}} \otimes \Bc(A) \x \Ec$--measurable functions satisfying
		\[
			\E^{\P^{\star}_e} \bigg[ \iint_{[0,T] \times A} |f^e(s,a) -f^e_m(s,a)|^2 \nu_s(\mathrm{d}a) \mathrm{d}s \bigg]< \infty,
			~\mbox{for each}~e \in E.
		\]
		Then it is enough to use the arguments in \cite[Lemma 3.2]{neufeld2014measurability} with the condition \eqref{eq:cond-measure} again to define $N_t(f^e)$ as limit of $N_t(f^e_m)$, which satisfies  \eqref{eq:meas_sto_int_N} and \eqref{eq:sto_int_Ne}.
		}
	\end{proof}

\paragraph*{Proof of Proposition \ref{Proposition:diffusion_McKV-relaxed}}
	Recall that the probability space $(\Om^\star, \Fc^\star, \P^\star)$ is equipped with $2 (n+d)$ i.i.d. martingale measures $(N^{\star,i})_{i=1, \dots, 2(n+d)}$ with intensity $\nu_0(\mathrm{d}a) \mathrm{d}t$ for some diffuse probability measure $\nu_0$ on $A$,
	which is extended on $(\Omh^\star, \Fch^\star, \Ph_{\omb})$ for every $\omb \in \Omb$.
	We will now follow the technical steps in \citeauthor*{el1990martingale} \cite{el1990martingale} to construct the family of martingale measures $(\Nh^{\omb})_{\omb \in \Omb}$ satisfying \eqref{eq:X_repres_Nt}, 
	and then check the measurability property in \eqref{eq:Nt_measurability}.

\medskip

	Let us first denote
	\begin{align} \label{eq:matrix}
		\Sigma(t,\xb,a,\nu)
		:=
	        \begin{pmatrix} 
			\sigma(t,\xb,a,\nu) & 0_{n \x n}  \\ 
			\mathrm{I}_{d} & 0_{d \x n} 
		\end{pmatrix},
		~\mbox{and}~
		\big(\Sigma \Sigma^{\top})^{+}(t,\xb,a,\nu\big)
		:=
		\lim_{\eps \searrow 0} 
			\big( \eps \mathrm{I}_{d+n} + (\Sigma \Sigma^{\top})(t,\xb,a,\nu) \big)^{-1},
	\end{align}
	where $(\Sigma \Sigma^{\top})^{+}$ is the pseudo--inverse of $\Sigma \Sigma^{\top}$. Then for all bounded Borel measurable function $f: [0,T] \x A \longrightarrow \R $, let
	\[
		\Gamma^{\omb}(s,f)
		:=
		\int_A \Sigma\Sigma^{\top} (s, \Xh, a, \mu_s(\omb))f(s,a) \Lambdah_s (\mathrm{d}a),
		~\mbox{and its pseudo--inverse}~
		\Gamma^{\omb, +}(s,f)
		:=
		\limsup_{\eps \searrow 0} 
			\big( \eps \mathrm{I}_{d+n} +
			\Gamma(s,\omb,f)
		\big)^{-1}.
	\]
	Denote also by $\mathbf{1}$ the constant function on $[0,T] \x A$ which equals to $1$.
	Furthermore, let $\pi_i: \R^{n+d} \longrightarrow \R$, $i = 1, \dots, n+d$, be the projection function defined by 
	$\pi_i( (z) := z_i$ for every $z := (z_1, \dots, z^{n+d})$,
	and $\Mh^{\omb, i} := \Mh^{\pi_i, \mu(\omb)}$ be the martingale defined in \eqref{eq:Mvarphi},
	whose quadratic variation process is given by
	\begin{align*}
		\big \langle \Mh^{\omb, i}, \Mh^{\omb, j} \big \rangle_t
		=
		\int_0^t \Gamma^{\omb}_{i,j}(s, \mathbf{1}) \mathrm{d}s,
		~t \in [0,T],~\Ph_{\omb} \mbox{--a.s.}
	\end{align*}
	
	$(i)$ 
	By \cite[Theorem III-2.]{el1990martingale}, there exists a $\Pc^{\Fh} \otimes \Bc(A)$--measurable function $\varphi: [0,T] \x \Omh \x A \longrightarrow A$ such that
	\begin{align*}
		\Lambdah_s(\hat \om, B)=\int_A \mathbf{1}_{B}(\varphi( s, \hat \om, a)) \nu_0 (\mathrm{d}a),
		~\mbox{for all}~
		(s, \hat \om) \in [0,T] \x \Omh,
		~B \in \Bc(A).
	\end{align*}
	This allows to define, for every $\omb \in \Omb$,
	two independent martingale measures 
	$(\Nt^{\star,\omb,i})_{1\leq i \leq n+d}$
	and $(\Nt^{\star,\omb,i})_{n+d+1\leq i \leq2(n+d)}$ from $(N^{\star,i})_{1\leq i\leq n+d}$ and $(N^{\star,i})_{n+d+1\le i\le 2(n+d)}$ as follows.
	For each $\omb \in \Omb$, let us define
	for all $B \in \Bc(A)$, $i=1,\dots, n+d$,
	\[
		\Nt^{\star, \omb, i}_t (B)
		:= 
		\sum_{k=1}^{n+d}\iint_{[0,t]\times A} 1_B(\varphi(s,a))\Sigma_{ik}
		(s, \Xh, \varphi(s,a),\mu_s(\omb)) N^{\star,k}(\mathrm{d}a,\mathrm{d}s),
		~t \in [0,T],~
		{\Ph}_{\omb}\mbox{--a.s.},
	\]
	and for all $B \in \Bc(A)$, $i=n+d+1, \dots, 2(n+d)$,
	\[
		\Nt^{\star, \omb, i}_t(B) 
		:= 
		\sum_{k=n+d+1}^{2(n+d)}\iint_{[0,t]\times A} 1_B(\varphi(s,a))\Sigma_{ik}
		(s, \Xh, \varphi(s,a),\mu_s(\omb)) N^{\star,k}(\mathrm{d}a,\mathrm{d}s),
		~t \in [0,T],~
		{\Ph}_{\omb}\mbox{--a.s.}
	\]
	By \cite[Theorem III-3.]{el1990martingale}, $(\Nt^{\star,\omb,i})_{1\leq i\leq n+d}$ and $(\Nt^{\star,\omb,i})_{n+d+1\leq i\leq  2(n+d)}$ are two independent martingale measures with intensity 
	$\Lambdah^{\Sigma, \omb}_t (\mathrm{d}a) \x  \mathrm{d}t$ defined by $\Lambdah^{\Sigma, \omb}_t (B):= \Gamma^{\omb}(t,\mathbf{1}_B)$ for all $B \in \Bc(A)$.
	
	\vspace{0.5em}
	
	Next, we define the martingale measure $(\Nt^{\omb,i})_{i=1, \dots, n+d}$,
	from $(\Nt^{\star,\omb,i}, \Mh^{\omb,i} )_{i=1, \dots, n+d}$ as follows.
	For each bounded $\Pc^{\Fh^\star} \otimes \Bc(A)$--measurable function $f: [0,T] \x \Omh^\star \x A \longrightarrow \R$, 
	and $i=1,\dots, n+d$, let
	\begin{align*}
		\iint_{[0,t]\times A} f(s,a) \Nt^{\omb, i}(\mathrm{d}a,\mathrm{d}s)
		:=&\
		\sum_{k=1}^{n+d} \int_0^t \Big(
			\Gamma^{\omb}(s,f) 
			\big(\Gamma^{\omb,+} \Gamma^{\omb} \Gamma^{\omb,+} \big)(s,\mathbf{1}) 
		\Big)^{i,k} 
		\mathrm{d} \Mh^{\omb, k}_s \\
		& +
		\sum_{k=1}^{n+d} 
		\int_0^t \int_A \Big( 
			f(s,a) I_{n+d} 
			- 
			\Gamma^{\omb}(s,f) \big(\Gamma^{\omb,+} \Gamma^{\omb} \Gamma^{\omb,+} \big)(s,\mathbf{1}) 
			\Big)^{i,k} 
		\Nt^{\star,\omb,k}(\mathrm{d}a,\mathrm{d}s).
	\end{align*}
	Let us refer to the proof of \cite[Proposition III-9., Theorem III-10.]{el1990martingale}) for the fact the above does define a martingale measure $(\Nt^{\star,\omb,i})_{i=1, \dots, n+d}$ with intensity 
	$\Lambdah^{\Sigma, \omb}_t (\mathrm{d}a) \x  \mathrm{d}t$,
	and that it satisfies $\Nt^{i,\omb}_t(A) =\Mh^{\omb, i}_t$, for each $i = 1, \dots, n+d$.

	\vspace{0.5em}
	
	Finally, let
	\[
		\Sigma^{-1} (s,\omb,\widehat{\om},a)
		:=
		\Sigma
		(\Sigma \Sigma^{\top})^+ \Sigma \Sigma^{\top} (\Sigma \Sigma^{\top})^+ (s,\omb,\widehat{\om},a),
	\]	
	we define $(\Nh^{\omb,i})_{i=1,\dots, n+d}$ as follows.
	For every bounded $\Pc^{\Fh^\star} \otimes \Bc(A)$--measurable function $f: [0,T] \x \Omh^\star \x A \longrightarrow \R$, $i=1, \dots, n+d$, 
	let
	\[
		\Nh^{\omb,i}_t(f)
		:=
		\sum_{k=1}^{n+d} \iint_{[0,t]\times A} \!\!\!\! f(s,a) \Sigma^{-1}_{ik}(s,\omb,a) \Nt^{\omb,k}(\mathrm{d}a,\mathrm{d}s)
		+
		\iint_{[0,t]\times A} \!\!\!\! \big( \mathrm{I}_{n+d} - \Sigma\Sigma^{\top} (\Sigma\Sigma^{\top})^+ \big)(s,\omb,a) 
				f(s,a) 
				\Nt^{\star,\omb,n+d+i}(\mathrm{d}a,\mathrm{d}s),
	\]
	where we notice that $\Sigma\Sigma^{\top} (\Sigma\Sigma^{\top})^+$ is the projection from $\R^{n+d}$ to the range of 
	$\Sigma\Sigma^{\top}$.
	It follows then $(\Nh^{\omb,i})_{i = 1, \dots, d}$ is a martingale measure with intensity $\Lambdah_t (\mathrm{d}a) \x  \mathrm{d}t$
	and satisfies \eqref{eq:X_repres_Nt}.
	
	\medskip
	
	$(ii)$ 
	Let us now consider a bounded $\Pc^{\widehat{\H}^\star} \otimes \Bc(A)$--measurable function $f:[0,T] \times \Omb \times \Omh^\star \times A  \longrightarrow \R$.
	By the above explicit construction of $\Nh^{\omb}$, it is clear that one can rewrite the stochastic integral
	\[
		\iint_{[0,t] \times A}  f^{\omb}(s,a) \Nh^{\omb}(\mathrm{d}s, \mathrm{d}a)
		=
		\sum_{i=1}^{2(n+d)} \iint_{[0,t] \x A} \phi^{\omb}(s,a) N^{\star,i}(\mathrm{d}a,\mathrm{d}s) 
		+
		\sum_{i=1}^{n+d} \int_0^t \psi^{\omb}_s \mathrm{d}\Mh^{\omb,i}_s,
		~\Ph_{\omb}\mbox{--a.s.},
	\]
	for some $\Pc^{\widehat{\H}^\star} \otimes \Bc(A)$--measurable function $\phi:[0,T] \times \Omb \times \Omh^\star \times A  \longrightarrow \R$,
	and $\Pc^{\widehat{\H}^\star}$--measurable function $\psi:[0,T] \times \Omb \times \Omh^\star  \longrightarrow \R$.
	Then one can apply the same arguments as in Lemma \ref{eq:StochInt_N} to choose a good version of the stochastic integral s.t.
	\[
		(t, \omb, \hat \om^\star)
		\longmapsto	
			\bigg(\iint_{[0,t] \times A}  f^{\omb} (s, a) \Nh^{\omb}(\mathrm{d}a,\mathrm{d}s) \bigg)(\hat \om^\star)\; \mbox{\rm is}\; 
		\Pc^{\widehat{\H}^\star}\mbox{\rm --measurable},
	\]
	i.e.  \eqref{eq:Nt_measurability} holds true.
	\qed

\subsection{Proof of Proposition \ref{Equivalence-Proposition_General} (general case)}
\label{subsec:Equivalence-Proposition}

	We finally provide the proof of the \Cref{Equivalence-Proposition_General} in the general case,
	where  the coefficients $(b,\sigma,\sigma_0)$ can be simplified to be
	\begin{align} \label{hypothesis_generalUncontrolled}
		(b,\sigma)(t,\xb,\nub,a)
		=
		(b,\sigma)(t,\xb,\nu,a),\;\mbox{and}\;
		\sigma_0(t,\xb,\nub,a):=\sigma_0(t,\xb,\nu),
	\end{align}
	for all $(t,\xb,\nub,a) \in [0,T] \x \Cc^n \x \Pc(\Cc^n \x A) \x A$ with $\nu(\mathrm{d}\xb):=\nub(\mathrm{d}\xb,A)$.
 	The main idea of the proof is the same as for the case where $\sigma_0$ is a constant function,
	and will provide an outline of the proof.
	
	\medskip

	In a nutshell, we aim to approximate the relaxed control $\Pb$ by weak control rules on $\Omb$, where $(X,Y,B,\muh)$ satisfies
	\begin{align} \label{eq:Y-X_equality}
		Y_{\cdot}=X_\cdot-\int_0^\cdot \sigma_0(s, X, \mu ) \mathrm{d}B_s,\; \Pb\text{\rm--a.s.},
	\end{align}    
	and for $\Pb$--a.e. $\omb \in \Omb$,
	the canonical processes $(\Xh,\Yh,\Lambdah,\Wh)$ satisfies $\Wh_\cdot=\int_0^\cdot \int_A \Nh^{\omb}(\mathrm{d}a,\mathrm{d}s),~\Ph_{\omb} \mbox{\rm --a.s.}$ and
	\begin{align} \label{eq:initial-equation}
		\Yh_t
		=
		\Xh_0 
		+
		\iint_{[0,t]\times A}
		b \big(r, \Xh, \mu(\omb), a \big) \Lambdah_r(\mathrm{d}a, \mathrm{d}r)
		+
		\iint_{[0,t]\times A}
		\sigma\big(r, \Xh, \mu(\omb), a \big) \Nh^{\omb}(\mathrm{d}a,\mathrm{d}r),\;\mbox{for all}\;t \in [0,T],\;\Ph_{\omb} \mbox{\rm --a.s.}
	\end{align}
	
	\vspace{0.5em}
	\underline{\it Step $1$}: In this first step, we rewrite \eqref{eq:initial-equation} as an equation that takes into account only $\Xh,$ and not $\Yh.$
	To do this, observe that, we can find a Borel measurable function $\Ic: (t,\xb,\pi,\bb) \in [0,T] \x \Cc^n \x \Pc(\Cc^n) \x \Cc^\ell \longrightarrow \Ic(t,\xb,\pi,\bb) \in \R^n$ satisfying $\Ic(t,\xb,\pi,\bb)=\Ic(t,\xb_{t \wedge \cdot},\pi \circ (\Xh_{t \wedge \cdot})^{-1},\bb_{t \wedge \cdot})$ and 
	\begin{align} \label{def_sto-integral}
	    \Ic \big(t,X,\mu,B\big)=\int_0^t \sigma_0(r,X,\mu) \mathrm{d}B_r,\;\Pb\mbox{--a.s.}
	\end{align}
	Using \eqref{eq:muh_property}, i.e. $\muh_t (\omb)=\Pb^{\Gcb_T}_{\omb} \circ (X_{t \wedge \cdot},Y_{t \wedge \cdot},W, \Lambda^t)^{-1},\;\mbox{for}\;\Pb\mbox{--a.e.}\;\omb \in \Omb$, $t \in [0,T]$,
	we obtain an equivalent formulation of \eqref{eq:Y-X_equality} on $\Omh,$ 
	\[
		\Yh_{\cdot}=\Xh_\cdot-\Ic \big(\cdot,\Xh,\mu(\omb),B(\omb)\big),\;\Ph_{\omb}\mbox{--a.s},\;\mbox{for}\;\Pb\;\mbox{--a.e.}\;\omb \in \Omb.
	\]
	and then, a reformulation of \eqref{eq:initial-equation} involving only $\Xh$,
	and one can consider $\Ic\big(\cdot, \Xh, \mu(\omb),B(\omb)\big)$ as a `conditional` stochastic integral w.r.t $B$ given the $\sigma$--field $\Gcb_T.$

	\medskip

	Further, for any $\R^n$--valued $\widehat{\H}^\star$--adapted continuous process $(U_t)_{t \in [0,T]}$, there exists a measurable map $\phi: \Omb \x \Cc^n \x \Cc^n \x \M \x \Cc^d \x \Om^\star \longrightarrow \Cc^n$ such that $U^{\omb}_t(\hat \om, \om^\star) = \phi_t\big(\omb,\Xh_{t \wedge \cdot}(\hat \om),\Yh_{t \wedge \cdot}(\hat \om), \Lambdah^t(\hat \om), \Wh_{t \wedge \cdot}(\hat \om), \om^\star \big)$, for all $(t,\omb,\hat \om,\om^\star) \in [0,T] \x \Omb \x \Omh \x \Om^\star$.
	notice that, for $\Pb\mbox{--a.e.}\;\omb \in \Omb$,
	one has $\E^{\Ph_{\omb}} \big[\sup_{t \in [0,T]} |\widehat{S}^{\omb}_t|^p \big] < \infty$,
	then it follows by \eqref{eq:muh_property} that, for $\Pb\mbox{--a.e.}\;\omb \in \Omb$,
	\begin{align} \label{eq:fund-property_I}
		\Lc^{\Ph_{\omb}} \Big( \Ic \big(\cdot,\widehat{S}^{\omb},\beta(\omb),B(\omb)\big), \widehat{S}^{\omb},\beta(\omb),B(\omb)  \Big)
		=
		\big( \Pb^{\Gcb_T}_{\omb} \otimes \P^\star \big) \circ \bigg( \int_0^\cdot \sigma_0\big(s, S\big(\omb,X,Y, \Lambda, W \big), \beta \big) \mathrm{d}B_s,S\big(\omb,X,Y, \Lambda, W \big),\beta,B  \bigg)^{-1},
	\end{align}
	for all Borel measurable functions $\beta: \Omb \longrightarrow \Pc(\Cc^n)$ such that $(\beta \circ (\Xh_{t \wedge \cdot})^{-1})_{t \in [0,T]}$ is a $\Gb$--predictable process satisfying $\E^{\Pb}[\int_{\Cc^n} \|\xb\|^p \beta(\mathrm{d}\xb)]< \infty.$

\medskip
	\underline{\it Step $2$}: 
	We now approximate $\Xh$ using $\Yh$ and under each $\Ph_{\omb}$,
	where the arguments are almost the same as in the proof of \Cref{Equivalence-Proposition_General} when $\sigma_0$ is constant. 
	More precisely, for $k \ge 1$,
	there exists $a^k_1,\dots,a^k_k \in A$ together with a sub--division $t^k_0=0<\dots<t^k_k=T$,
	as well as $\Pc(A)$--valued $\Fh^\star$--predictable processes $(\Lambdah^{k,1},\dots,\Lambdah^{k,k})$, 
	which are constant on each interval $[t^k_i,t^k_{i+1}],$ 
	and $(\Fh,\Ph^\star_{\omb})$--independent Brownian motions  $(\Zh^{\omb,k,1},\dots,\Zh^{\omb,k,k})$ 
	for $\Pb$--a.e. $\omb \in \Omb$.
	Moreover, let 
	\begin{align*}
		\Yh^{\omb,k}_t
		:=
		\Xh_0
		+
		\sum_{i=1}^k \int_{0}^{t} b(r, \Xh, \mu(\omb), a^k_i) \Lambdah^{k,i}_r \mathrm{d}r
		+
		\int_{0}^{t} \sigma(r, \Xh, \mu(\omb), a^k_i) \sqrt{\Lambdah^{k,i}_r}\mathrm{d}\Zh^{\omb,k,i}_{r},\;t \in [0,T],\;\Ph_{\omb}\mbox{--a.s.},
	\end{align*}
	and $\Xh^{\omb,k}_{\cdot}:=\Yh^{\omb,k}_{\cdot} + \Ic \big(\cdot,\Xh,\mu(\omb),B(\omb)\big)$, one has
	\begin{align*}
		\Lim_{k\to\infty} \E^{\Ph_{\omb}} \bigg[ \sup_{t \in [0,T]} \big| \Yh^{\omb,k}_t-\Yh_t \big|^p \bigg]=0,\;\mbox{then}\;\Lim_{k\to\infty} \E^{\Ph_{\omb}} \bigg[ \sup_{t \in [0,T]} \big| \Xh^{\omb,k}_t-\Xh_t \big|^p \bigg]=0.
	\end{align*}
	Further, for each $k \ge 1$
	\[
		(t, \omb, \omt^\star) \longmapsto
		\Big( \Yt^{\omb,k}_{t \wedge \cdot}(\hat \om^\star),\Xh^{\omb,k}_{t \wedge \cdot}(\hat \om^\star),(\Lambdah^{1,k})^t(\hat \om^\star),\dots, (\Lambdah^{k,k})^t(\hat \om^\star),\Zh^{\omb,k,1}_{t \wedge \cdot}(\hat \om^\star),\dots,\Zh^{\omb,k,k}_{t \wedge \cdot}(\hat \om^\star)   \Big)
		~\mbox{is}~
		\Pc^{\widehat{\H}^\star}\mbox{--measurable}.
	\]

	\medskip
	
	Next, let us introduce $X^{k,\circ}$ an $\R^n$--valued $\widehat{\H}^\star$--adapted continuous process satisfying,
	for $\Pb$--a.e. $\omb \in \Omb$,
	$X^{\omb,k,\circ}$ is the unique strong solution of:
	\begin{align} \label{eq:random-SDE}
		\Xh^{\omb,k,\circ}_t
		=&\
		\Xh_0
		+
		\sum_{i=1}^k\int_{0}^{t} b(r, \Xh^{\omb,k,\circ}, \Ph_{\omb} \circ (\Xh^{\omb,k,\circ})^{-1}, a^k_i)  \Lambdah^{k,i}_r \mathrm{d}r
		+
		\int_{0}^{t} \sigma(r, \Xh^{\omb,k,\circ}, \Ph_{\omb} \circ (\Xh^{\omb,k,\circ})^{-1}, a^k_i) \sqrt{\Lambdah^{k,i}_r}\mathrm{d}\Zh^{\omb,k,i}_{r} \nonumber
		\\
		&+
		\Ic \big(t,\Xh^{\omb,k,\circ},\Ph^\star_{\omb} \circ (\Xh^{\omb,k,\circ})^{-1},B(\omb)\big),
		~t \in [0,T],~\Ph_{\omb}\mbox{--a.s.},
	\end{align}
	and $\E^{\Ph_{\omb}} \big[\|X^{\omb,k,\circ} \|^p \big] < \infty$.
	The existence and uniqueness of solution to \eqref{eq:random-SDE} is just an extension of the classical Picard iteration scheme as in \cite[Theorem A.3.]{djete2019mckean}, adapted to this context.

	\medskip

	We next define $\Yh^{\omb,k,\circ}_\cdot:=\Xh^{\omb,k,\circ}_\cdot-\Ic \big(\cdot,\Xh^{\omb,k,\circ},\Ph^\star_{\omb} \circ (\Xh^{\omb,k,\circ})^{-1},B(\omb)\big)$.
	By the same arguments as in the constant $\sigma_0$ case (using \eqref{eq:Nt_measurability} and Picard iteration argument), 
	we can deduce that, for each $k \ge 1$,
	\[
		(t, \omb, \hat \om^\star) \longmapsto 
		\big( \Xh^{\omb,k,\circ}_{t \wedge \cdot}(\hat \om^\star), \Yh^{\omb,k,\circ}_{t \wedge \cdot}(\hat \om^\star), (\Lambdah^{1,k})^t(\hat \om^\star),\dots, (\Lambdah^{k,k})^t(\hat \om^\star),\Zh^{\omb,k,1}_{t \wedge \cdot}(\hat \om^\star),\dots,\Zh^{\omb,k,k}_{t \wedge \cdot}(\hat \om^\star)   \big)
		~\mbox{is}~
		\Pc^{\widehat{\H}^\star}\mbox{--measurable}.
	\]
  	Then by the definition of $\Ic$ in \eqref{def_sto-integral}, together with \eqref{eq:fund-property_I}, it is straightforward to check that
	\begin{align} \label{eq:property_I}
		&\ \int_{\Omb} \E^{\Ph^\star_{\omb}} \Big[ \Big|\Ic\big(t,\Xh^{\om,k,\circ}, \Lc^{\Ph^\star_{\omb}}(\Xh^{\om,k,\circ}),B(\omb) \big)- \Ic\big(t,\Xh, \mu(\omb),B(\omb) \big) \Big|^p\Big] \Pb(\mathrm{d}\omb) \nonumber
		\\
		\le&\ 
		\int_{\Omb} \int_0^t \E^{\Ph^\star_{\omb}} \Big[ \Big|\sigma_0 \big(r,\Xh^{\om,k,\circ}, \Lc^{\Ph^\star_{\omb}}(\Xh^{\om,k,\circ}) \big)- \sigma_0 \big(r,\Xh, \mu(\omb)\big) \Big|^p\Big] \mathrm{d}r\Pb(\mathrm{d}\omb).
	\end{align}
	It follows that, for some constant $C >0$ independent of $k$,
	\[
		\int_{\Omb} \E^{\Ph^\star_{\omb}} \bigg[\sup_{s \in [0,T]}|\Xh^{\omb,k,\circ}_s-\Xh^{\omb,k}_s|^p \bigg] \Pb(\mathrm{d}\omb)
		\le
		C
		\int_0^T \int_{\Omb} \E^{\Ph^\star_{\omb}} \bigg[\sup_{s \in [0,t]}|\Xh^{\omb,k}_s-\Xh_s|^p \bigg] \Pb(\mathrm{d}\omb) \mathrm{d}t,
	\]
	and further 
	\[
		\lim_{k \to \infty} \int_{\Omb} \E^{\Ph^\star_{\omb}} \big[\sup_{s \in [0,T]}|\Xh^{\omb,k,\circ}_s-\Xh_s|^p \big] \Pb(\mathrm{d}\omb)=0.
	\]

	\medskip \underline{\it Step $3$}: 
	Finally, let us construct the approximating weak control rules,
	where the arguments are the same as in the constant $\sigma_0$ case. 
	For each $k \ge 1$ and $m \ge 1$,
	there exist a sequence of Borel sets $(I^{k,1}_m,\dots,I^{k,k}_m)$ such that $\cup_{i=1}^k I^{k,i}_m=[0,T],$ 
	and for $\Pb$--a.e. $\omb \in \Omb,$ 
	$(\Wh^{\omb,m,1},\dots,\Wh^{\omb,m,k})$ is $(\Fh^\star,\Ph_{\omb})$--martingales with quadratic variation $\langle \Wh^{\omb,m,i} \rangle_{\cdot}=\hat c^{m,i}_{\cdot}:=\int_0^\cdot \mathbf{1}_{I^{k,i}_m}(r)\mathrm{d}r,$, $i\in\{1,\dots,k\}$. 
	Furthermore, it holds that
	\begin{align} \label{eq:conv-laststep}
		\Lim_{m\to\infty} \Big( \Wh^{\omb,m,i},\hat c^{m,i} \Big)
		=
		\bigg( \int_0^\cdot \sqrt{\Lambdah^{k,i}_r}\mathrm{d}\Zh^{\omb,m,i}_{r}, \int_0^\cdot \Lambdah^{k,i}_r\mathrm{d}r \bigg),
		~i \in \{1,\dots,k\},\;\Ph_{\omb}\mbox{\rm --a.e.}
	\end{align}

	Next, let $\Xh^{\omb, k,m}_t(\omt^\star)$ be a $\R^n$--valued $\widehat{\H}^\star$--adapted process
	such that, for $\Pb$--a.e. $\omb \in \Omb$,
	$\Xh^{\omb,k,m}$ is the unique strong solution of
	\begin{align*}
        		\Xh^{\omb,k,m}_t
		=&\
		\Xh_0
		+
		\sum_{i=1}^k\int_{0}^{t} b(r, \Xh^{\omb,k,m}, \Ph_{\omb} \circ (\Xh^{\omb,k,m})^{-1}, a^k_i)  \mathrm{d}\hat c^{m,i}_r
		+
		\int_{0}^{t} \sigma(r, \Xh^{\omb,k,m}, \Ph_{\omb} \circ (\Xh^{\omb,k,m})^{-1}, a^k_i) \mathrm{d}\Wh^{\omb,m,i}_{r}
		\\
		&+
		\Ic \big(\cdot,\Xh^{\omb,k,m},\Ph^\star_{\omb} \circ (\Xh^{\omb,k,m})^{-1},B(\omb)\big),\; t\in[0,T], \;\Ph_{\omb}\mbox{--a.s.}
	\end{align*}
	Let $\Yh^{\omb,k,m}_\cdot:=\Xh^{\omb,k,m}_\cdot-\Ic \big(\cdot,\Xh^{\omb,k,m},\Ph^\star_{\omb} \circ (\Xh^{\omb,k,m})^{-1},B(\omb)\big)$,
	it follows that
	\[
		(t, \omb,\hat \om^\star) \longmapsto
		 \Big(\Yh^{\omb,k,m}_{t \wedge \cdot}(\hat \om^\star),\Xh^{\omb,k,m}_{t \wedge \cdot}(\hat \om^\star),\Wh^{\omb,m,1}_{t \wedge \cdot}(\hat \om^\star),\dots,\Wh^{\omb,m,k}_{t \wedge \cdot}(\hat \om^\star)  \Big)
		~\mbox{is}~
		\Pc^{\widehat{\H}^\star}\mbox{--measurable}.
	\]
	Define, for each $m \ge 1$,
	a probability on $\Omb^{\star,k}:=\Cc^n \x \Cc^n \x (\Cc)^k \x (\Cc^d)^k \x \Cc^\ell \x \Pc(\Cc^n \x \Cc^n \x (\Cc)^k \x (\Cc^d)^k)$ by
	\begin{equation} \label{eq:def_Pb_star_m}
		\Pb^{\star,m}
		:=
		\int_{\Omb}\Lc^{\Ph_{\omb}} \Big(\Xh^{\omb,k,m},\Yh^{\omb,k,m},\hat c^{m,1},\dots,\hat c^{m,k}, \Wh^{\omb,m,1},\dots,\Wh^{\omb,m,k},B(\omb),\muh^{m}(\omb) \Big)\Pb(\mathrm{d}\omb), 
	\end{equation}
	where $\muh^m(\omb):=\Lc^{\Ph_{\omb}} \Big(\Xh^{\omb,k,m},\hat c^{m,1},\dots,\hat c^{m,k}, \Wh^{\omb,m,1},\dots,\Wh^{\omb,m,k} \Big).$

\medskip
	Similarly to \Cref{lemma:estimates}, by using an inequality of type \eqref{eq:property_I}, we get, for some constant $C> 0$
        \[
		\sup_{m \ge 1}
		\int_{\Omb}\E^{\Ph_{\omb}}
		\bigg[\sup_{t \in [0,T]}|\Xh^{\omb,k,m}_t|^{p^\prime} + \sup_{t \in [0,T]}|\Yh^{\omb,k,m}_t|^{p^\prime}
		\bigg] \Pb(\mathrm{d}\omb)
		\le 
		C 
		\bigg(1+\int_{\R^n}|x|^{p^\prime}\nu(\mathrm{d}x) \bigg) 
		< 
		\infty.
	\]
	Therefore, the sequence $(\Pb^{\star,m})_{m \ge 1}$ is relatively compact for the Wasserstein metric $\Wc_p.$	
	Along a possible subsequence $(m_j)_{j \ge 1}$, one has
	\begin{equation} \label{eq:cvg_P_star_m}
		\Lim_{j\to\infty} \Pb^{\star,m_j}
		=
		\Lc^{\Pb^\star} \Big( X^\star,Y^\star,c^{1,\star},\dots,c^{k,\star}, W^{1,\star},\dots,W^{k,\star},B^\star,\muh^{\star} \Big), \mbox{under}~\Wc_p,
	\end{equation}
	for some random elements $\big( X^\star,Y^\star,c^{1,\star},\dots,c^{k,\star}, W^{1,\star},\dots,W^{k,\star},B^\star,\muh^{\star} \big)$ in $(\Omb^\star, \Fb^\star, \Pb^\star)$. 
	Now, using \eqref{eq:fund-property_I}, it follows that
	\[
		Y^\star_{\cdot}=X^\star_\cdot-\int_0^\cdot \sigma_0(s, X^\star, \Lc^{\P^\star}(X^\star | B^\star,\muh^\star ) ) \mathrm{d}B_s,\; \Pb^\star\text{\rm--a.s.}
	\]
	Let us define, for all $t \in [0,T]$, 
	\[
		\muh^\star_t:=\muh^\star \circ \big( \Xh^\star_{t \wedge \cdot},\Yh^\star_{t \wedge \cdot},\hat c^{1,\star}_{t \wedge \cdot},\dots,\hat c^{k,\star}_{t \wedge \cdot}, \Wh^{1,\star}_{t \wedge \cdot},\dots,\Wh^{k,\star}_{t \wedge \cdot} \big)^{-1},
	\]
	where $(\Xh^\star,\Yh^\star,\hat c^{1,\star},\dots,\hat c^{k,\star}, \Wh^{1,\star},\dots,\Wh^{k,\star})$ is the canonical processes on $\Cc^n \x \Cc^n \x (\Cc)^k \x (\Cc^d)^k$,
	we obtain that
	\begin{equation} \label{eq:H_Hypo_star}
		\muh^\star_t
		=
		\Lc^{\P^\star} \Big( X^\star_{t \wedge \cdot},Y^\star_{t \wedge \cdot},c^{1,\star}_{t \wedge \cdot},\dots,c^{k,\star}_{t \wedge \cdot}, W^{1,\star}_{t \wedge \cdot},\dots,W^{k,\star}_{t \wedge \cdot} \big| \muh^\star,B^\star \Big),\;\P^\star\mbox{--a.s}.
	\end{equation}

	In addition, by the definition of $\Pb^{\star, m}$ in \eqref{eq:def_Pb_star_m} 
	together with the convergence results \eqref{eq:conv-laststep} and \eqref{eq:cvg_P_star_m}, 
	it follows that, for $\Pb^\star$--a.e. $\omb \in \Omb^\star,$ and for all $t \in [0,T]$
	\begin{align*}
		\Yh^{\star}_t
		=
		\Xh^\star_0
		+
		\sum_{i=1}^k\int_{0}^{t} b(r, \Xh^{\star}, \muh^\star(\omb) \circ (\Xh^{\star})^{-1}, a^k_i)  \mathrm{d}\hat c^{i,\star}_r
		&+
		\int_{0}^{t} \sigma(r, \Xh^{\star}, \muh^\star(\omb) \circ (\Xh^{\star})^{-1}, a^k_i) \mathrm{d}\Wh^{i,\star}_{r},\;\muh^\star(\omb)\mbox{--a.s.}
	\end{align*}
	Using \eqref{eq:H_Hypo_star},
	one has $\Yh^\star_{\cdot}=\Xh^\star_\cdot-\Ic \big(\cdot,\Xh^\star,\muh^\star(\omb) \circ (\Xh^{\star})^{-1},B^\star(\omb)\big),\;\Pb^\star_{\omb}\mbox{--a.s},\;\mbox{for}\;\Pb^\star\;\mbox{--a.e.}\;\omb \in \Omb^\star$.
	Then by \eqref{eq:conv-laststep}, one deduces that
	\[
		\Lc^{\Pb} \Big(B, \widehat{\beta} \circ \big(\Xh_0,\hat c^{1,\star},\dots,\hat c^{k,\star}, \Wh^{1,\star},\dots,\Wh^{k,\star} \big)^{-1} \Big)=\Lc^{\Pb^\star} \Big(B^\star,\muh^\star \circ \big(\Xh^\star_0,\hat c^{1,\star},\dots,\hat c^{k,\star}, \Wh^{1,\star},\dots,\Wh^{k,\star} \big)^{-1} \Big),
	\]
	where
	\[
		\widehat{\beta}(\omb):=\Lc^{\Ph^\star_{\omb}}\bigg( \Xh^{\omb,k,\circ}, \Yh^{\omb,k,\circ}, \int_0^\cdot \Lambdah^{k,1}_r\mathrm{d}r,\dots,\int_0^\cdot \Lambdah^{k,k}_r\mathrm{d}r, \int_0^\cdot \sqrt{\Lambdah^{k,k}_r}\mathrm{d}\Zh^{\omb,m,i}_{r},\dots,\int_0^\cdot \sqrt{\Lambdah^{k,k}_r}\mathrm{d}\Zh^{\omb,m,i}_{r} \bigg).
	\]
	This implies that that $ \Lim_{j\to\infty} \Pb^{\star,m_j} = \Pb^{\star, \infty}$, with
	\begin{align} \label{eq:last-convergence}
		\Pb^{\star, \infty}
		:=
		\int_{\Omb}\Lc^{\Ph^\star_{\omb}}\bigg( \Xh^{\omb,k,\circ}, \Yh^{\omb,k,\circ}, \int_0^\cdot \Lambdah^{k,1}_r\mathrm{d}r,\dots,\int_0^\cdot \Lambdah^{k,k}_r\mathrm{d}r, \int_0^\cdot \sqrt{\Lambdah^{k,k}_r}\mathrm{d}\Zh^{\omb,m,i}_{r},\dots,\int_0^\cdot \sqrt{\Lambdah^{k,k}_r}\mathrm{d}\Zh^{\omb,m,i}_{r}, B(\omb), \widehat{\beta}(\omb) \bigg) \Pb(\mathrm{d}\omb).
	\end{align}
	As the above holds true for any subsequence $(\Pb^{\star,m_j})_{j \ge 1}$,
	one obtains that $ \Lim_{m \to\infty} \Pb^{\star,m} = \Pb^{\star, \infty}$.
	
	\medskip
	
	To conclude, it is enough to use the same arguments as in the constant $\sigma_0$ case,
	together with \eqref{eq:fund-property_I}, 
	to define a sequence of weak control rules $(\overline{\Q}^{\star,k,m})_{(k,m) \in \N^\star \x \N^\star}$ by
	\begin{align*}
		\overline{\Q}^{\star,k,m}
		:=
		\int_{\Omb}\Lc^{\Ph_{\omb}} \Big(\Xh^{\omb,k,m},\Yh^{\omb,k,m},\Lambdah^{k,m}, \Wh^{\omb,k,m},B(\omb),\Lc^{\Ph_{\omb}} \Big(\Xh^{\omb,k,m},\Yh^{\omb,k,m},\Lambdah^{k,m}, \Wh^{\omb,k,m} \Big) \Big)\Pb(\mathrm{d}\omb),
	\end{align*}
	with $\Wh^{\omb,k,m}:=\sum_{i=1}^k\Wh^{\omb,m,i},$ and $\Lambdah^{k,m}(\mathrm{d}a,\mathrm{d}t):=\sum_{i=1}^k\delta_{a^k_i}\mathbf{1}_{I^{k,i}_m}(t)(\mathrm{d}a)\mathrm{d}t$.
	In particular, \eqref{eq:last-convergence} implies the convergence
	\[
		\Lim_{k\to \infty} \Lim_{m\to\infty} \Wc_p\big(\overline{\Q}^{\star,k,m},\Pb\big)=0.
	\]
	\qed
    
\end{appendix}

\bibliography{bibliographyDylan}
\end{document}